\documentclass[11pt]{amsart} 
\usepackage{amscd,amssymb,verbatim,rotating} 
\newtheorem{thm}{Theorem}[section]
\newtheorem{cor}[thm]{Corollary} 
\newtheorem{lem}[thm]{Lemma}
\newtheorem{prop}[thm]{Proposition}

\newcommand{\ep}{\varepsilon} 
\newcommand{\al}{\alpha}

\newcommand{\vp}{\varphi}

\newcommand{\bs}{\backslash}

\newcommand{\ol}{\overline}
\newcommand{\ti}{\widetilde}

\newcommand{\diam}{\operatorname{diam}}
 
\newcommand{\Lip}{\operatorname{Lip}}
\newcommand{\lip}{\operatorname{lip}} 

\newcommand{\loc}{{\operatorname{loc}}}
\newcommand{\Ann}{\operatorname{Ann}} 

\newcommand{\N}{{\mathbb N}} 
\newcommand{\R}{{\mathbb R}}

\newcommand{\cF}{{\mathcal F}}
 
\newcommand{\cA}{{\mathcal A}} 

\newcommand{\cC}{{\mathcal C}}

\begin{document}

\title{Nonlinear order isomorphisms on function spaces}

\begin{abstract}
Let $X$ be a topological space.  A subset of $C(X)$, the space of continuous real-valued functions on $X$, is a partially ordered set in the pointwise order.  Suppose that $X$ and $Y$ are topological spaces, and $A(X)$ and $A(Y)$ are subsets of $C(X)$ and $C(Y)$ respectively.  We consider the general problem of characterizing the order isomorphisms (order preserving bijections) between $A(X)$ and $A(Y)$.  
Under some general assumptions on $A(X)$ and $A(Y)$, and when $X$ and $Y$ are compact Hausdorff, it is shown that existence of an order isomorphism between $A(X)$ and $A(Y)$ gives rise to an associated homeomorphism between $X$ and $Y$.  This generalizes a classical result of Kaplansky concerning linear order isomorphisms between $C(X)$ and $C(Y)$ for compact Hausdorff $X$ and $Y$.  The class of near vector lattices is introduced in order to extend the result further to noncompact spaces $X$ and $Y$.
The main applications lie in the case when $X$ and $Y$ are metric spaces. Looking at  spaces of uniformly continuous functions, Lipschitz functions, little Lipschitz functions, spaces of differentiable functions, and the bounded, ``local" and ``bounded local" versions of these spaces, characterizations of when spaces of one type can be order isomorphic to spaces of another type are obtained.
\end{abstract}

\author{Denny H.\ Leung}
\address{Department of Mathematics, National University of Singapore, Singapore 119076}
\email{matlhh@nus.edu.sg}

\author{Wee-Kee Tang}
\address{Division of Mathematical Sciences,
 Nanyang Technological University, Singapore 637371 }
\email{WeeKeeTang@ntu.edu.sg }

\thanks{Research of the first author is partially supported by AcRF project no.\ R-146-000-157-112}

\maketitle

\tableofcontents

There is a long tradition in mathematics of studying a mathematical object by looking at maps from the object into a simpler object of the same type.  For instance, the dual group of a topological group is the set of characters, i.e., the group homomorphisms into the circle group; the dual space of a topological vector space is the set of continuous linear functionals.  In the case of a topological space $X$, it is natural to consider the space $C(X)$ of continuous real-valued functions on $X$.  The space $C(X)$ carries with it a multitude of mathematical structures.  It is an algebra (or ring) of functions, a vector lattice, and its subspace $C_b(X)$ consisting of the bounded functions in $C(X)$ is a Banach space.  All these aspects of $C(X)$ have been used classically  to characterize the space $X$ for compact Hausdorff $X$, with various generalizations to noncompact spaces.
The Banach-Stone Theorem \cite{B, St} shows that the isometric structure of $C(X)$ determines a compact Hausdorff space $X$ up to homeomorphism.  Subsequently, the validity of the Banach-Stone Theorem for the Banach space valued space of continuous function $C(X,E)$ has been used to study the geometry of the Banach space $E$. See, e.g., \cite{Be}.  The study of isometries on general Banach spaces is also well developed. We refer the reader to the two volume monograph by Fleming and Jamison \cite{FJ}.
Gelfand and Kolmogorov \cite{GK} proved that a compact Hausdorff space $X$ is determined up to homeomorphism by $C(X)$ as an algebra. Developments in this direction up to the 1970's is beautifully summarized in the classic monograph of Gillman and Jerison \cite{GJ}.  It includes, in particular, a generalization of the theorem of Gelfand and Kolmogorov to noncompact spaces by Hewitt \cite{H}, who identified the class of realcompact spaces and showed the important role played by them in this context.  Later advances saw versions of the theorem of Gelfand and Kolmogorov  on other function algebras as well as generalizations to biseparating maps, i.e., bijective maps that preserve disjointness of functions in both directions.
For instance, see \cite{A2-1,  ABN, AD, GJW, J, JW, L}.
As a vector lattice, Kaplansky \cite{K} showed that $C(X)$ determines a compact Hausdorff space $X$ up to homeomorphism.  In a sense, this result is the most general of the three classical results, as it is not hard to see that both the Banach-Stone  and Gelfand-Kolmogorov Theorems can be derived from Kaplansky's Theorem; see e.g., \cite{LL}. While a general function space may no longer be a vector lattice, it always retains the partial order determined pointwise.  Let $A(X)$ and $A(Y)$ be sets of real-valued functions defined on sets $X$ and $Y$ respectively.  Say that a bijectively map $T: A(X) \to A(Y)$ is an {\em order isomorphism} if $f\leq g$ if and only if $Tf \leq Tg$ for all $f,g \in A(X)$.
One may ask if Kaplansky's Theorem can be extended to order isomorphisms on general function spaces.
The recent paper of the first author and Lei Li mentioned above shows that Kaplansky's result is valid for linear order isomorphisms for rather general classes of function spaces.  See also the work of Jim\'{e}nez-Vargas and  Villegas-Vallecillos concerning linear order isomorphisms on spaces of Lipschitz functions \cite{JV2008}.
Within the last decade, a number of papers by F. and J. Cabello Sanchez have appeared that characterize nonlinear order isomorphisms on various function spaces \cite{C, CC1,CC2,CC3}.  In this paper, our aim is to present a unified and thorough study of nonlinear order isomorphisms between function spaces.  Instead of adapting our arguments on a case by case basis, we  present a general framework for the analysis of order isomorphisms that applies to a  class of spaces which we call near vector lattices.  The class includes spaces of continuous, uniformly continuous,   Lipschitz, little Lipschitz and differentiable functions and their "local" versions. In particular, they include as special cases all the main results of \cite{C, CC1, CC2, CC3}.
It is shown that an order isomorphism between any two such spaces must be a nonlinear weighted composition operator.  Modulo the composition map, such operators are called superposition operators.  We refer to the monograph \cite{AZ} for an in depth treatment of nonlinear superposition operators. Some of our methods may have applications in this area.
We go on to analyze extensively comparisons of function spaces under order isomorphisms.  Along the way, new properties of metric spaces manifest themselves.  Specifically, the connections between order isomorphisms having to do with Lipschitz or little Lipschitz spaces and the properties ``expansive", ``expansive at $\infty$" and ``almost expansive at $\infty$" seem to us to be rather intricate.

We now briefly describe the contents of the individual sections. The first section sets up a general framework for dealing with order isomorphisms. The main result is Theorem \ref{thm8}, which shows that an order isomorphism between sets of articulated, compatible and directed sets of functions on compact Hausdorff spaces gives rise to an associated homeomorphism between the spaces. We might add that we view the properties of being articulated and directed as part of the basic infrastructure that are necessary in the theory.  On the other hand, compatibility is a linking property between functions that is only required when studying nonlinear order isomorphisms (as opposed to linear order isomorphisms).

The second section introduces the class of near vector lattices.  The utility of this class lies in the fact that it is general enough to include many of the function spaces that are of interest, including all vector sublattices of $C(X)$ as well as spaces of differentiable functions.  On the other hand, a satisfactory general theory of order isomorphisms holds for this class of spaces.  We use the well established method of compactification to transcend the restriction of compactness of the underlying spaces in Theorem \ref{thm8}.  
It is at this point that efficacy of having near vector lattices shows, because, by Proposition \ref{prop18.5}, 
the space of continuous extensions (onto a suitable compactification) of the bounded functions in a near vector lattice remains a near vector lattice.  A second problem to overcome is that one has to make use of a whole host of  compactifications, since each one is effective only for an order bounded subset of the function space.
Lemma \ref{lem19.1} solves this problem.  It clears the final hurdle for Theorem \ref{thm18.1}, which shows that an order isomorphism between two near vector lattices gives rise to an associated homeomorphism between some compactifcations of the underlying topological spaces.

Section 3 gives the first application of Theorem \ref{thm18.1}.  Theorem \ref{thm19} shows that if $X$ and $Y$ are realcompact spaces such that $C(X)$ and $C(Y)$ are order isomorphic, then the associated homeomorphism 
obtained in Theorem \ref{thm18.1} restricts to a homeomorphism between $X$ and $Y$. 
It generalizes the previously known result of F.\ Cabello Sanchez for the case of compact $X$ and $Y$ \cite{C}.

From Section 4 onwards, we concentrate on metric spaces $X$ and $Y$.
First, a condition ($\spadesuit$) is identified so that for spaces having such a property, the associated homeomorphism 
obtained in Theorem \ref{thm18.1} restricts to a homeomorphism between $X$ and $Y$ (Theorem \ref{thm27}). 
Then it is shown that under an additional condition ($\heartsuit$), an order isomorphism must be a weighted composition operator (Theorem \ref{thm29}).  Examples B and C show the wide application of the conditions ($\spadesuit$) and ($\heartsuit$).

Section 5 contains the first part of the analysis of spaces of Lipschitz functions.  It begins with an observation that any complete metric space $X$ can be endowed with a complete bounded metric $d'$ so that $\Lip(X)$ is linearly order isomorphic to $\Lip(X')$ (Theorem \ref{thm5.3}).  Since order isomorphisms between spaces of Lipschitz functions on bounded metric spaces have been characterized \cite{CC1}, it leads to an immediate generalization to  general metric spaces (Theorem \ref{thm5.6}).  The analogous result for $\lip(X)$ is quite a bit more delicate. 
Here a property of a metric space which we call ``almost expansive at $\infty$" arises.  
Proposition \ref{prop5.12} shows that if $X$ is a complete metric space that is almost expansive at $\infty$, then $\lip(X)$ is order isomorphic to $\lip(X,d')$, where $d'$ is the same complete bounded metric mentioned above.
This result is only the first installment in a long intricate story.  Proving the converse, that is, determining when $\lip_\al(X)$ can be order isomorphic to $\lip_\al(Y)$ for a bounded metric space $Y$, has to wait for some of the machinery to be built in Section 6, and is completed in Theorem \ref{thm66}.
Further on, the issue of characterizing when two spaces of the type $\lip_\al(X)$ are order isomorphic finds it resolution  in Section 7 (Theorem \ref{lip9}).

In the long Section 6, we undertake an extensive analysis of comparing various spaces of functions under order isomorphism.  In the course of this analysis, some new properties of metric spaces naturally arise. In addition to the condition  of ``almost expansive at $\infty$" (definition following Theorem \ref{thm5.6}) already mentioned above,  we point out the notions of ``proximally compact" (definition following Proposition \ref{prop37}), 
 ``expansive" (definition following Lemma \ref{lem58.1}), and ``expansive at $\infty$" (definition following Lemma \ref{lem63}). Another result worthy of interest is that in our general set up, every order isomorphism is continuous with respect to the topology of uniform convergence on compact sets (Corollary \ref{cor35.1}).

The final section, Section 8, concludes with some results characterizing order isomorphisms between spaces of the same type.  At the end of the paper, we append a table summarizing the comparison results obtained. It is worth pointing out that the remaining open cases all concern spaces of differentiable functions.  To take the space $C^p(X)$ as an example, we have the following open problems.
Suppose that 
\begin{enumerate}
\item Suppose that $T:C^p(X)\to C^p(Y)$ is an order isomorphism.  Must the associated homeomorphism be differentiable on $X$? $C^p$ on $X$? To illustrate the extent of our ignorance, the answers are unknown even for $p =1$ and $X= Y = \R$.
\item Is it possible to have $p\neq q$ and some open sets $X$ and $Y$ in Banach spaces so that $C^p(X)$ is order isomorphic to $C^p(Y)$?  (Here, we ask that $C^p(X)$ contains a bump function.)
\end{enumerate}

\section{General framework}

Let $X$ be a Hausdorff topological space and let $A(X)$ be a subset of $C(X)$, the space of continuous  real-valued functions on $X$.   If $f, g \in C(X)$, let $\{f < g\} = \{g > f\} = \{x\in X: f(x) < g(x)\}$.
We say that $A(X)$ is {\em articulated} if the following conditions hold:

\begin{enumerate}
\item [(A1)] for any $x\in X$, there exist $f \leq g$ in $A(X)$ such that $x\in \{f < g\}$;
\item [(A2)] if $f\leq  g$ are functions in $A(X)$, and $U$ is an open set in $X$ containing a point $x\in \{f < g\}$, then there exists $u \in A(X)$, $f \leq u$, such that $x \in \{f < u\}\subseteq U$; similarly, there exists $v\in A(X)$, $v \leq g$, such that $x \in \{v < g\} \subseteq U$;
\item [(A3)] if $f,g, h\in A(X)$, $h \leq f,g$, and there exists $x\in X$ with $h(x) < f(x), g(x)$, then there is a function $u \in A(X)$, $h \leq u \leq f,g$, such that $h(x) < u(x)$; a similar statement holds if the  ``$\leq$" and ``$<$" signs are replaced by ``$\geq$" and ``$>$" respectively.
\end{enumerate}
\noindent{\bf Remark}. Suppose that $x\in \{f < g\}\cap U$, where $f\leq g$ are functions in $A(X)$ and $U$ is an open set in $X$. By assumption (A2), there exists $u\in A(X)$ such that $f\leq u$ and $x \in \{f < u\} \subseteq U$.  Choose $\ep > 0$ such that $u(x)> f(x) + \ep$.  By assumption (A2) again, there exists $w \in A(X)$, $w \geq f$, such that $x\in \{f < w\} \subseteq \{f + \ep < u\}$.  Thus $x\in \{f < w\} \subseteq \ol{\{f < w\}} \subseteq U$.

\bigskip

Let $A(X)$ and $A(Y)$ be an articulated subsets of $C(X)$ and $C(Y)$ respectively, where $X$ and $Y$ are Hausdorff topological spaces.  For the remainder of the section, we consider a fixed order isomorphism $T: A(X) \to A(Y)$.

\begin{prop}\label{prop1}
If $h \leq f,g$ are functions in $A(X)$ such that $\{h < f\} \cap \{h < g\} = \emptyset$, then $\{Th < Tf\} \cap \{Th < Tg\} = \emptyset$. 
\end{prop}

\begin{proof}
Suppose that $y \in \{Th < Tf\}\cap \{Th < Tg\}$.  By assumption (A3), there exists $u\in A(Y)$, $Th \leq u\leq Tf, Tg$, such that $u(y) > Th(y)$.  Thus $h \leq T^{-1}u \leq f,g$ and hence $T^{-1}u = h$.  But then $u = Th$, contrary to the choice of $u$.
\end{proof}

\begin{prop}\label{prop2}
If $h \leq f,g$ are functions in $A(X)$ such that $\ol{\{h < f\}} \subseteq \ol{\{h < g\}}$, then $\ol{\{Th < Tf\}} \subseteq \ol{\{Th < Tg\}}$.
\end{prop}

\begin{proof}
It suffices to show that $\{Th < Tf\} \subseteq \ol{\{Th < Tg\}}$. Suppose that there exists $y \in \{Th < Tf\} \bs \ol{\{Th < Tg\}}$. 
By assumption (A2), there exists $u \in A(Y)$ such that $Th \leq u $, $y \in \{Th < u\}$ and $\{Th< u\} \cap \ol{\{Th < Tg\}} = \emptyset$.
By assumption (A3), there exists $v\in A(Y)$ such that $Th \leq v \leq Tf, u$, and that $v(y)> Th(y)$.
In particular, $h \leq T^{-1}v \leq f$.  Thus ${\{h < T^{-1}v\}} \subseteq \ol{\{h < f\}} \subseteq \ol{\{h < g\}}$.
On the other hand, since $\{Th < v\} \cap \{Th < Tg\}\subseteq \{Th< u\} \cap {\{Th < Tg\}} = \emptyset$, it follows by applying Proposition \ref{prop1} to $T^{-1}$ that $\{h < T^{-1}v\}\cap \{h < g\} = \emptyset$.
Since $\{h < T^{-1}v\}$ is open, we must have $\{h < T^{-1}v\}\cap\ol{\{h < g\}} = \emptyset$.
Thus we conclude that $\{h < T^{-1}v\} = \emptyset$; equivalently, $h = T^{-1}v$.  But then $Th = v$, contradicting the choice of $v$.
\end{proof}

For each $f\in A(X)$, let $\cC^+_f(X)$ be the collection of sets $\{f < g\}$, where $f \leq g\in A(X)$.
Similarly, let $\cC^-_f(X)$ consist of the sets $\{g < f\}$, where $f \geq g \in A(X)$.  Set $\ol{\cC}^\pm_f(X) =\{\ol{U}: U\in \cC^\pm_f(X)\}$.
By Proposition \ref{prop2}, the map $\theta^+_f: \ol{\cC}^+_f(X) \to \ol{\cC}^+_{Tf}(Y)$ given by $\theta^+_f(\ol{\{f<g\}}) = \ol{\{Tf< Tg\}}$ is well defined. By the same proposition and its counterpart for $T^{-1}$, $\theta^+_f$ is  a bijection and preserves the natural order of set inclusion in  $\ol{\cC}^+_f(X)$ and $\ol{\cC}^+_{Tf}(Y)$.  
Similarly, the map $\theta^-_f: \ol{\cC}^-_f(X)\to \ol{\cC}^-_{Tf}(Y)$ defined by $\theta^-_f(\ol{\{g<f\}}) = \ol{\{Tg< Tf\}}$ for $g\leq f$ in $A(X)$ is an order preserving bijection.

\begin{prop}\label{prop3}
If $A\in \ol{\cC}^+_f(X)$ and $f\leq g\in A(X)$, then $f = g$ on $A$ if and only if $Tf = Tg$ on $\theta^+_{f}(A)$.  Similarly, if $A\in \ol{\cC}^-_f(X)$ and $g\leq f\in A(X)$, then $f = g$ on $A$ if and only if $Tf = Tg$ on $\theta^-_{f}(A)$.
\end{prop}

\begin{proof}
Suppose that $f\leq g,h$ in $A(X)$, $A = \ol{\{f < h\}} \in \ol{\cC}^+_f(X)$,  and that $f = g$ on $A$.
Then $\{f < h\} \cap \{f < g\} = \emptyset$. By Proposition \ref{prop1}, $\{Tf < Th\} \cap \{Tf < Tg\} = \emptyset$.
Hence $\ol{\{Tf < Th\}} \cap {\{Tf < Tg\}} = \emptyset$.  Thus $Tf \geq Tg$ on $\theta^+_f(A)$.  Since $Tf \leq Tg$, $Tf = Tg$ on $\theta^+_f(A)$. The converse follows by symmetry.  The second statement is entirely analogous.
\end{proof}

\begin{lem}\label{lem4}
Suppose that $U\cap \theta^+_f(A) \neq \emptyset$, where $f\in A(X)$, $A\in \ol{\cC}^+_f(X)$ and $U\in \cC^+_{Tf}(Y)$. Then there exists a nonempty set $V\in \cC^+_f(X)$ such that $V\subseteq  (\theta^+_f)^{-1}(\ol{U})\cap A$.
\end{lem}

\begin{proof}
There exist $g, h\in A(X)$, $f \leq g,h$, such that $U =\{Tf < Th\}$ and $\theta^+_f(A) = \ol{\{Tf < Tg\}}$.
Since $U$ is open, it follows from the assumption that $\{Tf < Th\} \cap \{Tf < Tg\}  \neq \emptyset$.
By Proposition \ref{prop1}, $\{f < h\}\cap \{f < g\} \neq \emptyset$. 
By assumption (A3), there exists $u\in A(X)$, such that $f\leq u \leq g,h$ and $f\neq u$.  Note that $(\theta^+_f)^{-1}(\ol{U})\cap A = \ol{\{f < h\}}\cap \ol{\{f< g\}}$. Thus the set $V = \{f < u\}$ satisfies the conditions of the lemma.
\end{proof}

We say that $A(X)$ is a {\em compatible set of functions} if for any pair $f \leq g$ in $A(X)$,  any $x\in \{f< g\}$ and any closed set $A$ in $X$ such that $x\notin A$, there exist open neighborhoods $U$ and $V$ of $x$ and $u, v \in A(X)$ so that $f\leq u,v\leq g$, $\ol{U} \cap A = \emptyset = \ol{V}\cap A$, and 
\[ u = \begin{cases}
          f & \text{on $U$}\\
          g & \text{on $A$}
          \end{cases},
          \quad
    v = \begin{cases}
          g &\text{on $V$}\\
          f &\text{on $A$}
          \end{cases}.\]
Compatibility of functions allows us to connect different mappings of the form $\theta^\pm_f$.

\begin{prop}\label{prop4}
Let $A(X)$ and $A(Y)$ be articulated, compatible sets of functions. Suppose that $f\leq g$ are functions in $A(X)$.  Let $A\in \ol{\cC}^+_f(X)$ and $B \in \ol{\cC}^-_g(X)$. If $A \subseteq B\cap\{f < g\}$,  then  $\theta^+_f(A) \cap \{Tf < Tg\} \subseteq  \theta^-_g(B)$.
Similarly, if $B \subseteq A \cap \{f< g\}$, then $\theta^-_g(B)\cap \{Tf < Tg\} \subseteq \theta^+_f(A)$. 
\end{prop}

\begin{proof}
We will prove the first assertion; the second one can be obtained similarly.  Assume to the contrary that there exists $y\in \theta^+_f(A)\cap\{Tf < Tg\}$ and $y \notin \theta^-_g(B)$.  Since $A(Y)$ is compatible, there exists $u\in A(Y)$, $Tf\leq u \leq Tg$, and an open neighborhood $U$ of $y$ such that $\ol{U} \cap \theta^-_g(B) = \emptyset$, and that
$u = Tf$ on $U$ and $u = Tg$ on $\theta^-_g(B)$.  By property (A2), we may assume that $U \in \cC^+_{Tf}(Y)$. Let $D = (\theta^+_f)^{-1}(\ol{U})$. Then $T^{-1}u = f$ on $D$ and $T^{-1}u = g$ on $B$ by Proposition \ref{prop3}.  In particular, $f = g$ on $A\cap D$.  By Lemma \ref{lem4}, $A\cap D \neq \emptyset$, yielding a contradiction to the assumption that $f < g$ on $A$.
\end{proof}

Say that $A(X)$ is {\em directed} if for any $f_1,f_2\in A(X)$, there exist $h_1, h_2\in A(X)$ such that $h_1\leq f_1, f_2 \leq h_2$.   For each $x\in X$, let $\cF_x$ be the collection of all sets of the form $\ol{\{Tf < Tg\}}$, where $f\leq g$ are functions in $A(X)$ such that $x\in \{f < g\}$.  Define $\cF_y$ in the same manner for all $y\in Y$, using the map $T^{-1}$ in place of $T$.

\begin{lem}\label{lem6}
Let $A(X)$ and $A(Y)$ be articulated, compatible and directed sets of functions. Suppose that $y \in \cap\cF_x$ for some $x\in X$.  If $y\in U$ for some $U\in \cC^\pm_{Tf}(Y)$, then $x\in (\theta^\pm_f)^{-1}(\ol{U})$.
\end{lem}

\begin{proof}
Suppose that the lemma fails.  Consider the ``$+$" case. We find  $f\leq g$ in $A(X)$ such that $y\in \{Tf < Tg\}$ and $x \notin \ol{\{f < g\}}$.
By assumption (A1) and the fact that $A(X)$ is directed, there are functions $h_1, h_2 \in A(X)$ such that $h_1 \leq f,g\leq h_2$ and that $h_1(x) < h_2 (x)$.  We may assume that $h_1(x) < f(x) = g(x)$; otherwise $f(x) = g(x) < h_2(x)$ and the proof is similar.  
By assumption (A2) and the remark thereafter, there exists $v\in A(X)$, $v \leq g$, such that $x\in \{v< g\}$ and $\ol{\{v< g\}} \subseteq  \{h_1 < g\} \cap (\ol{\{f < g\}})^c$.
Applying the same assumption, we find $w\in A(X)$, $h_1 \leq w$, such that $x\in \{h_1< w\}$ and $\ol{\{h_1< w\}} \subseteq  \{v < g\} \cap \{h_1 < f\}$.  Use (A2) a third time to obtain $u\in A(X)$, $u \leq f$, 
such that $x\in \{u < f\} \subseteq \{h_1 < w\}$.
Applying (A3) as well, we may further assume that $h_1 \leq u$. Observe that
\begin{align*}
\ol{\{u < f\}} \subseteq \ol{\{h_1 < w\}} &\subseteq \{h_1 < f\}, \\
&\text{ where } \ol{\{u < f\}} \in \ol{\cC}^-_f(X) \text{ and } \ol{\{h_1< w\}} \in \ol{\cC}^+_{h_1}(X);\\
\ol{\{h_1 < w\}} \subseteq \ol{\{v < g\}} &\subseteq \{h_1 < g\},\\ 
&\text{ where } \ol{\{h_1 < w\}} \in \ol{\cC}^+_{h_1}(X) \text{ and } \ol{\{v< g\}} \in \ol{\cC}^-_{g}(X).
\end{align*}
We have 
\begin{align*}
\{Tu < Tf\} & \subseteq \theta^-_f(\ol{\{u < f\}}) \cap \{Th_1 < Tf\}\\
& \subseteq  \theta^+_{h_1}(\ol{\{h_1 < w\}}) \cap \{Th_1 < Tg\}\\
& \subseteq  \theta^-_g(\ol{\{v < g\}}),
\end{align*}
where, for the first inclusion, we apply the second half of Proposition \ref{prop4} with  $A = \ol{\{h_1<w\}}$, $B = \ol{\{u< f\}}$; for the third inclusion, we apply the first half of same proposition with $A = \ol{\{h_1< w\}}$, $B= \ol{\{v<g\}}$.
Hence $\ol{\{Tu < Tf\}} \subseteq \theta^-_g(\ol{\{v < g\}})$.
Since  $f \leq g$ and $f = g$ on $\ol{\{v< g\}} \in \ol{\cC}^-_g(X)$, $Tf = Tg$ on $\theta^-_g(\ol{\{v<g\}})$ by Proposition \ref{prop3}.  However, $u \leq f$, $x\in \{u < f\}$ and $y \in \cap\cF_x$ imply that $y \in \ol{\{Tu < Tf\}}$.
Thus $y \in \theta^-_g(\ol{\{v<g\}})$. Hence $Tf(y) = Tg(y)$, contrary to the choice that $y \in \{Tf < Tg\}$.
\end{proof}

\begin{prop}\label{prop6.1}
Let $A(X)$ and $A(Y)$ be articulated, compatible and directed sets of functions. If $x\in X$ and $y \in \cap\cF_x$, then $x\in\cap \cF_y$.
\end{prop}

\begin{proof}
If $A\in \cF_y$, then $A = \ol{\{f<g\}}$, where $Tf \leq Tg$ are functions in $A(Y)$ such that $y\in \{Tf< Tg\}$.  In other words,  $A = (\theta^{+}_f)^{-1}(\ol{U})$, where $U = \{Tf< Tg\}$ is a set in $\cC^+_{Tf}(Y)$ that contains $y$.  By Lemma \ref{lem6}, $x\in A$.  This proves that $x\in \cap \cF_y$.
\end{proof}

\begin{prop}\label{prop7}
Let $A(X)$ and $A(Y)$ be articulated, compatible and directed sets of functions. 
For each $x\in X$, $\cF_x$ has the finite intersection property.  Moreover, the intersection $
\cap\cF_x$ consists of at most one point.
\end{prop}

\begin{proof}
Suppose that $f_i \leq g_i$, $1\leq i\leq k$, are functions in $X$ such that $x\in \cap^k_{i=1}\{f_i < g_i\}$.
Since $A(X)$ is directed, there exists $g\in A(X)$ such that $g_i \leq g$, $1\leq i \leq k$.
By assumption (A2) and the remark thereafter, there exist $u_i, w_i\in A(X)$ with $u_i \leq g$ and  $f_i \leq w_i$, such that 
\[x\in \{u_i < g\}\subseteq \{f_i < w_i\} \subseteq  \ol{\{f_i < w_i\}}\subseteq  \{f_i< g_i\},\ 1 \leq i \leq k.\]
Applying assumption (A3), we find $v$ in $A(X)$ such that $f_i, u_i \leq v\leq g$, $1\leq i \leq k$, and $x\in \{v< g\}$.
Since 
\[ \ol{\{v< g\}} \subseteq \ol{\{u_i < g\}} \subseteq \ol{\{f_i < w_i\}}\subseteq \{f_i < g\},\]
and $\ol{\{v< g\}} \in \ol{\cC}^-_g(X)$, $\ol{\{f_i < w_i\}}\in \ol{\cC}^+_{f_i}(X)$, Propositions \ref{prop4} and \ref{prop2} imply that 
\[\ol{\{Tv< Tg\}} \cap \{Tf_i < Tg\}\subseteq \ol{\{Tf_i < Tw_i\}}\subseteq \ol{\{Tf_i < Tg_i\}}.\]
As $f_i\leq v \leq g$, $\{Tv< Tg\} \subseteq \{Tf_i < Tg\}$, $1\leq i \leq k$. Therefore, 
\[ \{Tv < Tg\} \subseteq \ol{\{Tv< Tg\}} \cap \{Tf_i < Tg\}\subseteq \ol{\{Tf_i < Tg_i\}}, 1\leq i \leq k.\]
Because $v \leq g$, $v\neq g$ and $T$ is an order isomorphism, the set on the left is nonempty.
This proves that $\cF_x$ has the finite intersection property.

Suppose that there are distinct points $y_1$ and $y_2$ in $\cap \cF_x$.
By assumption (A1) and the fact that $A(X)$ is directed, there are $h_1 \leq h_2$ in $A(X)$ so that $h_1(x) < h_2(x)$ and $Th_1(y_i) < Th_2(y_i)$, $i=1,2$. 
By assumption (A2), there exists $V \in \cC^-_{Th_2}(Y)$ such that $y_1 \notin \ol{V}$ and $y_2 \in V$. 
By the compatibility of functions in $A(Y)$, there exist $u \in A(Y)$
and an open neighborhood $U$ of $y_1$
such that $Th_1 \leq u\leq Th_2$, $u = Th_1$ on $\ol{U}$ and $u = Th_2$ on $\ol{V}$.
By property (A2), we may assume that $U \in \cC^+_{Th_1}(Y)$.
Note that  $u = Th_1$ on $\ol{U}\in \ol{\cC}^+_{Th_1}(Y)$, $Th_1\leq u$, and thus $T^{-1}u = h_1$ on $(\theta^+_{h_1})^{-1}(\ol{U})$.
Similarly, $T^{-1}u = h_2$ on $(\theta^-_{h_2})^{-1}(\ol{V})$.
Since $h_1(x) \neq h_2(x)$ and $x\in (\theta^+_{h_1})^{-1}(\ol{U}) \cap (\theta^-_{h_2})^{-1}(\ol{V})$ by Lemma \ref{lem6}, we have a contradiction.
\end{proof}

Proposition \ref{prop7} suggests that the set mappings  $\theta^\pm_f$, $f\in A(X)$, may be induced by a point mapping $\vp$  between ``large'' subsets of $X$ and $Y$ in the sense  that $x\in U\in \cC^\pm_f(X)$ implies $\vp(x) \in \theta^\pm_f(\ol{U})$. The key, evidently, is to guarantee that $\cap\cF_x \neq \emptyset$ for any $x\in X$.
This is obvious if compactness is present.

\begin{thm}\label{thm8}
Let $X$ and $Y$ be compact Hausdorff spaces and let $A(X)$ and $A(Y)$ be articulated, compatible and directed sets of continuous real-valued functions on $X$ and $Y$ respectively.  If $T:A(X)\to A(Y)$ is an order isomorphism, there is a homeomorphism $\vp :X\to Y$ such that for any $f\leq g$ in $A(X)$ and any open set $W$ of $X$, $f = g$ on $W$ if and only if $Tf = Tg$ on $\vp(W)$.
\end{thm}

\begin{proof}
Since $Y$ is compact, by Proposition \ref{prop7}, for each $x\in X$, $\cap\cF_x$ contains exactly one point in $Y$.
Define $\vp: X\to Y$ by letting $\vp(x)$ be the unique point in $\cap\cF_x$.
Similarly, for each $y \in Y$, we  obtain a function $\beta:Y\to X$ by letting $\beta(y)$ be the unique point in $\cap\cF_y$. By Proposition \ref{prop6.1}, we see that $\vp$ and $\beta$ are mutual inverses.

Suppose that $\vp$ is not continuous at $x_0 \in X$.  Since $Y$ is compact Hausdorff, there is a net $(x_\gamma)$ in $X$ converging to $x_0$ such that $(\vp(x_\gamma))$ converges to some $z \neq \vp(x_0)$.
By assumption (A1) and the directedness of $A(X)$, there are functions $f \leq g$ in $A(X)$ such that $f(x_0)< g(x_0)$, $Tf(z)< Tg(z)$ and $Tf(\vp(x_0)) < Tg(\vp(x_0))$.
By assumption (A2), there exists $V\in \cC^-_{Tg}(Y)$ containing $\vp(x_0)$ such that $z \notin \ol{V}$.
Since $A(Y)$ is a compatible set of functions, there exist an open neighborhood $U$ of $z$ 
and $u \in A(Y)$ such that
$Tf \leq u \leq Tg$, $\ol{U} \cap \ol{V} = \emptyset$, $u = Tf$ on $U$ and $u = Tg$ on $\ol{V}$.
By property (A2), we may assume that $U \in \cC^+_{Tf}(Y)$.
For a  cofinal set of $\gamma$, $\vp(x_\gamma) \in U$.  Hence by Lemma \ref{lem6}, $x_\gamma \in (\theta^+_f)^{-1}(\ol{U})$.  Therefore, $x_0\in (\theta^+_f)^{-1}(\ol{U})$. Similarly, $\vp(x_0) \in V$ implies that $x_0 \in (\theta^-_g)^{-1}(\ol{V})$.
Since $u =Tf$ on $\ol{U} \in \ol{\cC}^+_{Tf}(Y)$, $u = Tg$ on $\ol{V}\in \ol{\cC}^-_{Tg}(Y)$ and $Tf \leq u \leq Tg$, $T^{-1}u = f$ on $(\theta^+_f)^{-1}(\ol{U})$ and $T^{-1}u = g$ on $(\theta^-_g)^{-1}(\ol{V})$ by Proposition \ref{prop3}.  It follows that $f(x_0) = T^{-1}u(x_0) = g(x_0)$, contrary to the choice of $f$ and $g$.
This proves that $\vp$ is continuous on $X$.  By symmetry, $\beta$ is continuous on $Y$.  Thus $\vp$ is a homeomorphism from $X$ onto $Y$.

Finally, suppose that $f\leq g$ in $A(X)$ and that $f= g$  on an open set $W$ in $X$.  Let $x_0 \in W$ and set $U = \{Tf < Tg\} \in \cC^+_{Tf}(Y)$.  If $\vp(x_0)\in U$, then $x_0 \in (\theta^+_f)^{-1}(\ol{U}) = \ol{\{f < g\}}$ by Lemma \ref{lem6}, which is impossible.  Hence $Tf(\vp(x_0)) = Tg(\vp(x_0))$.  The shows that $Tf = Tg$ on the set $\vp(W)$.  The converse follows by symmetry.
\end{proof}

\section{Near vector lattices}

In this section we consider applications of Theorem \ref{thm8}, even to noncompact spaces $X$ and $Y$.
A vector subspace $A(X)$ of $C(X)$ is said to be {\em unital} if it contains the constant function $1$. It {\em separates points from closed sets} if for any $x_0\in X$ and any closed set $A$ in $X$ not containing $x_0$, there exists $f\in A(X)$ such that $f(x_0) \neq 0$ and $f =0$ on $A$.
A unital vector subspace $A(X)$ of $C(X)$ that separates points from closed sets  is a {\em near vector lattice} if for any $f\in A(X)$, there exists $g\in A(X)$ such that $g(x) = f(x)$ if $f(x) \geq 1$, $g(x) = 0$ if $f(x) \leq 0$, and $0 \leq g(x) \leq 1$ if $0 \leq f(x) \leq 1$.
It is easy to see that in this definition, the numbers $0$ and $1$ may be replaced by any pair of real numbers $a,b$ such that $a < b$.

\begin{lem}\label{lem11}
Let $A(X)$ be a near vector lattice. Suppose that $f \in A(X)$ and $a< b< c < d$ are real numbers. Then there exists $g\in A(X)$ such that $a \leq g \leq d$, 
$g(x) = d$ if $f(x) \geq d$, $g(x) = a$ if $f(x) \leq a$ and $g(x) = f(x)$ if $b\leq f(x) \leq c$.
\end{lem}

\begin{proof}
Let $f_1\in A(X)$ be such that $f_1(x) = f(x)$ if $f(x) \geq b$, $f_1(x) =a $ if $f(x) \leq a$ and $a \leq f_1(x) \leq b$ if $a\leq f(x) \leq b$.
Then there exists $f_2\in A(X)$ such that $f_2(x) = -f_1(x)$ if $-f_1(x) \geq -c$, $f_2(x) = -d$ if $-f_1(x) \leq -d$ and $-d \leq f_2(x) \leq -c$ if $-d\leq -f_1(x) \leq -c$.  One may verify directly that $g= -f_2$ satisfies the conclusions of the lemma.
\end{proof}

The next proposition gives the motivation for considering near vector lattices.

\begin{prop}\label{prop13}
Let $A(X)$ be a near vector lattice on a Hausdorff topological space $X$. 
Define 
\[ B(X) = \{f\in A(X): 0 \leq f\leq 1\}.\]
Then $A(X)$ and $B(X)$ are  articulated, compatible and directed subsets of $C(X)$.
\end{prop}

\begin{proof}
Suppose that $f\in A(X)$.  Let $g$ be as given by the definition of a near vector lattice.  Then $g+1\in A(X)$ and $g+1 \geq f,0$.
If $f_1,f_2\in A(X)$, we can thus obtain $g_1,g_2\in A(X)$ such that $g_i \geq f_i, 0$, $i = 1,2$.
Then $g_1+g_2 \geq f_1, f_2$.  Similarly, there exists $h\in A(X)$ such that $h \leq f_1,f_2$.
This shows that $A(X)$ is directed.

Since $A(X)$ is unital, property (A1) is obvious.
Suppose that $x\in \{f< g\}$ for some $f\leq g$ in $A(X)$ and let $U$ be an open set containing $x$.  Since $A(X)$ separates points from closed sets, there exists $h\in A(X)$ such that $h(x) \neq 0$ and $h= 0$ outside $U$. We may assume that $h(x) = c > 0$. By Lemma \ref{lem11}, there exists $w\in A(X)$ such that $0 \leq w \leq c$, $w(z) = 0$ if $h(z) \leq 0$ and $w(z) = c$ if $h(z) \geq c$.  Set $u = f+w$. Then $f \leq u$, $u(x) = f(x) +c > f(x)$ and 
\[\{f < u\} = \{w > 0\} \subseteq \{h > 0\} \subseteq U.\]  
Thus $u$ fulfills the conditions in (A2).  
Let us also observe that $w = u-f$ is a bounded function.
The function $v$ in (A2) may be constructed similarly.
Suppose that $f,g, h\in A(X)$, $h \leq f,g$, and that $h(x) < f(x), g(x)$.
Choose $a > 0$ such that $f(x), g(x) > h(x) +a$.
Set $U = \{f > h+a\}\cap \{g> h+a\}$.  Then $U$ is an open neighborhood of $x$.
By the above, there exists a bounded nonnegative function $w\in A(X)$ such that $w(x) > 0$ and $\{w > 0\} \subseteq U$. By rescaling, we may assume that $0 \leq w \leq 1$.
Set $u = h+ aw$. Then $h \leq u\in A(X)$ and $h(x) < u(x)$.
If $w(z) > 0$, then $z\in U$ and hence $u(z)\leq h(z) + a \leq f(z), g(z)$.
If $w(z) = 0$, then $u(z) = h(z) \leq f(z),g(z)$.  Therefore, $u \leq f,g$.  This verifies one of the directions  of condition (A3).  The other direction may be obtained similarly.

Finally, we show that $A(X)$ is compatible.  Again, we limit ourselves to the construction of the function ``$u$" in the definition. Let $f\leq g$ be functions in $A(X)$, $x\in \{f< g\}$, and let $A$ be a closed set such that $x\notin A$. Set $a= (g-f)(x) > 0$.  By property (A2),  there exists $h_1 \in A(X)$, $f\leq h_1$ such that 
\[x \in\{f < h_1\}\subseteq \{g-f> \frac{a}{2}\}\cap A^c.\]
Set $h_2 = h_1-f$.
We may assume that $h_2(x) = 1$.
Since $A(X)$ is a near vector lattice, there exists
$v\in A(X)$ such that 
\[ v(z)
\begin{cases}
= (g-f-2ah_2)(z) &\text{if $(g-f-2ah_2)(z) \geq 0$}, \\
= -a/3 &\text{if $(g-f-2ah_2)(z) \leq -\frac{a}{3}$},\\
\in [-a/3,0] &\text{if $-\frac{a}{3}\leq (g-f-2ah_2)(z) \leq 0$}.
\end{cases}\]
By Lemma \ref{lem11}, there exists $w\in A(X)$ such that $1/6 \leq w\leq 1/4$ and
\[ w(z) = \begin{cases}
               \frac{1}{4} &\text{if $h_2(z) \geq \frac{1}{4}$},\\
               \frac{1}{6} &\text{if $h_2(z) \leq \frac{1}{6}$}.
               \end{cases}\]
Set $u = f+v+4a(w-1/6) \in A(X)$.
By direct verification, the set 
\[U= \{g-f-2ah_2<-\frac{a}{2}\}\cap\{h_2 > \frac{1}{4}\}\] 
is an open neighborhood of $x$.  Also,
\[ \ol{U} \subseteq \{h_2\geq \frac{1}{4}\} \subseteq \{h_2 > 0\} = \{f < h_1\} \subseteq A^c.\]
If $z\in U$, then $v(z) = -a/3$ and $w(z) = 1/4$.
Thus $u(z) = f(z)$.
If  $h_2(z) = 0$, then $w(z) = 1/6$ and $(g-f-2ah_2)(z) = (g-f)(z) \geq 0$; hence $v(z) = (g-f)(z)$.
It follows that $u(z) = f(z) + (g-f)(z) = g(z)$.
In particular, $u(z) = g(z)$ if $z\in A$.

If $(g-f-2ah_2)(z) <0$, then $h_2(z) > 0$ hence $(g-f)(z) > a/2$.    Thus $a/2-2ah_2(z) < 0$ and hence $h_2(z) > 1/4$.  In this case, $-a/3 \leq v(z) \leq 0$ and $w(z) = 1/4$.
It follows that 
\[ f(z) \leq u(z) = f(z) + v(z) + \frac{a}{3}\leq f(z) + \frac{a}{3} < f(z) + (g-f)(z) = g(z).\]

If $(g-f-2ah_2)(z) \geq 0$, then $v(z) = (g-f-2ah_2)(z) \geq 0$.  In particular, $u(z)\geq f(z)$.  We also have
\[u(z) = g(z)-  2ah_2(z) + 4a(w(z) - \frac{1}{6}).\]
Either $h_2(z) \leq 1/6$, in which case $w(z) = 1/6$ and thus $u(z) \leq g(z)$; or $h_2(z) \geq 1/6$ and hence 
\[u(z) \leq g(z) - \frac{a}{3} + 4a(w(z) - \frac{1}{6}) \leq g(z).\]
This completes the proof that $f \leq u \leq g$.  

Since $B(X)$ contains a largest element $1$ and a smallest element $0$, it satisfies property (A1) and is also directed.  It is easy to see that $B(X)$ inherits property (A3) and compatibility from $A(X)$.
Suppose that $f\leq g$ are functions in $B(X)$ and $U$ is an open neighborhood of a point $x\in \{f<g\}$.
Since $A(X)$ satisfies property (A2), there exists $w \in A(X)$, $f\leq w$, such that $x\in \{f< w\}\subseteq U$.
By property (A3) for $A(X)$, there exists $u\in A(X)$ such that $f\leq u \leq w,g$ and $f(x) < u(x)$.
But then $f\leq u \leq g$ and hence $u \in B(X)$.  This shows that $B(X)$ satisfies property (A2) as well.
\end{proof}

Let $X$ be a Hausdorff topological space and let $A(X)$ be a vector subspace of $C(X)$ that separates points from closed sets.
Let $\R_\infty$ be the set $[-\infty,\infty]$, endowed with the order topology.  The map $i:X \to \R_\infty^{A(X)}$ given by  $i(x)(f) = f(x)$ is a homeomorphic embedding.  We identify $X$ with $i(X)$ and denote the closure of $X$ in $\R_\infty^{A(X)}$ by $\cA X$.  $\cA X$ is a  compact Hausdorff space.  Every function $f\in A(X)$ has a unique continuous $\R_\infty$-valued extension to $\cA X$, given by the formula $\hat{f}(x) = x(f)$ for all $x\in \cA X$.  If $f$ is bounded, then $\hat{f}$ is real-valued (and bounded) on $\cA X$.

\begin{lem}\label{lem12}
Let $A(X)$ be a near vector lattice.  Assume that $f\in A(X)$, $x_0 \in \cA X$ and  $U$ is an open neighborhood of $x_0(f)$ in $\R_\infty$.  Then there exists $g\in A(X)$ such that $g \geq 0$, $\hat{g}(x_0) = 0$ and that 
\[\{x\in \cA X: \hat{g}(x) < 1\} \subseteq \{x\in \cA X: x(f) \in U\}.\]
\end{lem}

\begin{proof}
First consider the case where $a = x_0(f) \in \R$.
Since $\{x\in \cA X: x(f) \in U\} = \{x\in \cA X: x(f-a) \in U'\}$, where $U' = \{t- a: t\in U\}$, we may further assume that $a = 0$.
Choose $\ep > 0$ such that $(-3\ep,3\ep) \subseteq U$.
By Lemma \ref{lem11}, there are functions $g_1, g_2$ in $A(X)$ such that $\ep \leq g_1,g_2 \leq 2\ep$,
\begin{align*}
g_1(x) &=
\begin{cases}
2\ep &\text{if $f(x) \geq 2\ep$}\\
\ep&\text{if $f(x) \leq \ep$},
\end{cases}\\
g_2(x) &=
\begin{cases}
2\ep &\text{if $-f(x) \ geq 2\ep$}\\
\ep&\text{if $-f(x) \leq \ep$},
\end{cases}
\end{align*}
Let $g = \ep^{-1}(g_1+g_2-2\ep)$.
Since $g_1,g_2 \geq \ep$, $g \geq 0$. Moreover, $g(x) = 1$ if $f(x) \notin (-2\ep, 2\ep)$, $g(x) = 0$ if $-\ep \leq f(x) \leq \ep$.
Suppose that $x\in \cA X$ and $x(f) \notin U$.  Then $V = \{z\in \cA X: |z(f)| > 2\ep\}$ is an open neighborhood of $x$ in $\cA X$.  Since $g = 1$ on $V\cap X$, $\hat{g}(x) = 1$  by continuity of $\hat{g}$ and density of $X$ in $\cA X$.   This proves that  $\hat{g}(x) < 1$ implies that $x(f) \in U$.
We also have $g(x) = 0$ if $f(x) \in [-\ep, \ep]$, $x\in X$.  Thus, $g = 0$ on $W \cap X$, where $W= \{z\in \cA X: z(f) \in (-\ep,\ep)\}$ is an open neighborhood of $x_0$. Hence $\hat{g}(x_0) = 0$ by continuity.

Next, consider the case where $x_0(f) = \infty$.  If $x_0(f) = -\infty$, the argument is similar.
Choose $m\in \R$ such that $(m,\infty] \subseteq U$.
By Lemma \ref{lem11}, there exists $g_1\in A(X)$ such that $m+1 \leq g_1 \leq m+2$, $g_1(x) = m+2$ if $f(x) \geq m+2$ and $g_1(x) = m+1$ if $f(x) \leq m+1$.  Set $g = m+2 - g_1$. Then $g\geq 0$.  
Moreover, $\{\hat{f} > m+2\}$ is a neighborhood of $x_0$ in $\cA X$ and $g = 0$ on the set $\{\hat{f} > m+2\} \cap X$.
Hence $\hat{g}(x_0) = 0$.
Similarly, $\hat{g} = 1$ on the set $\{\hat{f} < m+1\}$.  Therefore, $\hat{f} \geq m+1$ on the set $\{\hat{g} < 1\}$.  In particular, $\hat{g}(x) < 1$ implies $x(f) >m$ and thus $x(f) \in U$.
\end{proof}

\begin{prop}\label{prop14}
Let $A(X)$ be a near vector lattice. If $U$ is an open neighborhood of a point $x_0$ in $\cA X$, then there exists a function $f\in A(X)$, $0 \leq f \leq 1$, such that $\hat{f}(x_0) = 1$ and $\hat{f} = 0$ outside $U$.
\end{prop}

\begin{proof}
There exist $f_1,\dots, f_n \in A(X)$ and open neighborhoods $V_1,\dots, V_n$ of $x_0(f_i)$ in $\R_\infty$ such that $\cap^n_{i=1}U_i \subseteq U$, where $U_i =\{x\in \cA X: x(f_i) \in V_i\}$.
By Lemma \ref{lem12}, there exist $g_1,\dots, g_n\in A(X)$ such that $g_i \geq 0$, $\hat{g_i}(x_0) = 0$, $W_i = \{\hat{g_i}<1\} \subseteq U_i$, $1\leq i\leq n$.
Let $g = \sum^n_{i=1}g_i \in A(X)$.  Then $g \geq 0$, $\hat{g}(x_0)  = 0$, and $\hat{g}\geq 1$ outside $U$.
By Lemma \ref{lem11}, there exists $h\in A(X)$ such that $1/3 \leq h \leq 2/3$, $h(x) = 2/3$ if $g(x) \geq 2/3$, $h(x) = 1/3$ if $g(x) \leq 1/3$.  Set $f = 2- 3h\in A(X)$.  Then $0 \leq f \leq 1$. Since $\{\hat{g} < 1/3\}$ is an open  neighborhood of $x_0$ in $\cA X$ and $f = 1$ on $\{\hat{g} <1/3\} \cap X$, $\hat{f}(x_0) = 1$ by continuity. Similarly, $\{\hat{g} > 2/3\}$ is an open neighborhood of $\cA X\bs U$ and $f= 0$ on $\{\hat{g} > 2/3\}\cap X$.  Thus $\hat{f} = 0$ on $\cA X\bs U$. 
\end{proof}

\begin{cor}\label{cor15}
Let $A(X)$ be a near vector
lattice.  If $U$ and $V$ are open subsets of $\cA X$ such that $\ol{U}\subseteq V$, then there exists a function $f\in A(X)$, $0\leq f\leq 1$, such that
\[ \hat{f}(x) = \begin{cases}
                     1 & x\in U\\
                     0 & x\notin V.
                   \end{cases}\]
\end{cor}

\begin{proof}
By Proposition \ref{prop14}, for each $x\in \ol{U}$, there exists $f_x\in A(X)$, $0 \leq f_x\leq 1$, such that $\hat{f_x}(x) = 1$ and $\hat{f_x} = 0$ outside $V$.  The sets $\{\hat{f_x} > 1/2\}, x\in \ol{U}$, form  an open cover of $\ol{U}$.  
Hence there exist $x_1,\dots, x_n \in \ol{U}$ such that $\ol{U} \subseteq \cup^n_{i=1}\{\widehat{f_{x_i}}> 1/2\}$.
Let $g = \sum^n_{i=1}f_{x_i}$.  Then $g\in A(X)$, $g\geq 0$, $\hat{g}(x) > 1/2$ for all $x\in \ol{U}$ and $\hat{g} = 0$ outside $V$.
By Lemma \ref{lem11}, there exists $h\in A(X)$, $1/4 \leq h\leq 1/2$, such that 
\[ h(x) = \begin{cases}
             1/2 & \text{if $g(x) \geq 1/2$}\\
             1/4 & \text{if $g(x) \leq 1/4$}.
             \end{cases}.\]
Since $g(x) > 1/2$ for all $x\in U\cap X$, $h(x) = 1/2$ for all $x\in U \cap X$.  By continuity of $\hat{h}$ and the density of $X$ in $\cA X$, $\hat{h}(x) = 1/2$ for all $x\in \ol{U}$.
Similarly, $\hat{h}(x) = 1/4$ on the set $\{\hat{g} < 1/4\}$, which contains $\cA X \bs V$.
Finally, take $f = 4h-1$.  Then $f$ has the desired properties.
\end{proof}

\begin{prop}\label{prop18.5}
Let $A(X)$ be a near vector lattice.  Then the space $A_b(\cA X)$ consisting of all functions $\hat{f}$, where $f$ is a bounded function in $A(X)$, is a near vector lattice of functions on $\cA X$.
\end{prop}

\begin{proof}
Obviously $A_b(\cA X)$ is a vector subspace of $C(\cA X)$. Since $A(X)$ is unital, so is $A_b(\cA X)$.  By Proposition \ref{prop14}, $A_b(\cA X)$ separates points from closed sets.
Let $f$ be a bounded function in $A(X)$.  Since $A(X)$ is a near vector lattice,  there exists $u\in A(X)$ such that $u(x) = f(x)$ if $f(x) \geq 2/3$, $u(x) = 1/3$ if $f(x) \leq 1/3$, and $1/3\leq u(x) \leq 2/3$ if $1/3\leq f(x) \leq 2/3$.  Note that $u$ is a bounded function. By Corollary \ref{cor15}, there exists $v\in A(X)$ such that $0\leq v\leq 1$, $\hat{v} = 1$ on $\{\hat{u} > 2/3\}$ and $\hat{v} = 0$ outside $\{\hat{u}> 1/3\}$. Let $g = u + v/3 - 1/3$.  Then $g$ is a bounded function in $A(X)$ and hence $\hat{g} \in A_b(\cA X)$.
Since 
\[ \{\hat{f} \leq 0\} \subseteq \{\hat{f} < \frac{1}{3}\} \subseteq \ol{\{f < \frac{1}{3}\}}\]
(where $\{f> 1/3\}\subseteq X$ and the closure is taken in $\cA X$),
$\hat{u} = 1/3$ and $\hat{v} =0$ on the set $\{\hat{f} \leq 0\}$.
Thus $\hat{f}(x) \leq 0$ implies $\hat{g}(x) = 0$.
Similarly, $\{\hat{f} \geq 1\} \subseteq \{\hat{f}> 2/3\} \subseteq \ol{\{f > 2/3\}}$ and hence on the set $\{\hat{f} \geq 1\}$, $\hat{u} = \hat{f} \geq  2/3$ and $\hat{v} =1$.  Therefore, $\hat{f}(x) \geq 1$ implies $\hat{g}(x) = \hat{f}(x)$.
Observe that $u \geq 1/3$ and $v\geq 0$ and thus $g \geq 0$.
Also, since $\{\hat{f} < 1\} \subseteq \ol{\{f< 1\}}$ and $u < 1$ on $\{f< 1\}$, $\hat{u} \leq 1$ on the set $\{\hat{f}<1\}$.
Thus $\hat{f}(x) < 1$ implies $\hat{g}(x) \leq \hat{u}(x) \leq 1$.
\end{proof}

\begin{cor}\label{cor18}
Let $A(X)$ be a near vector lattice.  Then $A_b(\cA X)$ and 
\[ B(\cA X) = \{\hat{f}: f\in A(X), 0 \leq f\leq 1\}\]
are articulated, compatible and directed sets of functions on $\cA X$.
\end{cor}

\begin{proof}
Since $B(\cA X)$ is precisely the set $\{\hat{f}\in A_b(\cA X): 0\leq \hat{f}\leq1\}$, the conclusions follow from Propositions \ref{prop18.5} and \ref{prop13}.
\end{proof}

\begin{prop}\label{prop16}
Let $A(X)$ be a near vector
lattice.  Suppose that $x_0 \in \cA X$ and  $f,g$ are functions in $A(X)$ such that $\hat{f} \geq  \hat{g}$ on an open  neighborhood of $x_0$ in $\cA X$. Then there exists $h\in A(X)$ such that $h \geq f, g$ and that $\hat{h} = \hat{f}$ on an open neighborhood of $x_0$ in $\cA X$.
\end{prop}

\begin{proof}
Let $U$ and $V$ be open sets in $\cA X$ such that $x_0 \in U \subseteq \ol{U} \subseteq V$ and that $\hat{f} \geq  \hat{g}$ on $V$.
The function $u= f-g \in A(X)$ and $u \geq  0$ on $V\cap X$. Since $A(X)$ is a near vector lattice, there exists $v\in A(X)$ such that $v(x) = u(x)$ if $u(x) \geq -1$, $v(x) = -2$ if $u(x) \leq -2$ and $-2\leq v(x) \leq -1$ if $-2\leq u(x)\leq -1$.  By Corollary \ref{cor15}, there exists $w \in A(X)$ such that $0\leq w\leq 1$, $\hat{w}(x) = 1$ if $x\in U$ and $\hat{w}(x) = 0$ if $x\notin V$.  Let $h = g+ v+ 2 -2w\in A(X)$.
If $x\in U\cap X$, then $u(x)\geq 0$ and hence $v(x) = u(x) \geq 0$; also, $w(x) = 1$.  Therefore,
 $h(x) = g(x)+u(x) + 2-2w(x) = f(x)$.  Thus $\hat{h} = \hat{f}$ on $U$.

If $x\in X$ and $f(x) - g(x) = u(x) < 0$, then $x\notin V$.  Hence $w(x) = 0$.  Observe that $v \geq -2$.  Thus $h(x) = g(x) + v(x) +2 \geq g(x) > f(x)$.  
Finally, if $x\in X$ and $f(x)-g(x) = u(x)  \geq 0$, then $v(x) = u(x)$.  Since $w(x) \leq 1$ as well,  \[h(x) = g(x) + u(x) + 2-2w(x) \geq g(x)+u(x) = f(x) \geq g(x).\]
This completes the proof  that $h \geq f, g$.
\end{proof}

A subset $S$ of an ordered vector space $E$  is said to be {\em order bounded} if there exist $u,v \in E$ such that $u \leq x \leq v$ for all $x\in A$.  A {\em compactification} of a Hausdorff topological space $X$ is a compact Hausdorff space $\hat{X}$ that contains a dense subset which is homeomorphic to $X$.  If $\hat{X}$ is a compactification of $X$, we will regard $X$ as a dense subspace of $\hat{X}$.

\begin{prop}\label{prop18}
Let $A(X)$ and $A(Y)$ be near vector lattices and let $T: A(X) \to A(Y)$ be an order isomorphism.  
For any order bounded subset $S$ of $A(X)$, there are compactifications $\hat{X}$ and $\hat{Y}$ of $X$ and $Y$ respectively, and a homeomorphism $\vp_S : \hat{X} \to \hat{Y}$ such that for any $f, g \in S$ and any open set $U$ in $\hat{X}$, $f \geq  g$ on $U\cap X$ if and only if $Tf \geq  Tg$ on $\vp_S(U)\cap Y$.  
\end{prop}

\begin{proof}
By translation, we may assume that $T0 = 0$ and that there is a function $f_1 \in A(X)$ such that $0 \leq f \leq f_1$ for all $f\in S$.  
Let 
\[ f_0 = 2f_1 + T^{-1}(2Tf_1)+ 1 + T^{-1}1.\]  
Then $0 \leq f  \leq   \frac{1}{2}f_0$, $0 \leq Tf \leq  \frac{1}{2}Tf_0$ for all $f\in S$, $f_0 \geq 1$ and $Tf_0 \geq 1$.
Define vector subspaces $F(X)$ and $F(Y)$ of $C(X)$ and $C(Y)$ respectively by
\[ F(X) = \{f/f_0: f\in A(X)\} \quad \text{and} \quad F(Y) = \{g/Tf_0: g\in A(Y)\}.\]
If $f\in A(X)$, there exists $u\in A(X)$ such that $u(x) = f(x)$ if $f(x) \geq 1$, $u(x) = 0$ if $f(x) \leq 0$ and $0 \leq u(x) \leq 1$ if $0 \leq f(x) \leq 1$.  Then $u/f_0 \in F(X)$.
Since $f_0 \geq 1$, if $f/f_0 \in F(X)$ and $(f/f_0)(x) \geq 1$, then $f(x) \geq 1$ and hence $(u/f_0)(x) = (f/f_0)(x)$.
If $(f/f_0)(x) \leq 0$, then $f(x) \leq 0$ and hence $(u/f_0)(x) = 0$.  If $0 \leq (f/f_0)(x) \leq 1$, then either $0 \leq f(x) \leq 1$ or $f(x) \geq 1$.  In the former case, $0\leq u(x) \leq 1$ and thus $0\leq (u/f_0)(x) \leq 1$.
In the latter case, $u(x) = f(x)$ and thus $0 \leq (u/f_0)(x) = (f/f_0)(x) \leq 1$.  Obviously, $F(X)$ is unital.  It also separates points from closed sets since $A(X)$ does so.  This proves that $F(X)$ is a near vector lattice.  Similarly, $F(Y)$ is a near vector lattice.

Denote by $\hat{X}$ and $\hat{Y}$ the $F(X)$- and $F(Y)$- compactifications of $X$ and $Y$ respectively.
By Corollary \ref{cor18}, the sets
\[ B(\hat{X}) = \{\hat{f}: f\in F(X), 0 \leq f\leq 1\}  \text{\ and\ } B(\hat{Y}) = \{\hat{g}: g\in F(Y), 0\leq g \leq 1\}\]
are articulated, compatible and directed sets of functions.
It is easy to check that the map $\hat{T}: B(\hat{X}) \to B(\hat{Y})$ defined by $\hat{T}(\hat{f}) = [T(ff_0)/Tf_0]\,^{\hat{ }}$
is an order isomorphism.
By Theorem \ref{thm8}, there is a homeomorphism $\vp_S: \hat{X} \to \hat{Y}$ such that for any $\hat{u} \leq \hat{v}$ in $B(\hat{X})$ and any open set $U$ of $\hat{X}$, $\hat{u} = \hat{v}$ on $U$ if and only if $\hat{T}\hat{u} = \hat{T}\hat{v}$ on $\vp_S(U)$.

Now suppose $f, g\in S$, $U$ is an open set in $\hat{X}$ and $f\geq  g$ on $U \cap X$.
Fix $y_0 \in \vp_S(U)\cap Y$.  Then $x_0 = \vp_S^{-1}(y_0) \in U$. Since $0 \leq f/f_0, g/f_0 \leq 1$ are functions in $F(X)$ , $(f/f_0)^{\hat{ }}, (g/f_0)^{\hat{ }} \in B(\hat{X})$.  Moreover, $f/f_0 \geq g/f_0$ on $U \cap X$ implies that $(f/f_0)^{\hat{ }}\geq (g/f_0)^{\hat{ }}$ on $U$.
By Proposition \ref{prop16}, there exists $h_1\in F(X)$ such that $h_1 \geq f/f_0, g/f_0$ and that $\hat{h_1} = (f/f_0)^{\hat{ }}$ on an open neighborhood $V$ of $x_0$ in $\hat{X}$.
Since $F(X)$ is a near vector lattice, there exists $h\in F(X)$ such that for $x\in X$, $h(x) = h_1(x)$ if $h_1(x) \leq 3/4$, $h(x) = 1$ if $h_1(x) \geq 1$ and $3/4\leq h(x) \leq 1$ if $3/4 \leq h_1(x)\leq 1$.
By choice of $f_0$, $0 \leq f/f_0, g/f_0 \leq 1/2$. Thus $h\geq f/f_0, g/f_0\geq 0$.  Clearly $h \leq 1$.  Therefore, $\hat{h} \in B(\hat{X})$.
Note that ${h_1} = f/f_0\leq 3/4$ on $V\cap X$.  Thus $h= h_1 = f/f_0$ on $V\cap X$.
Hence $\hat{h} = (f/f_0)^{\hat{ }}$ on the open neighborhood $V$ of $x_0$.
By the previous paragraph,  $\hat{T}(f/f_0)^{\hat{ }}= \hat{T}\hat{h}$ on $\vp_S(V)$.  But $\hat{T}\hat{h} \geq \hat{T}(g/f_0)^{\hat{ }}$.  Therefore,
$[Tf/Tf_0]^{\hat{ }} \geq [Tg/Tf_0]^{\hat{ }}$ on $\vp_S(V)$.  Consequently, 
$Tf\geq Tg$ on $\vp_S(V) \cap Y$.
Since $x_0 \in V$, $y_0 \in \vp_S(V)$.  Thus $Tf(y_0) \geq Tg(y_0)$.
As this holds for all $y_0 \in \vp_S(U) \cap Y$, we see that $Tf \geq Tg$ on $\vp_S(U) \cap Y$.
The reverse implication follows by symmetry.
\end{proof}

The next result allows us to remove the dependence of the homeomorphism on the particular order bounded set in Proposition \ref{prop18}.

\begin{lem}\label{lem19.1}
Let $X$ and $Y$ be Hausdorff topological spaces and let $T: A(X)\to A(Y)$  be a map between  vector subspaces $A(X)$ and $A(Y)$ of $C(X)$ and $C(Y)$ respectively. 
Suppose that $S_1\subseteq S_2$ are subsets of $A(X)$.
Assume that for $i=1,2$, there are compactifications $\hat{X}_i$ and $\hat{Y}_i$ of $X$ and $Y$ respectively and a homeomorphism $\vp_i:\hat{X}_i \to \hat{Y}_i$ such that for all $f, g \in S_i$ and any open set $U$ in $\hat{X}_i$, $f\geq g$ on  $U\cap X$ if and only if $Tf \geq Tg$ on $\vp_i(U)\cap Y$.
Also assume that if $y\in Y$ and $D$ is a closed subset of $Y$ not containing $y$, then there exist $f, g \in S_1$ with $Tf(y)\neq Tg(y)$ and $Tf = Tg$ on $D$. 
Then for all $f,g\in S_2$ and all open sets $U$ in $\hat{X}_1$, $f\geq g$ on  $U\cap X$ implies that $Tf \geq Tg$ on $\vp_1(U)\cap Y$.
\end{lem}

\begin{proof}
Let $U$ be an open set in $\hat{X}_1$. For notational convenience, let $U_1 = U$.  There is  an open set  $U_2$ in $\hat{X_2}$ such that  $U_1 \cap X=  U_2\cap X$.
First we show that  $\ol{\vp_1(U_1)\cap Y}^{{Y}} \subseteq  \ol{\vp_2(U_2)\cap Y}^{{Y}}$.  
Suppose otherwise.  Let $D =\ol{\vp_2(U_2)\cap Y}^{{Y}}$. Then $D$ is a closed set in $Y$ and there exists  $y\in \vp_1(U_1)\cap Y$ such that  $y \notin D$.
By the assumption, there exist  $f', g' \in S_1\subseteq S_2$ with $Tf'(y)\neq Tg'(y)$ and $Tf' = Tg'$ on $D$. 
In particular, $Tf' =Tg'$ on $\vp_2(U_2) \cap Y$ and hence $f' = g'$ on $U_2\cap X= U_1\cap X$.  But then $Tf' = Tg'$ on $\vp_1(U_1)\cap Y$.  This contradicts the fact that $Tf'(y) \neq Tg'(y)$.

Now assume that $f, g\in S_2$ and $f\geq g$ on $U\cap X = U_1\cap X= U_2\cap X$.
Then $Tf \geq Tg$ on $\vp_2(U_2)\cap Y$.  
By continuity of $Tf$ and  $Tg$, $Tf \geq Tg$ on  $\ol{\vp_2(U_2)\cap Y}^{{Y}}$, which contains $\ol{\vp_1(U_1)\cap Y}^{{Y}}$ by the previous paragraph.  This proves that $Tf \geq Tg$ on $\vp_1(U_1) \cap Y= \vp_1(U)\cap Y$, as desired.
\end{proof}

\begin{thm}\label{thm18.1}
Let $A(X)$ and $A(Y)$ be near vector lattices and let $T: A(X) \to A(Y)$ be an order isomorphism.  
There are compactifications $\hat{X}$ and $\hat{Y}$ of $X$ and $Y$ respectively, and a homeomorphism $\vp : \hat{X} \to \hat{Y}$ such that for any $f, g \in A(X)$ and any open set $U$ in $\hat{X}$, $f \geq  g$ on $U\cap X$ if and only if $Tf \geq  Tg$ on $\vp(U)\cap Y$.  
\end{thm}

\begin{proof}
Without loss of generality, we may assume that $T0 =0$. Let $f_0 = T^{-1}1\in A(X)$.
Since $A(X)$ is directed by Proposition \ref{prop13}, there exists $f_1\in A(X)$ such that $f_1 \geq f_0, 1$.  The set $S = \{f\in A(X): 0 \leq f\leq f_1\}$ is an order bounded subset of $A(X)$ such that $f\in S$ for any $f\in A(X)$ with $0\leq f\leq 1$ and $T^{-1}g\in S$ for any $g\in A(Y)$ with $0\leq g\leq 1$.
By Proposition \ref{prop18}, there are compactifications $\hat{X}$ and $\hat{Y}$ of $X$ and $Y$ respectively, and a homeomorphism $\vp:\hat{X}\to \hat{Y}$ such that for any $h_1,h_2\in S$ and any open set $U$ in $\hat{X}$, $h_1\geq h_2$ on $U\cap X$ if and only if $Th_1\geq Th_2$ on $\vp(U)\cap Y$.
Suppose that $f,g\in A(X)$ and $U$ is an open set in $\hat{X}$.  We show that $f\geq g$ on $U\cap X$ implies $Tf\geq Tg$ on $\vp(U)\cap Y$. The converse follows by symmetry.
Choose $u_1, u_2\in A(X)$ such that 
\[ u_1 \leq 0, f, g \text{ and } f_1,f,g \leq u_2.\]
Let $S' = \{h\in A(X): u_1\leq h\leq u_2\}$.
Then $S'$ is an order bounded set in $A(X)$ containing $S$.
By Proposition \ref{prop18}, there are compactifications $\hat{X}'$ and $\hat{Y}'$ of $X$ and $Y$ respectively, and a homeomorphism $\vp': \hat{X}'\to \hat{Y}'$ such that for any $h_1, h_2\in S'$ and any open set $U'$ in $\hat{X}'$, $h_1\geq h_2$ on $U'\cap X$ if and only if $Th_1\geq Th_2$ on $\vp'(U')\cap X$.
Suppose $y \in Y$ and $D$ is a closed set in $Y$ not containing $y$.
There is a closed set $D_1 \in \cA Y$ such that $D_1\cap Y = D$.  By  Proposition \ref{prop14}, there exists $h \in A(Y)$, $0 \leq h\leq 1$, such that $h(y) \neq 0$ and $\hat{h}(z) = 0$ for all $z\in D_1$, where $\hat{h}$ is the continuous extension of $h$ onto $\cA Y$.
In particular, $h(z) =0$ for all $z\in D$.
Since  the functions $0= T^{-1}0$ and $T^{-1}h$ lie in $S$, this verifies that Lemma \ref{lem19.1} applies to the map $T:A(X)\to A(Y)$, the sets $S \subseteq S'\subseteq A(X)$ and the homeomorphisms $\vp:\hat{X}\to \hat{Y}$, $\vp':\hat{X}'\to \hat{Y}'$. Now $f, g\in S'$ and $f \geq g$ on $U\cap X$, where $U$ is an open set in $\hat{X}$. By the lemma, $Tf \geq Tg$ on the set $\vp(U)\cap Y$.
\end{proof}

In view of Theorem \ref{thm18.1}, it would be of interest to provide examples of near vector lattices. If $X$ is a Hausdorff topological space and $A(X)$ is a subspace of $C(X)$, let  $A_b(X)$ consist of the bounded functions in $A(X)$ and let $A^\loc(X)$ 
be the space of all functions $f\in C(X)$ such that for every $x_0\in X$, there are an open neighborhood $U$ of $x_0$ and a function $g\in A(X)$ such that $f=g$ on $U$.  The space $A^\loc_b(X)$ is the subspace of all bounded functions in $A^\loc(X)$.

\bigskip

\noindent{\bf Examples A}. Let $X$ be a Hausdorff topological space unless otherwise specified. 
\begin{enumerate}
\item  Any unital vector sublattice $A(X)$ of $C(X)$ that separates points from closed sets is a near vector lattice.
\item If $A(X)$ is a near vector lattice, then so is $A_b(X)$.
\item Let $X$ be an open set in a Banach space.  For $1\leq p\leq \infty$, let $C^p(X)$ be the space of all $p$-times continuously differentiable real-valued functions on $X$.  
If $C^p(X)$ separates points from closed sets, then $C^p(X)$ is a near vector lattice.
\item Let $X$ be an open set in a Banach space and let $1\leq p \leq \infty$.  Denote by $C^p(\ol{X})$, the space of all  continuous functions $f\in C(\ol{X})$ such that its restriction $f_{|X}\in C^p(X)$.  Denote by  $C^p_*(\ol{X})$ the subspace of all functions $f\in C^p(\ol{X})$ such that $D^kf$ is bounded on $X$  for $0\leq k\leq p$ ($0\leq k < \infty$ if $p =\infty$). These spaces (considered as subspaces of $C(\ol{X})$) are near vector lattices provided that they separate points from closed sets.
\end{enumerate}

\begin{proof}
Items (a) and (b) are obvious.
For (c) and (d), let $f$ be a given function in one of the respective spaces.  Choose a $C^\infty$ function $h:\R\to \R$ such that $h(t) =0$ if $t \leq 0$, $h(t) = t$ if $t \geq 1$ and $0\leq h(t)\leq 1$ if  $0\leq t\leq 1$. Then it is easy to verify that $g = h\circ f$ is a function with the desired properties in the definition of near vector lattice.
\end{proof}

\section{Spaces $C(X)$ with $X$ realcompact}

Recall that a completely regular Hausdorff topological space $X$ is {\em realcompact} if for any $x\in \beta X \bs X$, there exists $f\in C(X)$ such that $\hat{f}(x) = \infty$, where $\hat{f}$ is the continuous (extended real-valued) extension of $f$ onto $\beta X$. The aim of this section is to prove the following theorem.  When $X$ and $Y$ are compact, this result was obtained by F.\ Cabello Sanchez \cite{C}.

\begin{thm}\label{thm19}
Let $X$ and $Y$ be realcompact spaces and let $T:C(X)\to C(Y)$ be an order isomorphism.
Then there exists a homeomorphism $\vp: X\to Y$ such that for any open set $U$ in $X$ and any $f, g \in C(X)$, $f \geq g$ on $U$ if and only if $Tf \geq Tg$ on $\vp(U)$.
\end{thm}

Note that if $f_0\in C(X)$ is a function that is strictly positive on $X$, then $\{f/f_0: f\in C(X)\} = C(X)$.
Thus, taking $A(X) = C(X)$ and $A(Y) = C(Y)$ in the proof of Proposition  \ref{prop18}, the spaces $F(X)$ and $F(Y)$ 
are $C(X)$ and $C(Y)$ respectively.  Hence the compactifications $\hat{X}$ and $\hat{Y}$ in Theorem \ref{thm18.1} are equal to $\beta X$ and $\beta Y$ respectively.  By the theorem, there is a homeomorphism $\vp:\beta X \to \beta Y$ such that for any $f, g\in C(X)$ and any open set $U$ in $\beta X$, $f \geq g$ on $U \cap X$ if and only if $Tf \geq Tg$ on $\vp(U)\cap Y$.  
To complete the proof of Theorem \ref{thm19}, it remains to show that $\vp$ maps $X$ onto $Y$.  
The argument below is inspired by  results of a similar nature in \cite{A, ABN}.

\begin{lem}\label{lem5}
Let $Y$ be a realcompact space and let $y_0 \in \beta Y \bs Y$.  There exist  open sets $U_n$ and $V_n$ in $\beta Y$, $n \in \N$, such that
\begin{enumerate}
\item $\ol{U_n} \subseteq V_n$ for all $n$;
\item $y_0 \in \ol{\cup^\infty_{n=m}U_n}$ for all $m$;
\item $Y \bigcap \cap^\infty_{m=1}\ol{\cup^\infty_{n=m}V_n} = \emptyset$;
\item $\ol{V_n} \cap \ol{V_m} = \emptyset$ if  $n\neq m$.
\end{enumerate}
\end{lem}

\begin{proof}
There exists $0 \leq f\in C(Y)$ such that $\hat{f}(y_0) = \infty$, where $\hat{f}$ is the continuous extension of $f$ onto $\beta Y$.  Let 
\[ Y_k = \{y\in Y: f(y) \in \cup^\infty_{n=0}[4n+k,4n+k+1]\},\ 0 \leq k \leq 3.\]
Then $Y = \cup^3_{k=0}Y_k$.  Hence there exists $k$ such that $y_0 \in \ol{Y_k}$, where the closure is taken in $\beta Y$.  Without, loss of generality, assume that $y_0 \in \ol{Y_0}$.
Define $U_n = \hat{f}^{-1}(4n-\frac{1}{2},4n+\frac{3}{2})$ and $V_n = \hat{f}^{-1}(4n-1,4n+2)$ for all $n$.
Clearly, $U_n$ and $V_n$ are open sets in $\beta Y$ such that $\ol{U_n} \subseteq V_n$.
Moreover,
\[y_0 \in \ol{Y_0} \subseteq \ol{\cup^\infty_{n=1}U_n} = \ol{\cup^{m-1}_{n=1}{U_n}} \bigcup  \ol{\cup^\infty_{n=m}U_n}\]
for any $m$. Since $\hat{f}(y)$ is finite for all $y\in\ol{\cup^{m-1}_{n=1}{U_n}}$,
$y_0 \notin\ol{ \cup^{m-1}_{n=1}{U_n}}$.  Thus $y_0\in \ol{\cup^\infty_{n=m}U_n}$.
Similarly, if $y \in \ol{\cup^\infty_{n=m}V_n}$, then $\hat{f}(y)\geq 4m-1$.  Thus $y \in  \cap^\infty_{m=1}\ol{\cup^\infty_{n=m}V_n}$ implies $\hat{f}(y) = \infty$ and hence $y \notin Y$.
Finally, if $y \in \ol{V_m}$, then $\hat{f}(y) \in [4m-1, 4m+2]$.  Hence $y \notin \ol{V_n}$ for any $n \neq m$.
\end{proof}

\begin{lem}\label{lem6.1}
Let $Y$ be a subspace of a topological space $Z$ and let $(V_n)$ be a sequence of open sets in $Z$ such that 
\begin{enumerate}
\item $Y \bigcap \cap^\infty_{m=1}\ol{\cup^\infty_{n=m}V_n} = \emptyset$;
\item $\ol{V_n} \cap \ol{V_m} = \emptyset$ if $n\neq m$.
\end{enumerate}
Suppose that $g_n:Y\to \R$ is a continuous function on $Y$ such that $\{g_n\neq 0\} \subseteq V_n$ for all $n$.
Define $g:Y\to \R$ by $g(y) = g_n(y)$ if $y\in Y\cap V_n$ and $0$ otherwise.  Then $g$ is continuous on $Y$.
\end{lem}

\begin{proof}
Since the sets $V_n$ are pairwise disjoint, $g$ is well-defined.  
Suppose that $y_0 \in Y$.  If $y_0 \notin \ol{\cup^\infty_{n=1}V_n}$, choose an open set $U$ in $Z$ containing $y_0$ such that $U \cap\cup^\infty_{n=1}V_n = \emptyset$. 
If $y\in U \cap Y$, then $g(y) = 0 = g(y_0)$.  Hence $g$ is continuous at $y_0$.  Now assume that $y_0 \in \ol{\cup^\infty_{n=1}V_n}$.
There exists $m$ such that $y_0 \notin \ol{\cup^\infty_{n=m}V_n}$.
Thus $y_0 \in \cup^{m-1}_{n=1}\ol{V_n}$.
Pick $1 \leq n_0<m$ such that $y_0 \in \ol{V_{n_0}}$.  Then $y_0 \notin \ol{V_n}$, $1\leq n\neq n_0 < m$ and $y_0\notin  \ol{\cup^\infty_{n=m}V_n}$.  Hence there exists an open neighborhood $W$ of $y_0$ in $Z$ such that $W \cap \cup_{n\neq n_0}V_n = \emptyset$. As a result, $g = g_{n_0}$ on the set $W \cap Y$.  Therefore, $g$ is continuous at $y_0$.
\end{proof}

\begin{prop}\label{prop21}
Let $X$ and $Y$ be realcompact spaces.  Suppose that $T:C(X) \to C(Y)$ is a bijection.  Let $\vp:\beta X\to \beta Y$ be a homeomorphism so that for any $f, g\in C(Y)$ and any open set $V$ in $\beta Y$, $f = g$ on $Y\cap V$ if and only if  $T^{-1}f = T^{-1}g$ on the set $X\cap \vp^{-1}(V)$.  Then $\vp(X) = Y$.
\end{prop}

\begin{proof}
By symmetry, it suffices to show that $\vp(X) \subseteq Y$.
Assume that there exists $x_0\in X$ so that $\vp(x_0) = y_0 \in \beta Y \bs Y$.
Choose sets $U_n$ and $V_n$ in $\beta Y$ by Lemma \ref{lem5}.
For each $n$, there exists a continuous function $h_n:\beta Y\to [0,1]$ such that $h_n=1$ on $U_n$ and $h_n = 0$ outside $V_n$.  Consider the function $g_n : Y\to \R$ given by $g_n(y) = h_n(y)T(n1_X)(y)$.
Taking $Z = \beta Y$ in Lemma \ref{lem6.1} and applying that lemma, we find a continuous function $g: Y\to \R$ such that  $g(y) = g_n(y)$ if $y\in Y\cap V_n$ and $0$ otherwise.
Now $g = T(n1_X)$ on $Y\cap U_n$.  By the assumption, $T^{-1}g = n1_X$ on $X\cap\vp^{-1}(U_n)$.
As a result, $T^{-1}g(x) \in [m,\infty)$ on $X\cap\vp^{-1}(\cup^\infty_{n=m}U_n)$. 
Since $y_0 \in \ol{\cup^\infty_{n=m}U_n}$ for all $m$, $x_0 \in \ol{\vp^{-1}(\cup^\infty_{n=m}U_n)}$ for all $m$.
But this implies that $T^{-1}g(x_0) \in [m,\infty)$ for all $m$, which is impossible.
\end{proof}

\begin{proof}[Proof of Theorem \ref{thm19}]
We have seen above that  there exists a homeomorphism $\vp:\beta X \to \beta Y$ such that for any open set $U$ in $\beta X$, $f \geq g$ on $U\cap X$ if and only if $Tf \geq Tg$ on $\vp(U)\cap Y$.  By Proposition \ref{prop21}, $\vp(X) = Y$.  Thus the restriction of $\vp$ to $X$ is a homeomorphism from $X$ onto $Y$ satisfying Theorem \ref{thm19}.
\end{proof}

\begin{cor}\label{cor25}
Let $X$ and $Y$ be realcompact spaces.  Then $C(X)$ is order isomorphic to $C(Y)$ if and only if $X$ are $Y$ are homeomorphic.
\end{cor}

If $X$ is compact Hausdorff, then $C(X)$ and $C_b(X)$ are identical.  For a non-compact realcompact space $X$, one can distinguish $C(X)$ and $C_b(X)$ order isomorphically.  It is a classical fact that every space $C(X)$ is (linearly) order  isomorphic to a space $C(Y)$ for some realcompact space $Y$; see, e.g., \cite[Theorem 3.9 and 8.8(a)]{GJ} .

\begin{cor}\label{cor26}
Let $X$  be a realcompact space and let $Y$ be a topological space.  If $X$ is not compact, then $C(X)$ is not order isomorphic to $C_b(Y)$.
\end{cor}

\begin{proof}
By \cite[Theorem 3.9]{GJ}, $C_b(Y)$ is (linearly) order isomorphic to $C_b(Y')$ for some Hausdorff completely regular space $Y'$.  Also, $C_b(Y')$ is (linearly) order isomorphic to $C(\beta Y')$. Thus, if $C(X)$ is order isomorphic to $C_b(Y)$, then it is order isomorphic to $C(\beta Y')$.  By Corollary \ref{cor25}, $X$ and $\beta Y'$ are homeomorphic.  Therefore, $X$ is compact.
\end{proof}

\section{Function spaces defined on metric spaces}

In the last section, we have seen that an order isomorphism between $C(X)$ spaces with $X$ realcompact leads to a homeomorphism of the underlying topological spaces.  In this section, we will see that a similar result holds in many instances when the underlying spaces $X$ and $Y$ are metric spaces.

The following result can be shown using the proof of  \cite[Corollary 4.3]{LL} . In the proof, it suffices to assume that the set of extensions of the bounded functions in $A(X)$ to the compactification $\cA X$ separate points from closed sets in $\cA X$; similarly for $A(Y)$.  By Proposition \ref{prop14}, near vector lattices possess this latter property.  A set of points $S$ in a metric space is {\em separated} if there exists $\ep > 0$ such that $d(x_1,x_2)> \ep$ whenever $x_1$ and $x_2$ are distinct points in $S$.

\begin{prop}
\label{prop26}
Let $A(X)$ and $A(Y)$ be near vector lattices defined on metric spaces $X$ and $Y$ respectively.
Assume that 
$A(Y) = A^\loc(Y)$ or $A^\loc_b(Y)$, or
\begin{enumerate}
\item[($\spadesuit$1)]  $Y$ is complete and for any separated sequence $(y_n)$ in $Y$, there exists  $g\in A(Y)$ such that the sequence $(g(y_n))$ diverges in $\R$.
\end{enumerate}
If $\vp:\cA X\to \cA Y$ is a homeomorphism, then $\vp(X) \subseteq Y$.
\end{prop}

Let us say that a vector subspace $A(Y)$ of $C(Y)$ satisfies condition
\begin{enumerate}
\item[($\spadesuit$2)] if $g\in A(Y)$ and $h\in C^\infty(\R)$ with $\|h^{(k)}\|_\infty <\infty$ for all $k\geq 1$, then $h\circ g\in A(Y)$.
\end{enumerate}
$A(Y)$ satisfies ($\spadesuit$) if it satisfies both ($\spadesuit$1) and ($\spadesuit $2).

\begin{lem}\label{lem28}
Let $Y$ be a metric space and let $A(Y)$ be a vector subspace of $C(Y)$ that satisfies condition ($\spadesuit$).
Assume that $1\leq g_0 \in A(Y)$.  Then $F(Y) = \{g/g_0: g\in A(Y)\}$ is a vector subspace of $C(Y)$ that satisfies condition ($\spadesuit$1).
\end{lem}

\begin{proof}
Since $A(Y)$ satisfies ($\spadesuit$), $Y$ is complete.
Let $(y_n)$ be a separated sequence in $Y$.
In the first case, assume that $(g_0(y_n))$ is bounded.
By using a subsequence if necessary, we may assume that $(g_0(y_n))$ converges to a real number $a$.  Because $g_0 \geq 1$, $a \neq 0$.
Since $A(Y)$ satisfies condition ($\spadesuit$1), 
there exists $g\in A(Y)$ such that $(g(y_n))$ diverges in $\R$.
Then $g/g_0 \in F(Y)$ and $((g/g_0)(y_n))$ diverges in $\R$.

Next, consider the case where $(g_0(y_n))$ is unbounded. By using a subsequence if necessary, we may assume that $g_0(y_{n+1}) > 2g_0(y_n)$ for all $n$.
One can then construct a function $h\in C^\infty(\R)$ with $\|h^{(k)}\|_\infty < \infty$ for all $k \geq 1$, such that $h(g_0(y_{2n})) = g_0(y_{2n})$ and $h(g_0(y_{2n-1})) = 0$ for all $n$.
By condition ($\spadesuit$2), $g = h\circ g_0\in A(Y)$.
Then $g/g_0\in F(Y)$ and $((g/g_0)(y_n))$ diverges in $\R$.
\end{proof}

\begin{thm}\label{thm27}
Let $A(X)$ and $A(Y)$ be near vector lattices defined on metric spaces $X$ and $Y$ respectively.
Assume that $A(X) = A^\loc(X)$ or $A^\loc_b(X)$ or satisfies ($\spadesuit$), and the same holds for $A(Y)$.
If $T: A(X) \to A(Y)$ is an order isomorphism, then there 
is a homeomorphism $\vp: X\to Y$ such that for any $f, g\in A(X)$ and any open set $U$ in $X$, $f \geq g$ on $U$ if and only if $Tf \geq Tg$ on $\vp(U)$.
\end{thm}

\begin{proof}
By Theorem \ref{thm18.1}, 
there are compactifications $\hat{X}$ and $\hat{Y}$ of $X$ and $Y$ respectively, and a homeomorphism $\vp:\hat{X}\to \hat{Y}$ such that for any $f,g\in A(X)$ and any open set $U$ in $\hat{X}$, $f\geq g$ on $U\cap X$ if and only if $Tf \geq Tg$ on $\vp(U)\cap Y$.
To complete the proof, it suffices to show that $\vp(X) = Y$.
We will show that $\vp(X)\subseteq Y$, the reverse inclusion follows by symmetry.
From the proofs of Theorem \ref{thm18.1} and of Proposition \ref{prop18}, we see that the compactification $\hat{Y}$ is induced by a near vector lattice  $F(Y)$, where $F(Y) = \{g/Tf_0: g\in A(Y)\}$ for some function $f_0 \in A(X)$ such that $Tf_ 0 \geq 1$.
If $A(Y) = A^\loc(Y)$ or $A^\loc_b(Y)$ respectively, then $F(Y) = F^\loc(Y)$ and $F^\loc_b(Y)$ respectively. (In the latter case, observe that there is a real constant function $M$ such that $1 \leq Tf_0 \leq M$.)
By Proposition \ref{prop26}, $\vp(X)\subseteq Y$.
Now suppose that $A(Y)$ satisfies ($\spadesuit$).  By Lemma \ref{lem28}, $F(Y)$ satisfies condition ($\spadesuit$1).  Thus $\vp(X)\subseteq Y$ by Proposition \ref{prop26}.
\end{proof}

Let $(X,d)$ be a metric space.  The space of Lipschitz functions $\Lip(X)$ consists of all $f:X\to \R$ such that there is a finite constant $K$ with $|f(x) - f(y)|\leq K\,d(x,y)$ for all $x,y \in X$.
The space $\lip(X)$ of little Lipschitz functions consists of all $f\in \Lip(X)$ such that
\[ \lim_{d(x,y)\to 0}\frac{|f(x)-f(y)|}{d(x,y)} = 0.\]
The space of uniformly continuous real-valued functions on $X$ is denoted by $U(X)$.
A set of functions $A(X)$ is said to be {\em uniformly separating} (cf.\ \cite{GJa}; see also \cite{CC1}) if whenever $U$ and $V$ are subsets of $X$ with $d(U,V) > 0$, then there exists $f\in A(X)$ such that $f = 1$ on $U$ and $f = 0$ on $V$.  $\Lip(X)$ and $U(X)$ are always uniformly separating, while $\lip(X)$ may not be.  However, if $0 < \vp < 1$, then $\lip_\al(X) = \lip(X,d^\al)$ is uniformly separating.

\medskip

\noindent{\bf Examples B}.
Let $X$ be a metric space. 
\begin{enumerate}
\item Let $A(X)$ be one of the spaces $C(X)$, $C_b(X)$, $\Lip^\loc(X)$, $\Lip^\loc_b(X)$, $\lip^\loc(X)$, $\lip^\loc_{b}(X)$,  $U^\loc(X)$ or $U^\loc_b(X)$.  In case $A(X) = \lip^\loc(X)$ or $\lip^\loc_b(X)$, we also assume that $A(X)$ separates points from closed sets.  Then $A(X)$ is a  near vector lattice such that $A(X) = A^\loc(X)$ or $A^\loc_b(X)$.
\item Let $A(X)$ be one of the spaces $\Lip(X)$, $\Lip_b(X)$, $\lip(X)$, $\lip_{b}(X)$, $U(X)$ and $U_b(X)$, where $X$ is a complete metric space.  In case $A(X) = \lip(X)$ or $\lip_b(X)$, we also assume that $A(X)$  is uniformly separating. Then $A(X)$ is a near vector lattice that satisfies ($\spadesuit$).
\item Let $X$ be an open subset of a Banach space and let $1\leq p \leq \infty$.  Suppose that $A(X)$ is  either $C^p(X)$ or $C^p_b(X)$, and that $A(X)$ separates points from closed sets.
Then $A(X)$ is a near vector lattice such that $A(X) = A^\loc(X)$ or $A^\loc_b(X)$.
\item Let $X$ be an open subset of a Banach space and let $1\leq p \leq \infty$. Suppose that $A(\ol{X})$ is one of the spaces $C^p(\ol{X})$, $C^p_b(\ol{X})$, $C^p_*(\ol{X})$, and that $A(\ol{X})$ separates points from closed sets in $\ol{X}$. Then $A(\ol{X})$ is a  near vector lattice (on $\ol{X}$) such that $A(\ol{X}) = A^\loc(\ol{X})$ or $A^\loc_b(\ol{X})$ or  satisfies condition ($\spadesuit$).
\end{enumerate}

\begin{proof}
(a) If $A(X)$ is a one of the spaces in (a), then $A(X)$ is a unital vector sublattice of $C(X)$ that separates points from closed sets.  Hence $A(X)$ is a near vector lattice by Examples A (a).
Obviously, $A(X) = A^\loc(X)$ or $A^\loc_b(X)$.

\noindent (b)  All of the spaces $A(X)$ in (b) are unital vector sublattices of $C(X)$ that separate points from closed sets.
Thus they are near vector lattices by Examples A (a).  We show that all of them satisfy ($\spadesuit$).
By hypothesis, $X$ is complete.  Let $(x_n)$ be a separated sequence in $X$.  Set $U = \{x_{2n-1}: n\in \N\}$ and $V = \{x_{2n}:n\in \N\}$.  Then $d(U,V) > 0$.  Since $A(X)$ is uniformly separating, then exists $f\in A(X)$ such that $f=1$ on $U$ and $f = 0$ on $V$.  In particular, $(f(x_n))$ diverges.  Thus $A(X)$ satisfies ($\spadesuit$1).
Condition ($\spadesuit$2) for any of the spaces $A(X)$ follows easily from the Mean Value Theorem.

\noindent (c) Let $A(X)$ be as given. Then $A(X)$ is a near vector lattice by Examples A (c) and (b). Clearly, if $A(X) = C^p(X)$, then $A(X) = A^\loc(X)$, and if $A(X) = C^p_b(X)$, then $A(X) = A^\loc_b(X)$.

\noindent (d) The spaces $C^p(\ol{X})$ and $C^p_*(\ol{X})$ are near vector lattices by Examples A (d).  Then it follows that  $C^p_b(\ol{X})$ is a near vector lattice by Examples A (b).
If $A(X) = C^p(\ol{X})$ or $C^p_b(\ol{X})$, then $A(\ol{X}) = A^\loc(\ol{X})$ or $A^\loc_b(\ol{X})$ respectively.
We show that the space $C^p_*(\ol{X})$ satisfies condition ($\spadesuit$).  Obviously $\ol{X}$ is complete.
Let $(x_n)$ be a separated sequence in $X$.  Denote the norm on the Banach space $E$ containing $X$ by $\|\cdot\|$.
There exists $\ep > 0$ such that $\|x_n-x_m\| > 3\ep$ if $n\neq m$.
Since $C^p_*(X)$ separates points from closed sets, there exists $h\in C^p_*(E)$ such that $h(0) > 0$ and $h(x) = 0$ if $\|x\| > \ep$.
Define $f: \ol{X}\to \R$ by $f(x) = h(x-x_{2n})$ if $\|x-x_{2n}\|\leq \ep$ for some $n$, and $f(x) = 0$ otherwise.
Then $f\in C^p_*(\ol{X})$ and $(f(x_n))$ diverges.  Thus $C^p_*(\ol{X})$ satisfies condition ($\spadesuit$1). 
Condition ($\spadesuit$2) is obvious.
\end{proof}

Examples B provide a large number of spaces to which Theorem \ref{thm27} is applicable.
We seek to further strengthen the theorem into one which gives a functional representation of the order isomorphism $T$.
Consider the following property of a vector subspace $A(X)$ of $C(X)$ at a point $x\in X$.
\begin{enumerate}
\item[($\heartsuit_x$)]   Either $x$ is an isolated point of $X$ or if $f\in A(X)$, $f\geq 0$ and $f(x)= 0$, then there exists $g\in A(X)$ such that $x \in \ol{\{f < g\}}  \cap \ol{\{g < 0\}}$.
\end{enumerate}

\begin{prop}\label{prop28}
Let $A(X)$ and $A(Y)$ be near vector lattices.  Suppose that $T:A(X)\to A(Y)$ is a bijective  function such that  there is a homeomorphism  $\vp:X\to Y$ so that for any $f, g\in A(X)$ and any open set $U$ in $X$, $f\geq g$ on $U$ if and only if $Tf \geq Tg$ on $\vp(U)$.  Assume that  $A(X)$ satisfies condition ($\heartsuit_x$). Set $y = \vp(x)$. Then there is a nondecreasing function $\Phi(y,\cdot): \R\to \R$ such that $Tf(y) = \Phi(y, f(\vp^{-1}(y)))$ for all $f\in A(X)$.
\end{prop}

\begin{proof}
It suffices to show that if $f,g \in A(X)$ and $f(x) = g(x)$, then $Tf(y) = Tg(y)$.  Indeed, once this has been shown, define $\Phi(y,\cdot): \R\to \R$ by $\Phi(y,t) = (Tt)(y)$, where $t$ denotes the constant function with value $t$.  Given $f\in A(X)$, let $t = f(x)$. Then $Tf(y) = Tt(y) = \Phi(y,t)$.  Since $T$ is order preserving  $\Phi(y,\cdot):\R\to\R$ is nondecreasing.

We now proceed to prove the assertion above.
Assume, if possible, that there are $f, g\in A(X)$ such that $f(x) = g(x)$ and $Tf(y) >  Tg(y)$. In particular, from the assumption, $x$ cannot be an isolated point of $X$.  By the assumption, $A(X)$ satisfies condition ($\heartsuit_x$).
Let $a = Tf(y) - Tg(y) > 0$.  Since $A(Y)$ is a near vector lattice, there exists $h\in A(Y)$ such that $h(z) = Tf(z) - Tg(z) $ if $Tf(z) - Tg(z) \geq a/2$, $h(z)= 0$ if $Tf(z)- Tg(z) \leq 0$, and $0\leq h(z) \leq a/2$ if $0\leq Tf(z) - Tg(z) \leq a/2$.  Then $h\geq 0$ and $h = Tf - Tg$ on a neighborhood of $y$.  Let $u = T^{-1}(h+Tg) \in A(X)$.
Since $Tu \geq Tg$, $u \geq g$.  Also, $Tu = Tf$ on a neighborhood of $y$ and hence $u = f$ on a neighborhood of $x$.  In particular, $u(x) = f(x) = g(x)$.  Thus $0 \leq u-g \in A(X)$ and $(u-g)(x) = 0$.
Since $A(X)$ satisfies condition ($\heartsuit_{x}$), there exists $v\in A(X)$ such that 
\begin{equation}\label{eq1}
x \in \ol{\{u-g < v\}} \cap \ol{\{v < 0\}} = \ol{\{u < g+ v\}} \cap \ol{\{g+ v < g\}}. 
\end{equation}
Now $Tu \leq T(g+v)$ on the set $\vp(\{u < g+v\})$ and $T(g+v) \leq Tg$ on the set $\vp(\{g+v< g\})$.
By (\ref{eq1}), $y$ lies in the closure of both of these sets.  Thus, by continuity, $Tu(y) \leq T(g+v)(y)\leq Tg(y)$.  However, $Tu(y) = Tf(y) > Tg(y)$, yielding a contradiction.  This completes the proof of the claim.
\end{proof}

The next theorem is an immediate consequence of Theorem \ref{thm27} and Proposition \ref{prop28}.  

\begin{thm}\label{thm29}
Let $A(X)$ and $A(Y)$ be near vector lattices defined on metric spaces $X$ and $Y$ respectively.
Assume that $A(X) = A^\loc(X)$ or $A^\loc_b(X)$ or satisfies condition ($\spadesuit$), and the same holds for $A(Y)$.  Suppose that  $T: A(X) \to A(Y)$ is an order isomorphism and let $\vp:X\to Y$ be the associated homeomorphism obtained in Theorem \ref{thm27}.  Let $x\in X$ and $y = \vp(x)$.  If $A(X)$ and $A(Y)$ satisfy condition ($\heartsuit_x$) and ($\heartsuit_y$) respectively, then there is an increasing homeomorphism $\Phi(y,\cdot):\R\to \R$ such that $Tf(y) = \Phi(y, f(\vp^{-1}(y)))$ for all $f\in A(X)$.
\end{thm}

\begin{proof}
Proposition \ref{prop28} applies to both $T$ and $T^{-1}$ at $x$ and $y = \vp(x)$.
Thus there are nondecreasing functions $\Phi(y,\cdot):\R\to \R$ and $\Psi(x,\cdot):\R\to \R$ such that
$Tf(y) = \Phi(y,f(\vp^{-1}(y)))$ and $T^{-1}g(x) = \Psi(x, g(\vp(x)))$ for all $f\in A(X)$ and all $g\in A(Y)$.  Considering the equations $TT^{-1}t = t$ and $T^{-1}Tt = t$ for all constant functions $t$
show that $\Phi(y,\cdot)$ and $\Psi(x,\cdot)$ are mutual inverses.  Thus $\Phi(y,\cdot):\R\to \R$ is an increasing bijection and so must be a homeomorphism.
\end{proof}

Say that $A(X)$ satisfies condition ($\heartsuit$) if it satisfies condition ($\heartsuit_x)$ at all points $x\in X$.

\medskip

\noindent{\bf Examples C}.
Let $X$ be a metric space unless otherwise stated.
\begin{enumerate}
\item If $A(X)$ is a near vector lattice that satisfies condition ($\heartsuit_x)$ at some $x\in X$, then $A_b(X)$ satisfies the same condition at $x$. 
\item If $A(X)$ is a unital vector sublattice of $C(X)$ that separates points from closed sets and satisfies condition ($\heartsuit$), then $A^\loc(X)$ and $A^\loc_b(X)$ satisfy condition ($\heartsuit$).
\item Let $A(X)$ be one the spaces $C(X)$,  $\Lip(X)$, $\lip_\al(X)$,  where $0< \al < 1$,  or $U(X)$. Then $A(X)$ is a  unital  vector sublattice of $C(X)$ that separates points from closed sets and satisfies condition ($\heartsuit$); hence the same is true of $A_b(X)$, $A^\loc(X)$ and $A^\loc_b(X)$.
\item Let $X$ be an open set in a Banach space and let $1\leq p \leq \infty$.  Suppose that  $A(X)$ is one of the  spaces $C^p(X)$ or  $C^p_b(X)$. If $A(X)$ separates points from closed sets, then $A(X)$ satisfies condition $(\heartsuit)$.
\item Let $X$ be an open set in a Banach space and let $1\leq p \leq \infty$.  Suppose that  $A(\ol{X})$ is one of the  spaces $C^p(\ol{X})$ or $C^p_b(\ol{X})$.  If $A(\ol{X})$ separates points from closed sets, then $A(\ol{X})$ satisfies condition $(\heartsuit)$.
\item Let $X$ be an open set in a Banach space and let $1\leq p \leq \infty$.   If $C^p_*(\ol{X})$ separates points from closed sets, then $A(\ol{X}) = C^p_*(\ol{X})$ satisfies condition ($\heartsuit_x)$ for all $x\in X$.
\end{enumerate}

\begin{proof}
\noindent (a) Suppose that $f\in A_b(X)$, $f\geq 0$ and $f(x) = 0$, where $x$ is an accumulation point of $X$.  Then there exists $h\in A(X)$ such that $x\in \ol{\{f< h\}}\cap \ol{\{h< 0\}}$.  In particular, $h(x) = 0$. Since $A(X)$ is  a near vector lattice, by Lemma \ref{lem11},  there exists $g\in A_b(X)$ such that $g(z) = h(z)$ if $|h(z)| \leq 1$.  Then $g = h$ on an open neighborhood of $x$.  Hence $x\in \ol{\{f< g\}}\cap \ol{\{g< 0\}}$.

\noindent(b)  Suppose that $0\leq f\in A^\loc(X)$ and that $f(x) = 0$ for an accumulation point $x$ of $X$.
There exist an open neighborhood $U$ of $x$ and $f'\in A(X)$ such that $f = f'$ on $U$.  Since $A(X)$ is a vector lattice, we may replace $f'$ by $f'\vee0$ if necessary to assume that $f' \geq 0$.  Because $A(X)$ satisfies condition ($\heartsuit$), there exists $g'\in A(X)$ such that $x \in \ol{\{f' < g'\}}\cap \ol{\{g' < 0\}}$.
By continuity, $0 = f'(x) \leq g'(x) \leq 0$; hence $g'(x) = 0$.  Thus $V = \{|g'| < 1\}\cap U$ is an open neighborhood of $x$.  Let $g = (g'\wedge 1)\vee -1$.  Then $g \in A^\loc_b(X)$,
\[ \{f < g\}  \cap V  = \{f' < g'\}\cap V \text{ and }  \{g < 0\}\cap V =  \{g' < 0\} \cap V.\]
Therefore, $x \in  \ol{\{f < g\}}\cap \ol{\{g < 0\}}$.  This shows that $A^\loc(X)$ and $A^\loc_b(X)$ satisfy property ($\heartsuit$).

\noindent(c) 
Let $0 \leq f\in A(X)$ and let $x_0$ be an accumulation point in $X$ where $f(x_0) = 0$.
Choose a sequence $(x_n)$ in $X$ converging to $x_0$ such that $0 < d(x_{n+1},x_0) < d(x_n, x_0)/5$ for all $n$.
Set $r_n = d(x_n,x_0)$.
Define $g_n: X\to \R$ by 
\[ g_n(x) = \begin{cases}
             (f(x_n)+r_n)[1-\frac{2}{r_n}d(x,x_n)]^+ &\text{if $n$ is even},\\
             -[\frac{r_n}{2}-d(x,x_n)]^+ &\text{if $n$ is odd}.
            \end{cases}\]
The functions $g_n$ are disjointly supported functions in $\Lip(X)$; the Lipschitz constant of $g_n$ is at most $2(\frac{f(x_n)}{r_n}+1)$ and $\|g_n\|_\infty \leq f(x_n) + r_n$.
It follows that $g = \sum g_n$ converges uniformly on $X$ and thus $g \in U(X) \subseteq C(X)$.

Observe that if $m < n$, $d(y,x_n) < r_n/2$ and $d(z,x_m)< r_m/2$, then 
\begin{equation}\label{eq2} 
d(y,z) \geq d(z,x_0) - d(y,x_0) \geq \frac{r_m}{2} - \frac{3r_n}{2} > \frac{r_m}{2} - \frac{3r_m}{10} = \frac{r_m}{5}
\end{equation}
and 
\begin{equation}\label{eq3} 
|g(y)-g(z)|\leq |g(y)|+|g(z)| \leq f(x_n)+{r_n} + f(x_m) + {r_m}.
\end{equation}

If $A(X) = \Lip(X)$, then there is a constant $K$ such that 
\[f(x_n)= f(x_n)-f(x_0) \leq Kd(x_n,x_0) = Kr_n \text{ for all $n$.}\]
By (\ref{eq2}), (\ref{eq3}) and the fact that each $g_n$ is Lipschitz with constant at most $2(K+1)$, one can conclude that $g\in \Lip(X)$.

If $A(X) = \lip_\al(X)$, $0< \al < 1$, then 
\begin{equation}\label{eq6} 
\lim_n\frac{f(x_n)}{r_n^\al} = \lim_n\frac{f(x_n)-f(x_0)}{d(x_n,x_0)^\al} = 0.
\end{equation}
As observed above,  $g_n$ is Lipschitz with respect to the metric $d$ with constant at most $2(\frac{f(x_n)}{r_n}+1)$.
Let $y,z\in X$. If $y, z \notin \cup B(x_n,\frac{r_n}{2})$, then $g(y) = g(z) = 0$.
Suppose there exists $n$ such that $y\in B(x_n,\frac{r_n}{2})$ and $z \notin \cup_{m\neq n}B(x_m,\frac{r_m}{2})$.  Then $g(y) = g_n(y)$ and $g(z) = g_n(z)$. Thus
\begin{align}\label{eq4}
\frac{|g(y)-g(z)|}{d(y,z)^\al} &\leq 2(\frac{f(x_n)}{r_n}+1)d(y,z)^{1-\al}\\ \notag
&= 2\bigl(\frac{f(x_n) + r_n}{r^\al_n}\bigr)\bigl(\frac{d(y,z)}{r_n}\bigr)^{1-\al}.
\end{align}
Also, since $\|g_n\|_\infty \leq f(x_n)+r_n$, 
\begin{equation}\label{eq5}
\frac{|g(y)-g(z)|}{d(y,z)^\al} \leq \frac{2(f(x_n)+r_n)}{d(y,z)^\al}.
\end{equation}
If $d(y,z) \leq r_n$, use the estimate (\ref{eq4}); while if $d(y,z) > r_n$, employ the estimate (\ref{eq5}). 
In either case, keeping  (\ref{eq6}) in mind, we can conclude that 
\[ \frac{|g(y)-g(z)|}{d(y,z)^\al} \to 0 \text{ as $d(y,z) \to 0$}\]
and that the expression is bounded independent of $y, z$ and $n$.

On the other hand, assume that $y \in B(x_m,\frac{r_m}{2})$ and $z \in B(x_n,\frac{r_n}{2})$, where $m< n$.  Then $g(y) = g_n(y)$ and $g(z) = g_m(z)$.
By (\ref{eq2}) and (\ref{eq3}),
\[  \frac{|g(y)-g(z)|}{d(y,z)^\al} \leq \frac{ f(x_n)+{r_n} + f(x_m) + {r_m}}{(\frac{r_m}{5})^\al}.\]
In particular, the expression is bounded independent of $y, z, m$ and $n$. Furthermore, since  $d(y,z)\to 0$ implies $m \to \infty$, one can deduce that 
\[ \lim_{d(y,z)\to 0}\frac{|g(y)-g(z)|}{d(y,z)^\al} = 0.\]
Therefore, $g \in \lip_\al(X)$.

Finally, $g(x_n) > f(x_n)$ if $n$ is even, $g(x_n) < 0$ if $n$ is odd, and $(x_n)$ converges to $x_0$. Hence $x_0 \in \ol{\{f < g\}}\cap \ol{\{g < 0\}}$.

\noindent (d), (e) and (f) (Refer to \cite[Step 1.5 on p.293]{CC2}.) 
Suppose that $0 \leq f\in A(X)$, respectively, $A(\ol{X})$,  and $f(x_0) =0$ for some $x_0 \in X$.
Then $Df(x_0) = 0$.  Let $E$ be the ambient Banach space containing $X$. Choose a nonzero $x^*\in E^*$ and define $h: X\to \R$ by $h(x) = x^*(x-x_0)$. Then $h \in C^p(X)$. Since $Df(x_0) = 0$, $x_0 \in \ol{\{f < h\}}\cap \ol{\{h < 0\}}$.  Let $\gamma: \R\to \R$ be a $C^\infty$ function such that $\gamma(t) = t$ if $|t| \leq 1$ and $\gamma(t) = 0$ if $|t| > 2$.
Clearly $\gamma^{(k)}$ is bounded on $\R$ for any $k\in \N\cup\{0\}$.
In particular,  $g = \gamma\circ h \in C^p_b(X)$.
Furthermore, $\|D^kg(x)\| \leq |\gamma^{(k)}(h(x))| \|x^*\|^k$
for all $x\in X$ and all $k$ with $1\leq k \leq p$ ($1\leq k < \infty$ if $p = \infty$)
 and $g$ has a continuous extension $\ol{g}$ onto $\ol{X}$.
Thus $\ol{g}\in C^p_*(\ol{X}) \subseteq C^p_b(\ol{X}) \subseteq C^p(\ol{X})$.
Since $g = h$ on a neighborhood of $x_0$, it is clear that $x_0 \in \ol{\{f < g\}}\cap \ol{\{g < 0\}}$. 
This proves that $C^p(X), C^p_b(X)$ satisfy condition ($\heartsuit$), and that $C^p(\ol{X}), C^p_b(\ol{X})$, and $C^p_*(\ol{X})$ satisfy condition ($\heartsuit_x$) for all $x\in X$.

It remains to prove that $C^p(\ol{X})$ and $C^p_b(\ol{X})$ satisfy condition ($\heartsuit_x$) for any $x\in \ol{X}\bs X$ whenever these spaces separate points from closed sets.
Let $x_0 \in \ol{X} \bs X$ and assume that $f\in C^p(\ol{X})$, $f\geq 0$, and $f(x_0) = 0$.
Choose a sequence of distinct points $(x_n)$ in $X$ that converges to $x_0$.
There are open neighborhoods $U_n$ of $x_n$, $U_n \subseteq X$, $\diam U_n \to 0$ so that $\ol{U_n} \cap \ol{\cup_{m\neq n}U_m} = \emptyset$ for all $n$ (where the closures are taken in $\ol{X}$).
Denote by $(E,\|\cdot\|)$ the Banach space containing $X$. Since $C^p(\ol{X})$ separates points from closed sets, there exists $h\in C^p(E)$ such that $h(0) =1 $ and $h(x)=0$ if $\|x\| \geq 1$.
We may assume that $f$ is bounded; otherwise, replace $f$ with $\gamma\circ f$, where $\gamma$ is as in the previous paragraph.
Choose $\ep_n > 0$ so that $\|x-x_n\|< \ep_n$ implies $x\in U_n$.
Define $g:\ol{X}\to\R$ by 
\[ g(x) = \begin{cases}
             f(x_{2n-1})+ \frac{1}{n}h(\frac{x-x_{2n-1}}{\ep_{2n-1}}) &\text{if $\|x-x_{2n-1}\| < \ep_{2n-1}$},\\
             -\frac{1}{n}h(\frac{x-x_{2n}}{\ep_{2n}}) &\text{if $\|x-x_{2n}\| < \ep_{2n}$},\\
            0 &\text{otherwise}.
            \end{cases}\]
Since $(f(x_n))$ converges to $0$, it is easy to see that $g$ is a function in $C^p_b(\ol{X})$.
Moreover, $x_{2n-1} \in \{f < g\}$ and $x_{2n} \in \{g  < 0\}$ for all $n$.
Thus $x_0 \in \ol{\{f< g\}} \cap\ol{\{g < 0\}}$.
\end{proof}

In general, the space $C^p_*(\ol{X})$ may not satisfy condition ($\heartsuit$).
\medskip

\noindent{\bf Example D}.  Let $X = (0,1) \subseteq \R$. The space $C^2_*(\ol{X})$ fails condition ($\heartsuit_0$).

\begin{proof}
Consider the function $f:[0,1]\to \R$, $f(x) = x$.  Clearly, $f\in C^2_*(\ol{X})$. Assume that there exists $g\in C^2_*(\ol{X})$ such that $0 \in \ol{\{f< g\}} \cap \ol{\{g< 0\}}$. Then there exists a sequence $(x_n)$ in $(0,1)$ strictly decreasing to $0$ such that $x_{2n-1} = f(x_{2n-1}) < g(x_{2n-1})$ and $g(x_{2n}) < 0$ for all $n$. By the Mean Value Theorem, there exists $t_n \in (x_{n+1}, x_{n})$ such that 
\[ g'(t_n) = \frac{g(x_n) - g(x_{n+1})}{x_n-x_{n+1}} 
\begin{cases}
> \frac{x_n}{x_n - x_{n+1}} > 1 & \text{if $n$ is odd},\\
< 0 & \text{if $n$ is even}.
\end{cases}\]
By the Mean Value Theorem again, there exists $s_n \in (t_{n+1},t_n)$ such that 
\[ |g''(s_n)| = \bigl|\frac{g'(t_n)-g'(t_{n+1})}{t_n-t_{n+1}}\bigr| \geq \frac{1}{t_n -t_{n+1}} \to \infty.\] 
This contradicts the assumption that $g''$ is bounded on $(0,1)$.
\end{proof}

\section{More on spaces of Lipschitz functions}\label{sec5}

The main result in \cite{CC1} shows that for a pair of complete metric spaces $X$ and $Y$ with finite diameters, the spaces $\Lip(X)$ and $\Lip(Y)$ are order isomorphic if and only if $X$ and $Y$ are Lipschitz homeomorphic.
In this section, we show that for any complete metric space $X$, $\Lip(X)$ is (linearly) order isomorphic to some $\Lip(X')$, where $X'$ is a complete metric space with finite diameter.  This result is a close relative of  Proposition 1.7.5 in \cite{W}.  We include complete proofs since the statements are slightly different and some of the estimates obtained in the proof will be useful subsequently.
The result may be exploited to give a  characterization of pairs of complete metric spaces which support  order isomorphic spaces of Lipschitz functions.  

Let $(X,d)$ be a  metric space with a distinguished point $e$.  Given $f\in \Lip(X)$, let
\[ L(f) = \sup_{p\neq q}\frac{|f(p)-f(q)|}{d(p,q)}\]
be its {\em Lipschitz constant}.
Let $\xi \in \Lip(X)$ be the function $\xi(x) = d(x,e)\vee 1$.
Define another metric $d'$ on $X$ by 
\begin{equation}\label{eq5.1} 
d'(p,q) = \sup_{\substack{f\in \Lip(X)\\L(f), |f(e)| \leq 1}}\bigl|\frac{f(p)}{\xi(p)}- \frac{f(q)}{\xi(q)}\bigr|. 
\end{equation}
We summarize some of the properties of the metric $d'$ in the next proposition.

\begin{prop}\label{prop5.1}
\begin{enumerate}
\item $d'$ is a bounded metric on $X$. In fact, $d'(p,q) \leq 4$ for all $p,q\in X$.  
\item For any $p,q\in X$, let
\[ \rho(p,q) = \frac{d(p,q)}{\xi(p) \vee\xi(q)}.\]
Then
\[ \rho(p,q) \leq d'(p,q) \leq 3\rho(p,q)\]
for all $p,q \in X$.
\item If $p,q \in X$ and $\xi(p)\leq \xi(q)$, then 
\begin{equation*}\label{eq5.1.1}
d'(p,q) \leq d'(p,q)\xi(p) \leq 3d(p,q).
\end{equation*}
\item If $X$ is complete with respect to $d$, then $X$ is complete with respect to $d'$.
\end{enumerate}
\end{prop}

\begin{proof}
(a) Suppose that $f\in \Lip(X)$, $L(f), |f(e)| \leq 1$. 
For any $p, q \in X$, 
\[|f(p)|  \leq |f(e)|+ |f(p) - f(e)| \leq 1 + d(p,e) \leq 2 \xi(p).\]
Hence 
\[ \bigl|\frac{f(p)}{\xi(p)}- \frac{f(q)}{\xi(q)}\bigr| \leq 4.\]
Thus $d'(p,q) \leq 4$ for all $p,q$.
It is clear that $d'$ is a metric on $X$.

\noindent (b) Fix $p, q\in X$.  We may assume that $\xi(p) \leq \xi(q)$.
Define $f:\{e,p,q\}\to \R$ by $f(e) = 0$, $f(p) = d(p,e)$ and 
\[ f(q) = d(p,e) - d(p,q).\]
Note that the definition is consistent even if some of the points $e, p ,q$ coincide.
Furthermore, $f$ is a Lipschitz function with respect to $d$ with Lipschitz constant at most $1$. 
Hence $f$ extends to a function in $\Lip(X,d)$, still denoted by $f$,  with Lipschitz constant at most $1$, see e.g. \cite[Theorem 1.5.6]{W}. Obviously $|f(e)| = 0 \leq 1$.
By definition of $d'$, we have
\[d'(p,q) \geq  \bigl|\frac{f(p)}{\xi(p)}- \frac{f(q)}{\xi(q)}\bigr| = f(p)\bigl(\frac{1}{\xi(p)}-\frac{1}{\xi(q)} \bigr)+ \frac{d(p,q)}{\xi(q)}\geq \rho(p,q). \]
On the other hand, 
consider $f\in \Lip(X)$ with $L(f) \leq 1$ and $|f(e)|\leq 1$.
Then 
\[ |f(p)| \leq |f(p) - f(e)| + |f(e)| \leq d(p,e) +1\leq 2 \xi(p).\]
Hence
\begin{align*}
\bigl|\frac{f(p)}{\xi(p)}- \frac{f(q)}{\xi(q)}\bigr| & = |f(p)\bigl(\frac{1}{\xi(p)}-\frac{1}{\xi(q)} \bigr)+ \frac{f(p) - f(q)}{\xi(q)}| \\& \leq 2\xi(p)\bigl(\frac{1}{\xi(p)}-\frac{1}{\xi(q)} \bigr)+ \frac{d(p,q)}{\xi(q)}\\
& = \frac{1}{\xi(q)}[2(\xi(q) - \xi(p)) + d(p,q)]\\
& \leq 3 \frac{d(p,q)}{\xi(q)}
=3\rho(p,q).
\end{align*}
Taking supremum over all $f\in \Lip(X)$ with $L(f), |f(e)| \leq 1$ gives $d'(p,q) \leq 3\rho(p,q)$.

\noindent(c) 
The first half of the inequality  is obvious since $\xi(p) \geq 1$.
Suppose that $\xi(p) \leq\xi(q)$.  By (b), 
\[ d'(p,q)\xi(p) \leq 3\rho(p,q)\xi(p) \leq 3 d(p,q).\]

\noindent (d) Assume that $(X,d)$ is complete.
Suppose that $(x_n)$ is a $d'$-Cauchy sequence in $X$.  
If  $(\xi(x_n))$ is unbounded,
by taking a subsequence if necessary, we may assume that $\xi(x_{n+1}) > 2\xi(x_n)$ for all $n$. In particular, $\xi(x_n) =d(x_n,e)$ if $n > 1$.
If $m < n$, then
\[ d(x_n,x_m) \geq d(x_n,e) -d(x_m,e) \geq \xi(x_n) - \xi(x_m) \geq \frac{1}{2} \xi(x_n).\]
Hence 
\[d'(x_n,x_m) \geq \rho(x_n,x_m) = \frac{d(x_n,x_m)}{\xi(x_n)} \geq \frac{1}{2}.\]
This contradicts the fact that $(x_n)$ is $d'$-Cauchy. 
Therefore, $(\xi(x_n)$) is bounded. By (b), $d(x_m,x_n)\leq Cd'(x_m,x_n)$ for some constant $C<\infty$ and hence $(x_n)$ is $d$-Cauchy. Let $x_0$ be the limit of $(x_n)$ with respect to $d$. By (c),
$d'(x_n,x_0) \leq {3d(x_n,x_0)} \to 0$.  Thus $(x_n)$ converges to $x_0$ with respect to $d'$.
\end{proof}

\begin{prop}\label{prop5.2}
Suppose that $f$ is a real-valued function on $X$.  Then $f\in \Lip(X,d)$ if and only if $f/\xi\in \Lip(X,d')$.
\end{prop}

\begin{proof}
Suppose that $f\in \Lip(X,d)$. Set $c = L(f) \vee |f(e)| \vee 1$.
Let $g = f/c$.  Then $L(g), |g(e)| \leq 1$.
For any $p,q\in X$, $p\neq q$,
\[cd'(p,q) \geq c \bigl|\frac{g(p)}{\xi(p)}- \frac{g(q)}{\xi(q)}\bigr|
 =  \bigl|\frac{f(p)}{\xi(p)}- \frac{f(q)}{\xi(q)}\bigr|. \]
This shows that $f/\xi\in \Lip(X,d')$. Moreover, $L'(f/\xi)\leq L(f)\vee|f(e)|\vee 1$, where $L'(g)$ denotes the Lipschitz constant of $g$ with respect to $d'$.

Suppose that $g = f/\xi \in \Lip(X,d')$ and let $p, q$ be distinct points in $X$.
We may assume that $\xi(p) \leq \xi(q)$.
 By Proposition \ref{prop5.1}(a), $d' \leq 4$.  Hence
\[ |g(q)| \leq |g(q) - g(e)| + |g(e)| \leq L'(g)d'(q,e) +|g(e)|\leq 4L'(g) + |g(e)|.\]
Then
\begin{align*}
|f(p) - f(q)| &\leq |g(p)-g(q)|\xi(p) + |g(q)|(\xi(q) - \xi(p))\\
&\leq L'(g)d'(p,q)\xi(p) +(4L'(g) + |g(e)|)d(p,q)\\
&\leq (7L'(g) + |g(e)|)d(p,q),
\end{align*}
where we have used Proposition \ref{prop5.1}(c) in the last inequality.
This proves that $f\in \Lip(X,d)$.
\end{proof}

The results of Propositions \ref{prop5.1} and \ref{prop5.2} can be summarized as follows.

\begin{thm}\label{thm5.3}
Let $(X,d)$ be a complete metric space with a distinguished point $e$.  Let $X'$ be the metric space $(X,d')$, where $d'$ is given by equation (\ref{eq5.1}).  Then $X'$ is a complete metric space of  finite diameter and $\Lip(X)$ is linearly order isomorphic to $\Lip(X') = \Lip_b(X')$.
\end{thm}

We can now extend the characterization of order isomorphisms between spaces of Lipschitz functions defined on metric spaces with finite diameter \cite{CC1} to general metric spaces.
First we recall 

\begin{thm}\label{thm5.5}
\cite[Theorem 1]{CC1} Let $X$ and $Y$ be complete metric spaces with finite diameter.
If $T: \Lip(X)\to \Lip(Y)$ is an order isomorphism, then there are a Lipschitz homeomorphism $\vp: X\to Y$ and a function $\Phi:Y\times \R\to \R$ such that $\Phi(y,\cdot):\R\to \R$ is  an increasing homeomorphism for all $y \in Y$, and that $Tf(y) = \Phi(y,f(\vp^{-1}(y)))$ for all $f\in \Lip(X)$ and all $y \in Y$.
\end{thm}

\begin{thm}\label{thm5.6}
Let $(X,d_X)$ and $(Y,d_Y)$ be complete metric spaces with distinguished points $e_X$ and $e_Y$ respectively. Define $\xi(x) = 1\vee d_X(x,e_X)$ and $\zeta(y) = 1\vee d_Y(y,e_Y)$.
Then $\Lip(X)$ is order isomorphic to $\Lip(Y)$ if and only if there are a homeomorphism $\vp :X\to Y$ and  a finite constant $C >0$ such that
\begin{equation}\label{eq5.2} 
\frac{1}{C}\rho_X(p,q) \leq 
\rho_Y(\vp(p),\vp(q)) \leq C\rho_X(p,q)
\end{equation}
for all $p,q\in X$, where 
\[ \rho_X(p,q) = \frac{d_X(p,q)}{\xi(p) \vee\xi(q)}\]
 and $\rho_Y$ is defined similarly.
\end{thm}

\begin{proof}
Let $X'$ and $Y'$ be the respective metric spaces $X$ and $Y$ endowed with the metrics $d'_X$ and $d'_Y$ given by 
\begin{align*} 
d'_X(p,q) &= \sup_{\substack{f\in \Lip(X)\\L(f), |f(e_X)| \leq 1}}\bigl|\frac{f(p)}{\xi(p)}- \frac{f(q)}{\xi(q)}\bigr|,\\
d'_Y(u,v) &= \sup_{\substack{f\in \Lip(Y)\\L(f), |f(e_Y)| \leq 1}}\bigl|\frac{f(u)}{\zeta(u)}- \frac{f(v)}{\zeta(v)}\bigr|. 
\end{align*}
By Proposition \ref{prop5.1}(a), $X'$ and $Y'$ have finite diameter.
Let $T:\Lip(X)\to \Lip(Y)$ be an order isomorphism.
It is clear from Proposition \ref{prop5.2} that the map $\ti{T}:\Lip(X')\to \Lip(Y')$, $\ti{T}f = T(\xi f)/\zeta$ is an order isomorphism.
By Theorem \ref{thm5.5}, there is a Lipschitz homeomorphism $\vp: X'\to Y'$.
Inequality (\ref{eq5.2}) follows from Proposition \ref{prop5.1}(b).

Conversely, given (\ref{eq5.2}), 
$(X,d_X')$ and $(Y,d_Y')$ are Lipschitz homeomorphic by Proposition \ref{prop5.1}(b).
Hence $\Lip(X')$ and $\Lip(Y')$ are order isomorphic.
By Theorem \ref{thm5.3}, $\Lip(X)$ and $\Lip(Y)$ are order isomorphic. 
\end{proof}

Proposition \ref{prop5.2} can be extended to little Lipschitz spaces.  
Let $(X,d)$ be a complete metric space with a distinguished point $e$.
Say that $X$ is   {\em almost expansive at $\infty$} if for all $\ep >0$, there exists $C<\infty$ such that $d(p,q)< \ep$ if $d(p,e) \geq C$ and $d(p,q)< d(p,e)/C$.
For the remainder of the section, let $(X,d)$ be a complete metric space with a distinguished point $e$ that is almost expansive at $\infty$.
Choose $1 \leq C_1 < C_2 < \cdots$ such that $d(p,q)< \frac{1}{k+2}$ if $d(p,e) \geq C_k$ and $d(p,q)< d(p,e)/C_k$.
If $x\in X$ and $0 \leq r_1 < r_2$, let 
\[ \Ann(x,r_1,r_2) = \{z\in X: r_1 < d(z,x) < r_2\}.\]

\begin{lem}\label{lem5.7}
Suppose that $d(p,e) \geq C_k$.  Then $\Ann(p,\frac{1}{k+2}, d(p,e)/C_1) = \emptyset$.
\end{lem}

\begin{proof}
Suppose that $q\in X$ and $d(p,q) \geq \frac{1}{k+2}$.
By choice of $C_k$, $d(p,q) \geq d(p,e)/C_k\geq 1/3$.
By choice of $C_1$, $d(p,q) \geq d(p,e)/C_1$.
\end{proof}

Let $\Gamma$ be a subset of $X\bs B(e,C_1)$ that is maximal with respect to the condition that $B(p,1) \cap B(q,1) = \emptyset$ if $p$ and $q$ are distinct points in $\Gamma$.

\begin{lem}\label{lem5.8}
The set $X_0 = \cup_{p\in \Gamma}B(p,1)$ is both open and closed in $X$, and  $X=  B(e,C_1) \cup  X_0$.
\end{lem}

\begin{proof}
Clearly $X_0$ is an open set.
Let $(x_n)$ be a sequence in $X_0$ converging to some $x_0\in X$.
Choose $p_n \in \Gamma$ such that $x_n \in B(p_n,1)$. By Lemma \ref{lem5.7}, $d(x_n,p_n) \leq 1/3$.
If $p_n \neq p_m$, then 
\[ d(x_n,x_m) \geq d(p_n,p_m) -d(x_n,p_n) -d(x_m,p_m) \geq 1- \frac{1}{3} - \frac{1}{3} = \frac{1}{3}.\]
Thus, there exists $p\in \Gamma$ such that $p_n = p$ for all sufficiently large $n$.
By Lemma \ref{lem5.7}, $B(p,1) = \{x\in X: d(x,p) \leq 1/3\}$ is closed in $X$.
Hence $x_0\in B(p,1)\subseteq X_0$.  This proves that $X_0$ is closed.

If there exists $q \notin B(e,C_1)$ such that $q \notin X_0$, by the maximality of $\Gamma$, there exists $p\in \Gamma$ such that $B(p,1) \cap B(q,1) \neq \emptyset$.
Let $x\in B(p,1) \cap B(q,1)$. By Lemma \ref{lem5.7}, $d(x,p), d(x,q) \leq 1/3$.
Hence $q \in B(p,1)\subseteq X_0$, contradicting the choice of $q$.
\end{proof}

As above, define $\xi:X\to \R$ by $\xi(x) = d(x,e)\vee 1$.  
Let $\zeta:X\to \R$ be the function given by
\[\zeta(x) =\begin{cases}
                 \xi(p) &\text{if $x\in B(p,1)$ for some  $p \in \Gamma$},\\
                 1 &\text{if $x\notin X_0$.}
                 \end{cases}\]
Let $X'$ be the metric space $(X,d')$, with  the metric $d'$ given in (\ref{eq5.1}).

\begin{lem}\label{lem5.9}
There exists $1 \leq K <\infty$ such that  for all $x\in X$, 
\[ \frac{1}{K}\xi(x) \leq \zeta(x) \leq K\xi(x).\]
\end{lem}

\begin{proof}
If $x\notin X_0$, then $d(x,e) \leq C_1$ by Lemma \ref{lem5.8}. Thus $\xi(x) \leq C_1$.  Hence
\[ \frac{\xi(x)}{C_1} \leq \zeta(x) = 1 \leq \xi(x).\]
If $x\in B(p,1)$ for some $p\in \Gamma$, then $d(x,p) \leq 1/3$ by Lemma \ref{lem5.7}.
Thus 
\[ |\zeta(x) -\xi(x)| = |\xi(p) -\xi(x)| \leq d(x,p) \leq \frac{1}{3} \leq \frac{\xi(x)}{3}.\]
The lemma holds if we take $K \geq 4/3$ so that $K^{-1} \leq C_1^{-1}\wedge 2/3$.
\end{proof}

For the rest of the section, $K$ will denote a  constant  satisfying Lemma \ref{lem5.9} such that $K \geq 2$.

\begin{lem}\label{lem5.10}
Suppose that $u\in B(p,1)$ for some $p\in \Gamma$ and $v\notin B(p,1)$.
Then $d'(u,v) \geq 2(3KC_1)^{-1}$.
\end{lem}

\begin{proof}
If $v\notin X_0$, then $d(v,e) \leq C_1$ by Lemma \ref{lem5.8}. Hence $\xi(v) \leq C_1\leq \xi(p)$.
Also, 
\[\xi(u) \leq \xi(p) +d(u,p) \leq \xi(p) +1 \leq 2\xi(p).\]
Hence $\xi(u) \vee \xi(v) \leq 2\xi(p)\leq K\xi(p)$.
If $v\in X_0$, there exists $q\in \Gamma$, $q\neq p$ such that $v\in B(q,1)$.
Without loss of generality, assume that $\xi(p) = d(p,e) \geq d(q,e) = \xi(q)$.
By Lemma \ref{lem5.9}, $\xi(u) \leq K\zeta(u) = K\xi(p)$ and $\xi(v) \leq K\xi(q)$.
Hence $\xi(u)\vee \xi(v) \leq K\xi(p)$.
By Lemma \ref{lem5.7}, $d(v,p) \geq \xi(p)/C_1$ and $d(u,p) \leq 1/3$.
Thus 
\[ d(u,v) \geq d(v,p) - d(u,p)\geq \frac{\xi(p)}{C_1} - \frac{1}{3} \geq \frac{2\xi(p)}{3C_1}.\]
Therefore, by Proposition \ref{prop5.1}(b),
\[ d'(u,v) \geq  \frac{d(u,v)}{\xi(u)\vee \xi(v)} \geq \frac{d(u,v)}{K\xi(p)} \geq \frac{2}{3KC_1}.\] 
\end{proof}

\begin{lem}\label{lem5.11}
The functions $\xi/\zeta$ and $\zeta/\xi$ are Lipschitz with respect to the metric $d'$ on $X$.
\end{lem}

\begin{proof}
Let $u$ and $v$ be distinct points in $X$.  Consider three cases.

\medskip

\noindent\underline{Case 1}.  $u, v\notin X_0$.

\noindent In this case, $\xi(u) \vee \xi(v) \leq C_1$ by Lemma \ref{lem5.8}.
Proposition \ref{prop5.1}(b) implies that
$d'(u,v) \geq d(u,v)/C_1$.
Then
\[ \bigl|\frac{\xi}{\zeta}(u) - \frac{\xi}{\zeta}(v)\bigr| = |\xi(u) - \xi(v)|\leq d(u,v) \leq C_1d'(u,v).\]

\noindent\underline{Case 2}. $u \in B(p,1)$ for some $p\in \Gamma$ and $v \notin B(p,1)$.

\noindent By Lemma \ref{lem5.10}, $d'(u,v) \geq 2(3KC_1)^{-1}$.
By Lemma \ref{lem5.9},
\[ \bigl|\frac{\xi}{\zeta}(u) - \frac{\xi}{\zeta}(v)\bigr| \leq \bigl|\frac{\xi}{\zeta}(u)\bigr| +  \bigl|\frac{\xi}{\zeta}(v)\bigr| \leq 2K \leq 3K^2C_1d'(u,v).\]

\noindent\underline{Case 3}. There exists $p\in \Gamma$ such that $u,v \in B(p,1)$.

\noindent In this case, $\xi(u) \leq \xi(p) + d(u,p) \leq \xi(p) + 1 \leq 2\xi(p)$.
Similarly $\xi(v) \leq 2\xi(p)$.
By Proposition \ref{prop5.1}(b),
\[ \bigl|\frac{\xi}{\zeta}(u) - \frac{\xi}{\zeta}(v)\bigr| = \frac{|\xi(u)-\xi(v)|}{\xi(p)} \leq \frac{d(u,v)}{\xi(p)}
\leq \frac{2d(u,v)}{\xi(u)\vee \xi(v)} \leq 2d'(u,v).\]
This completes the proof that $\xi/\zeta$ is Lipschitz with respect to $d'$.  Since $\xi/\zeta$ is also bounded below by $1/K$, it is routine to check that its reciprocal $\zeta/\xi$ is also Lipschitz with respect to $d'$.
\end{proof}

\begin{prop}\label{prop5.12}
Let $(X,d)$ be a complete metric space with a distinguished point $e$ that is almost expansive at $\infty$. 
Take $X' =(X,d')$, where the metric $d'$ is given by equation (\ref{eq5.1}).  Define $\zeta: X\to \R$ as above.
Then $T: \lip(X)\to \lip(X')$, $Tf = f/\zeta$, is an order isomorphism.
\end{prop}

\begin{proof}
Use the notation developed from just prior to Lemma \ref{lem5.7}.
By Proposition \ref{prop5.2},   $f\in \Lip(X)$ if and only if $f/\xi\in \Lip(X')$.
Suppose that $f\in \Lip(X)$, then $f/\xi, \xi/\zeta \in \Lip(X')$, where the latter follows from Lemma \ref{lem5.11}. Since $X'$ has finite diameter, the product $f/\zeta\in \Lip(X')$.
Conversely, if $f/\zeta\in \Lip(X')$, then from Lemma \ref{lem5.11} again, $f/\xi = (f/\zeta)\cdot(\zeta/\xi) \in \Lip(X')$.
Hence $f\in \Lip(X)$.

Suppose that $f\in \lip(X)$. In particular, $f \in \Lip(X)$ and hence $f/\zeta\in \Lip(X')$ by the last paragraph. Given $\ep > 0$, choose $\delta > 0$ such that $|f(u) - f(v)| <\ep d(u,v)$ if $d(u,v) < \delta$.  Fix $k\in \N$ such that $2(k+2)^{-1} < \delta$.
Consider $u,v$ such that $d'(u,v) < 2(3KC_1)^{-1} \wedge \delta/(2C_k)$.
By Lemma \ref{lem5.10}, either $u,v\notin X_0$ or there exists $p\in \Gamma$ such that $u,v\in B(p,1)$.
In the first case, $\zeta(u) = \zeta(v) =1$ and $\xi(u), \xi(v) \leq C_1$ by Lemma \ref{lem5.8}.
By Proposition \ref{prop5.1}(b), $d'(u,v) \geq d(u,v)/C_1$.  In particular, $d(u,v) < \delta$.
Hence
\[ \bigl|\frac{f}{\zeta}(u) - \frac{f}{\zeta}(v)\bigr| = |f(u) -f(v)| < \ep d(u,v) \leq C_1\ep d'(u,v).\]
In the latter case,
$\xi(u),\xi(v) \leq 2\xi(p)$ from the proof of Case 3 in Lemma \ref{lem5.11}.
By Proposition \ref{prop5.1}(b),
\[d'(u,v) \geq \frac{d(u,v)}{\xi(u)\vee\xi(v)} \geq \frac{d(u,v)}{2\xi(p)}.\]
If $d(p,e) \geq C_k$, then $d(u,v) < 2(k+2)^{-1} < \delta$ by Lemma \ref{lem5.7}.
On the other hand, if $d(p,e) < C_k$, then $\xi(p) <  C_k$.
Hence $d'(u,v) \geq d(u,v)/(2C_k)$.
Therefore, $d(u,v) \leq 2C_kd'(u,v) < \delta$.
In either situation, we have
\[  \bigl|\frac{f}{\zeta}(u) - \frac{f}{\zeta}(v)\bigr| = \frac{|f(u) -f(v)|}{\xi(p)} < \frac{\ep d(u,v)}{\xi(p)} \leq 2\ep d'(u,v). \]
This completes the proof that $f/\zeta \in \lip(X')$ if $f\in \lip(X)$.

Conversely, suppose that $g = f/\zeta\in \lip(X')$. By the first paragraph, $f\in \Lip(X)$. 
Given $\ep > 0$, choose $\delta >0$ so that $|g(u)- g(v)|< \ep d'(u,v)$ if $d'(u,v) < \delta$.
Consider $u,v\in X$ such that $d(u,v) < \delta/3 \wedge 2(9KC_1)^{-1}$.
By Proposition \ref{prop5.1}(c), $d'(u,v) \leq 3d(u,v) < \delta \wedge 2(3KC_1)^{-1}$.
In particular, Lemma  \ref{lem5.10} implies that  either $u,v\notin X_0$ or there exists $p\in \Gamma$ such that $u,v\in B(p,1)$.
In the first case, $\zeta(u) = \zeta(v) =1$. Thus
\[ |f(u)-f(v)| = |g(u) - g(v)| < \ep d'(u,v) \leq 3\ep d(u,v).\]
In the latter case, we may assume that $\xi(u) \leq \xi(v)$.  Then $\zeta(u) = \zeta(v)  =\xi(p)$ and $\xi(p) \leq 2\xi(u)$ by the proof of Case 3 in Lemma \ref{lem5.11}.  Hence
\[ |f(u) -f(v)| = \xi(p)|g(u)-g(v)| < 2\xi(u)\ep d'(u,v) \leq 6\ep d(u,v),
\]
where the last step follows again from Proposition \ref{prop5.1}(c).
This proves that $f\in \lip(X)$ if $f/\zeta\in \lip(X')$.
\end{proof}

\noindent{\bf Remark}. Note that $X'$ has finite diameter and hence $\lip(X) \sim \lip(X') = \lip_b(X')$ if $X$ is almost expansive at $\infty$.  A strong converse to Proposition \ref{prop5.12} for H\"{o}lder metric spaces will be shown below.  See Theorem \ref{thm66}.

\section{Comparing function spaces under order isomorphism}

We have seen in Corollary \ref{cor26} that if $X$ is a  noncompact realcompact space, then $C(X)$ is never order isomorphic to any space of the type $C_b(Y)$.
This serves as a prototype of the sort of results to be considered in this section.  
Precisely, we seek to determine conditions under which two different  spaces among the ones listed in Examples B or Examples C can be order isomorphic.

\subsection{General principles}
We begin by listing several general principles before going into specific cases.

\begin{prop}\label{prop31}
Let $A(X)$ and $A(Y)$ be near vector lattices defined on metric spaces $X$ and $Y$ respectively, where $A(X)  = A^\loc(X)$ and $A(Y) = A^\loc_b(Y)$, or $A(Y)= A_b(Y)$ and  satisfies ($\spadesuit$).
If $X$ is not compact, then $A(X)$ and $A(Y)$ are not order isomorphic.
\end{prop}

\begin{proof}
Let $T: A(X)\to A(Y)$ be an order isomorphism.  Obtain a homeomorphism $\vp:X\to Y$ as in Theorem \ref{thm27}.
Suppose that $X$ is not compact. 
There exists a sequence of distinct points $(x_n)$ in $X$ with no convergent subsequence.
Hence, for each $n$, we may choose an open neighborhood $U_n$ of $x_n$ so that $\diam U_n\to 0$ and $\ol{U_n} \cap \ol{U_m} = \emptyset$ if $m \neq n$.
Then $\ol{\cup_{n\in N}U_n} = \cup_{n\in N}\ol{U_n}$ for any subset $N$ of $\N$.
Set $y_n = \vp(x_n)$ and $a_n = (T^{-1}n)(x_n)$ for all $n$.  Since $A(X)$ has property (A2), for each $n$, there exists $f_n\in A(X)$ such that $f_n(x_n) > a_n$ and $f_n =0$ outside $U_n$.  
Let $f$ be the
the pointwise sum $f= \sum f_n$.  
If $x \notin \ol{\cup U_n}$, then $f = 0$ on the neighborhood $(\ol{\cup U_n})^c$ of $x$.
If $x\in \ol{\cup U_n}$, then $x\in \ol{U_{n_0}}$ for some $n_0$ and $x\notin \ol{\cup_{m\neq n_0}U_m}$.  Hence $f= f_{n_0}$ on the neighborhood $(\ol{\cup_{m\neq n_0}U_m})^c$ of $x$
This shows that $f\in A^\loc(X) = A(X)$.
Moreover, $f \geq T^{-1}n$ on an open neighborhood of $x_n$ for each $n$.
By Theorem \ref{thm27}, $Tf \geq n$ on an open neighborhood of $y_n$ for all $n$.
But then $Tf$ is an unbounded function in $A(Y)$, which contradicts the assumption that $A(Y)$ consists of bounded functions.
\end{proof}

\begin{prop}\label{prop36}
Let $A(X)$ and $A(Y)$  be near vector lattices defined on metric spaces $X$ and $Y$ respectively.  Assume that $A(X) = A^\loc(X)$ or $A^\loc_b(X)$ or  satisfies ($\spadesuit$), and the same holds for $A(Y)$. Suppose that $A(X)$ satisfies ($\heartsuit$) and that there is a dense subset $Y'$ of $Y$ such that $A(Y)$ satisfies condition ($\heartsuit_y$) for all $y\in Y'$.  If $A(X)$ is a vector sublattice of $C(X)$, and $A(X)$ is order isomorphic to $A(Y)$, then $A(Y)$ is a vector sublattice of $C(Y)$.
\end{prop}

\begin{proof}
Suppose that $T: A(X)  \to A(Y)$ is an order isomorphism.  We may assume that $T0 =0$.
Obtain a homeomorphism $\vp:X\to Y$ from Theorem \ref{thm27}.
Let $X' = \vp^{-1}(Y')$. By Theorem \ref{thm29}, applied to both $T$ and $T^{-1}$, if $x\in X'$ and $y = \vp(x)$, there are increasing homeomorphisms $\Phi(y,\cdot),\Psi(x,\cdot):\R\to \R$ such that
\[ Tf(y) = \Phi(y,f(\vp^{-1}(y))) \text{ and } T^{-1}g(x) = \Psi(x,g(\vp(x)))\]
for all $f\in A(X)$ and $g\in A(Y)$.  Moreover, $\Phi(y,\cdot)$ and $\Psi(x,\cdot)$ are mutual inverses.  Since $T0 = 0$, $\Phi(y,0) = 0 = \Psi(x,0)$.
Let $g\in A(Y)$, then $T^{-1}g\in A(X)$. Since $A(X)$ is a vector sublattice of $C(X)$, the pointwise supremum $f = T^{-1}g\vee 0 \in A(X)$.
If  $x \in X'$, then
\[ f(x) = [T^{-1}g(x)]^+ = [\Psi(x,g(\vp(x)))]^+.\]
If  $y\in Y'$, then  $\vp^{-1}(y) \in X'$.  Thus,
\[ Tf(y) = \Phi(y, f(\vp^{-1}(y))) = \begin{cases} 
                                                      g(y) &\text{if $\Psi(\vp^{-1}(y),g(y)) \geq 0$}\\
                                                      0 &\text{if $\Psi(\vp^{-1}(y),g(y))<0$}.
                                                      \end{cases}\]
As $\Psi(\vp^{-1}(y),\cdot)$ is an increasing homeomorphism with $\Psi(\vp^{-1}(y),0) = 0$, $\Psi(\vp^{-1}(y),g(y))\geq 0 $ if and only if $g(y) \geq 0$.
Therefore, $Tf(y) = [g(y)]^+$.
For any $z\in Y$, there is a sequence $(y_n)$ in $Y'$ converging to $z$.  Then $Tf(z) = \lim Tf(y_n) = \lim [g(y_n)]^+ = [g(z)]^+$.
We have shown that if $g\in A(Y)$, then its positive part, taken pointwise, is a function in $A(Y)$.  Since $A(Y)$ is also a vector subspace of $C(Y)$, it follows that $A(Y)$ is a vector sublattice of $C(Y)$.
\end{proof}

\begin{prop}\label{prop34}
Let $A(X)$ and $A(Y)$ be near vector lattices defined on metric spaces $X$ and $Y$ respectively.
Suppose that $A(X) = A^\loc(X)$ or $A^\loc_b(X)$ or satisfies condition ($\spadesuit$).
Assume the same for $A(Y)$.
Let $T:A(X)\to A(Y)$ be an order isomorphism, with the associated homeomorphism $\vp : X\to Y$ given by Theorem \ref{thm27}.   
Let $G$ be a subset of $X$ so that $A(X)$ satisfies ($\heartsuit_x)$ for all $x\in G$ and $A(Y)$ satisfies ($\heartsuit_y$) for all $y \in \vp(G)$. 
Then there is an order isomorphism $S: C(G)\to C(\vp(G))$ such that $S(f_{|G}) = (Tf)_{|\vp(G)}$ for all $f\in A(X)$.  $S$ is continuous if both $C(G)$ and $C(\vp(G))$ are equipped with the topology of uniform convergence on compact sets. Moreover, if $A(Y)$ consists of bounded functions, then $S(C_b(G)) \subseteq C_b(\vp(G))$.
\end{prop}

\begin{proof}
By Theorem \ref{thm29}, there exists  a function $\Phi:\vp(G)\times \R\to \R$ such that $\Phi(y,\cdot)$ is an increasing homeomorphism for each $y\in \vp(G)$ and that 
$Tf(y) = \Phi(y,f(\vp^{-1}(y)))$ for all $f\in A(X)$ and all $y\in \vp(G)$.
Let $h\in C(G)$.  We claim that $y \mapsto \Phi(y,h(\vp^{-1}(y)))$ is a continuous function on $\vp(G)$.
Let  $(y_n)$ be a sequence in $\vp(G)$ converging to $y_0\in \vp(G)$.
Set $t_n = h(\vp^{-1}(y_n))$ and $t_0 =h(\vp^{-1}(y_0))$.
It suffices to show that $(\Phi(y_n,t_n))$ has a subsequence converging to $\Phi(y_0,t_0)$.
By using a subsequence if necessary, we may assume that $(t_n)$ is monotone.  We consider the case where it is increasing.
Denoting by $t$ the constant function with value $t$, we note that $t_n, t_0\in A(X)$ and 
\[ \Phi(y_n,t_n) = Tt_n(y_n) \leq Tt_0(y_n).\]
Thus 
\[\limsup\Phi(y_n,t_n) \leq \lim Tt_0(y_n) = Tt_0(y_0) = \Phi(y_0,t_0).\]
On the other hand, suppose that $a < \Phi(y_0,t_0)$.
As $\Phi(y_0,\cdot)$ is continuous and $(t_n)$ converges to $t_0$, there exists $n_0$ such that $Tt_{n_0}(y_0) = \Phi(y_0,t_{n_0}) > a$.
By continuity of $Tt_{n_0}$, there exists $m_0> n_0$ such that $Tt_{n_0}(y_m) > a$ for all $m \geq m_0$.
Then, if $m \geq m_0$,
\[a < Tt_{n_0}(y_m) \leq Tt_m(y_m) = \Phi(y_m,t_m).\]
This shows that  $\liminf \Phi(y_n,t_n) \geq a$ for any $a < \Phi(y_0,t_0)$ and completes the proof of the claim.

Define $S: C(G)\to C(\vp(G))$ by $Sh(y) = \Phi(y,h(\vp^{-1}(y)))$ for all $h\in C(G)$ and all $y\in \vp(G)$.  Obviously, $h_1\leq h_2$ in $C(G)$ implies $Sh_1 \leq Sh_2$ and $S(f_{|G}) = (Tf)_{|\vp(G)}$ for all $f\in A(X)$.  By symmetry, there is a map $S': C(\vp(G)) \to C(G)$ 
given by $S'g(x) = \Psi(x,g(\vp(x)))$, where $\Psi(x,\cdot)$ is the inverse of the map $\Phi(\vp(x),\cdot)$.  It is easy to see that $S'= S^{-1}$.
Hence $S$ is an order isomorphism.

Let $f_0\in C(G)$.  By the continuity and monotonicity of $\Phi(y,\cdot)$ for each $y$, the sequences $(S(f_0-\frac{1}{n}))$ and $(S(f_0+\frac{1}{n}))$ converge pointwise monotonically  to $Sf_0$.  By Dini's Theorem, both sequences converge uniformly to $Sf_0$ on compact subsets of $\vp(G)$.  Suppose that $\ep >0$ and $K$ is a compact subset of $\vp(G)$.  There exists $n$ such that $|S(f_0\pm \frac{1}{n}) - Sf_0| \leq \ep$ on $K$.
If $f\in C(G)$ and $|f-f_0| \leq \frac{1}{n}$ on the compact set $\vp^{-1}(K)$, then
\[ S(f_0-\frac{1}{n}) \leq Sf \leq S(f_0+\frac{1}{n}) \]
on $K$. Hence $|Sf - Sf_0| \leq \ep$ on $K$.
This shows that $S$ is continuous if both $C(G)$ and $C(\vp(G))$ are equipped with the topology of uniform convergence on compact sets.

Finally, suppose that $A(Y)$ consists of bounded functions.  Then $\Phi(\cdot,t) = Tt$ is a bounded function on $Y$  for any $t \in \R$.  If $f\in C_b(G)$, choose $t_1,t_2\in \R$ such that $t_1 \leq f(x) \leq t_2$ for all $x\in G$.  Then 
\[ \Phi(y,t_1) \leq \Phi(y,f(\vp^{-1}(y))) = S(y) \leq \Phi(y,t_2)\]
for all $y\in \vp(G)$.  Hence $S{f}$ is bounded.
\end{proof}

\begin{cor}\label{cor35.1} 
In Proposition \ref{prop34}, assume in addition that $A(X)$ and $A(Y)$ both satisfy condition ($\heartsuit$). Then $T$ is continuous if both $A(X)$ and $A(Y)$ are equipped with the topology of uniform convergence on compact sets.
\end{cor}

Let $A(X)$ be a set of real-valued functions on a topological space $X$ and let $G$ be a subset of $X$.  Say that a function $f:G\to \R$ belongs to $A^\loc(G)$ if for every $x_0\in G$, there exists an open neighborhood $U$ of $x_0$ in $G$ and $g\in A(X)$ such that $f = g$ on $U$.

\begin{cor}\label{cor35}
In the notation of Proposition \ref{prop34}, $S$ is an order isomorphism  from $A^\loc(G)$ onto $A^\loc(\vp(G))$.  Furthermore, if $A(Y)$  consists of bounded functions, then $S(A^\loc_b(G)) \subseteq A^\loc_b(\vp(G))$.
In particular, $S$ is an order isomorphism from $A^\loc_b(G)$ onto $A^\loc_b(\vp(G))$ if both $A(X)$ and $A(Y)$ consist of bounded functions.
\end{cor}

\begin{proof}
We will show that $S(A^\loc(G)) \subseteq A^\loc(\vp(G))$.  Then $S(A^\loc(G)) = A^\loc(\vp(G))$ by symmetry.
Suppose that $f\in A^\loc(G)$ and $y_0\in \vp(G)$.  Let $x_0 = \vp^{-1}(y_0)$.  There exists an open neighborhood $U$ of $x_0$ in $G$ and a function $g \in A(X)$ such that $f =g$ on $U$.
Then $Sf = S(g_{|G})= (Tg)_{|\vp(G)}$ on $\vp(U)$, which is an open neighborhood of $y_0$ in $\vp(G)$.
Of course, $Tg \in A(Y)$.  This proves that $Sf \in A^\loc(\vp(G))$.

If all functions in $A(Y)$ are bounded, then by the previous paragraph and the last statement in Proposition \ref{prop34}, 
\[ S(A^\loc_b(G)) = S(A^\loc(G)\cap C_b(G)) \subseteq  A^\loc(\vp(G)) \cap C_b(\vp(G)) = A^\loc_b(\vp(G)).\]
\end{proof}

\subsection{Specifics}\label{subsec6.2}
In this part, spaces $X$ and $Y$ will always be metric spaces, possibly with additional properties.  
The metric on both $X$ and $Y$ will be denoted by $d$, even though they may differ.
We adopt the convention that when a space from Examples B or C is mentioned, it is assumed to satisfy the conditions given in these examples. For instance, for the space $\lip(X)$, $X$ will be assumed to be a complete metric space and $\lip(X)$ itself will be supposed to be uniformly separating.
If $A(X)$ and $A(Y)$ are spaces of functions, we write $A(X)\sim A(Y)$ to mean that they are order isomorphic.
The next result is an immediate consequence of Proposition \ref{prop31}.

\begin{prop}\label{prop36.1}
Let 
\begin{enumerate}
\item $A(X) = C(X)$, $\Lip^\loc(X)$, $\lip^\loc(X)$, $U^\loc(X)$ or $C^p(X)$;
\item $A(\ol{X}) = C^p(\ol{X})$;
\item $A(Y)=C_b(Y)$, $\Lip_b(Y)$, $\Lip^\loc_b(Y)$, $\lip_b(Y)$, $\lip^\loc_b(Y)$, $U_b(Y)$, $U^\loc_b(Y)$, or $C^p_b(Y)$;
\item $A(\ol{Y}) = C^p_b(\ol{Y})$ or $C^p_*(\ol{Y})$.  
\end{enumerate}
If $A(X)  \sim A(Y)$ or $A(\ol{Y})$, then $X$ is compact.  If  $A(\ol{X})  \sim A(Y)$ or $A(\ol{Y})$, then $\ol{X}$ is compact.
\end{prop}

\begin{cor}\label{cor36.2}
\begin{enumerate}
\item There exists $Y$ such that any of the following holds if and only if $X$ is compact.
\begin{alignat*}{3}
C(X) &\sim C_b(Y), & C(X)&\sim U^\loc_b(Y), & C(X)& \sim U_b(Y), \\
\Lip^\loc(X)&\sim \Lip^\loc_b(Y), & && \Lip^\loc(X)&\sim \Lip_b(Y),\\ 
\lip^\loc(X)&\sim \lip^\loc_b(Y),  & &&
\lip^\loc(X)&\sim \lip_b(Y),\\
U^\loc(X)&\sim U^\loc_b(Y), & U^\loc(X)&\sim U_b(Y),& U^\loc(X)&\sim C_b(Y).
\end{alignat*}
\item For any $X$ and $Y$ and any $p$ and $q$, 
\begin{align*}
C^p(X)&\not\sim C^q_b(Y), C^q_b(\ol{Y}), C^q_*(\ol{Y}),\\
C^p(\ol{X}) & \not\sim C^q_b(Y).
\end{align*}
\item $C^p(\ol{X}) \sim C^q_b(\ol{Y})$ for some $Y$ and some $q$ if and only if $X$ is a bounded open set in $\R^n$ for some $n$.
\end{enumerate}
\end{cor}

\begin{proof}
(a) If $X$ is compact, take $Y = X$ and we have equality of the spaces in all cases.  Any one of the order isomorphisms imply that $X$ is compact by Proposition \ref{prop36.1}.

\noindent (b) Assume that one of the given order isomorphisms exists. By Theorem \ref{thm27}, we have in the respective cases, $X$  is homeomorphic to $Y$, $X$ is homeomorphic to $\ol{Y}$, or $\ol{X}$ is homeomorphic to $Y$. We also conclude from Proposition \ref{prop36.1}
that $X$ is compact in the first three cases, and $\ol{X}$, is compact in the last case.  
In the last case, it would imply that $Y$ is compact.  However, since $X$ and $Y$ are open sets in Banach spaces, they are never compact.

\noindent (c) As  in the proof of (b), if $C^p(\ol{X})\sim C^q_b(\ol{Y})$, then $\ol{X}$ is a compact set and hence $X$ is a bounded open set in $\R^n$ for some $n$. Conversely, if $X$ is a bounded open set in $\R^n$, then $C^p(\ol{X}) = C^p_b(\ol{X})$.
\end{proof}

The next result is an easy application of Proposition \ref{prop36}.

\begin{prop}\label{prop37}
Let $A(X)$ be  one of the spaces $C(X), C_b(X)$, 
$\Lip(X)$, $\Lip_b(X)$, $\Lip^\loc(X)$, $\Lip^\loc_b(X)$, $\lip_\al(X)$, $\lip_{\al,b}(X)$, $\lip_\al^\loc(X)$, $\lip_{\al,b}^\loc(X)$,  $U(X)$, $U_b(X)$, $U^\loc(X)$ or $U^\loc_b(X)$,
Let $A(Y)$, respectively, $A(\ol{Y})$,  be one of the spaces $C^p(Y)$, $C^p_b(Y)$, $C^p(\ol{Y})$, $C^p_b(\ol{Y})$ or $C^p_*(\ol{Y})$.  Then $A(X)\not\sim A(Y), A(\ol{Y})$.
\end{prop}

A metric space $X$ is {\em discrete} if all of its points are isolated.  It is {\em separated} if there exists $\ep > 0$ such that $d(x,x') > \ep$ if $x$ and $x'$ are distinct points in $X$.
For any $\ep > 0$. let 
\[X_\ep =\{x\in X: d(x, X\bs\{x\}) > \ep\}.\]
$X$ is said to be {\em proximally compact} if every sequence in $X$ has a subsequence that either converges or is contained in $X_\ep$ for some $\ep > 0$.  $X$ is {\em locally proximally compact} if for any $x_0\in X$, there exists $r> 0$ such that $\ol{B(x_0,r)}$ is proximally compact.  Observe that every proximally compact metric space is complete.

\begin{prop}\label{prop38.1}
If a metric space is discrete and proximally compact, then it is separated.
\end{prop}

\begin{proof}
Suppose that $X$ is discrete and proximally compact but not separated.  There are sequences $(x_n)$ and $(x'_n)$ in $X$ such that $0 < d(x_n,x'_n)\to 0$.  Since $X$ is discrete, we may assume that $(x_n)$ has no convergent subsequence.  
Since $X$ is proximally compact, $(x_n)$ has a subsequence contained in $X_\ep$ for some $\ep > 0$.
This contradicts the choice of $(x_n)$ and $(x'_n)$. 
\end{proof}

\begin{prop}\label{prop39}
Let $X$ and $Y$ be metric spaces and let $T: C(X)\to C(Y)$ be an order isomorphism.
If $X$ is not discrete, then for any $0< \al <1$, there exists $f\in U_b(X)$ such that $Tf \notin \lip_\al^\loc(Y) \cup \Lip^\loc(Y)$.
\end{prop}

\begin{proof}
By Theorem \ref{thm29}, there exist a homeomorphism $\vp : X\to Y$ and a function $\Phi: Y\times \R \to \R$ such that $\Phi(y,\cdot):\R\to \R$ is an increasing homeomorphism for all $y\in Y$, and that $Tf(y) = \Phi(y,f(\vp^{-1}(y)))$ for all $f\in C(X)$ and all $y\in Y$.
Suppose that $X$ is not discrete.  
There is a sequence of distinct points $(x_n)$ in $X$ converging to a point $x_0\in X$, with $x_0 \neq x_n$ for all $n$.
Set $y_n = \vp(x_n)$, $n \geq 0$, and let $r_n = d(y_n,y_0)$. 
By using a subsequence if necessary, we may assume that $r_n > 4 r_{n+1}$ for all $n$. 
For each $n \in \N$, define $g_n:Y \to \R$ by $g_n(y) = n[(\frac{r_n}{2})^\al-d(y,y_n)^\al]^+$.
Clearly, $g_n \in C(Y)$ and $\|g_n\|_\infty \leq n(r_n/2)^\al\to 0$.
By Corollary \ref{cor35.1} applied to $T^{-1}$, $(T^{-1}g_n)$ converges uniformly to $T^{-1}0$ on the compact set $K = \{x_n: n \geq 0\}$.  
In particular, 
\[\lim_n T^{-1}g_n(x_n)  = \lim_n [(T^{-1}g_n(x_n) -T^{-1}0(x_n)) +  T^{-1}0(x_n)]= T^{-1}0(x_0).\]
It is easy to construct $f\in U_b(X)$ such that $f(x_n) = T^{-1}g_n(x_n)$ for all $n\in \N$ and $f(x_0) = T^{-1}0(x_0)$.
But $Tf(y_n) = g_n(y_n) = n(r_n/2)^\al$ and $Tf(y_0) = 0$.
It is clear that $Tf$ is not Lipschitz with respect to either of the metrics $d$ or $d^\al$ on any neighborhood of $y_0$.
\end{proof}

\begin{prop}\label{prop41.1}
Let $X$ and $Y$  be metric spaces and let $T: C(X)\to C(Y)$ be an order isomorphism such that $T0 = 0$.  Assume that $X$ is not proximally compact. Fix $0 < \al < 1$. There exists  $f\in C(X)\bs U(X)$ such  that $Tf\in \Lip^\loc(Y)\cap \lip^\loc_\al(Y)$.
If $T(C_b(X)) \subseteq C_b(Y)$, we may require additionally that $Tf$ be bounded.
\end{prop}

\begin{proof}
There are sequences $(x_n)$ and $(x'_n)$ in $X$ so that $(x_n)$ has no convergent subsequence and $0 < d(x_n,x'_n) \to 0$.  It follows that $(x'_n)$ has no convergent subsequence.  We may assume that the points in $(x_n) \cup (x'_n)$ are all  distinct.  Let $y_n = \vp(x_n)$ and $y_n' = \vp(x_n')$ for all $n$, where $\vp : X\to Y$ is the homeomorphism associated with $T$.
Since the points in $(y_n)\cup (y_n')$ are all distinct and neither $(y_n)$ nor $(y_n')$ has a convergent subsequence, for each $n$, 
\[c_n = \inf_{m\neq n}d(y_n,y_m) \wedge \inf_m d(y_n, y'_m) > 0.\]
Choose $(r_n)$ converging to $0$ so that $0 < 4r_n < c_n$.
Let $a_n = T1(y_n)$ for all $n$.
Define $h$ on $Y$ by 
\[ h(y) = \begin{cases}
              a_n(1-\frac{d(y,y_n)}{r_n}) &\text{ if $d(y,y_n) \leq r_n$ for some $n$},\\
              0 &\text{ otherwise}.
              \end{cases}\]
Then $h\in \Lip^\loc(Y)\cap \lip^\loc_\al(Y)$.  Since $h(y_n) = a_n = T1(y_n)$ and $h(y_n') = 0 = T0(y_n')$ for all $n$,
$T^{-1}h(x_n) = 1$ and $T^{-1}h(x_n') = 0$. 
As  $d(x_n,x'_n) \to 0$, this shows that $T^{-1}h\notin U(X)$.
The first part of the proof is completed by taking $f = T^{-1}h$.

Assume additionally that $T(C_b(X)) \subseteq C_b(Y)$.  Then $T1$ is a bounded function and hence $(a_n)$ is a bounded sequence.  Thus $h$ is a bounded function and hence so is $Tf$.
\end{proof}

\begin{cor}\label{cor42.1}
Let $X$ be a metric space that is not proximally compact.  For any $\ep > 0$, there exists $f\in \Lip^\loc(X)\bs U(X)$ such that $\|f\|_\infty \leq \ep$.
\end{cor}

\begin{proof}
Take $X = Y$ and $T:C(X)\to C(X)$ to be the identity map in Proposition \ref{prop41.1}.
By the proposition, there exists $h\in \Lip^\loc_b(X) \bs U(X)$. Then $f = {\ep}(\|h\|_\infty+1)^{-1}h\in \Lip^\loc_b(X) \bs U(X)$ and $\|f\|_\infty\leq \ep$.
\end{proof}

\begin{thm}\label{thm41}
\begin{enumerate}
\item There exists $Y$ such that any of the following holds if and only if $X$ is discrete.
\begin{align*}
C(X) \text{ or }U^\loc(X)& \sim \Lip^\loc(Y) \text{ or } \lip^\loc_\al(Y), \\
C_b(X)\text{ or }U^\loc_b(X)&\sim \Lip^\loc_b(Y)\text{ or } \lip^\loc_{\al,b}(Y).
\end{align*}
\item There exists $Y$ such that any of the following holds if and only if $X$ is a finite set.
\begin{align*}
C(X), U^\loc(X) & \sim \Lip^\loc_b(Y), \Lip_b(Y), \lip^\loc_{\al,b}(Y), \lip_{\al,b}(Y) \\
C_b(X), U_b(X), U^\loc_b(X)&\sim \Lip^\loc(Y), \lip^\loc_\al(Y),\\
C(X), U^\loc(X)&\sim \lip_\al(Y).
\end{align*}
\item There exists $Y$ such that any of the following holds if and only if $X$ is separated.
\[
\Lip_b(X), \lip_{\al,b}(X)\sim C_b(Y), U^\loc_b(Y).
\]
\end{enumerate}
\end{thm}

\begin{proof}
In all the cases, by Proposition \ref{prop34}, any order isomorphism as given above extends to an order isomorphism from $C(X)$ onto $C(Y)$.
In view of Proposition \ref{prop39}, the existence of any one of the given order isomorphisms implies that $X$ and $Y$  are discrete.  (For (c), apply the proposition to the map $T^{-1}$.)
In particular, this proves (a), since the relevant spaces coincide if we take $Y = X$ when $X$ is discrete.

\noindent (b)  In addition, with any of the order isomorphisms given in the first two lines of (b), we may conclude from Proposition \ref{prop31} that either $X$ or $Y$ is compact.  Hence both are compact.  A compact discrete space is necessarily finite.  

Since $X$ is discrete, $C(X) = U^\loc(X)$.  Thus, if either of the order isomorphisms in the last line of (b) holds, $C(X)\sim \lip_\al(Y)$.  By Proposition \ref{prop34},  the given order isomorphism extends to an order isomorphism from $C(X)$ onto $C(Y)$.   Therefore,  $C(Y) = \lip_\al(Y)$.
Since $Y$ is discrete , if $(y_n)$ is an infinite sequence of distinct points in $Y$, there is a function $g\in C(Y)$ such that $g(y_n) = nd(y_n,y_1)^{\vp}$ for all $n$. But then $g \notin\lip_\al(Y)$, contradicting the above.   Hence $Y$ is finite and therefore so is $X$.

Of course, if $X$ is a finite set, then all the given order isomorphisms are trivially true if we take $Y = X$.

\noindent(c) Suppose that any of the given order isomorphisms hold, label it as $T$.  We may assume that $T0 = 0$.  
Use Proposition \ref{prop34} to extend $T$ to an order isomorphism $S: C(X) \to C(Y)$ such that $S(C_b(X)) \subseteq C_b(Y)$.
Since $Y$ is discrete,
it follows that $U^\loc_b(Y) = C_b(Y)$.
If $X$ is not separated, by Proposition \ref{prop38.1}, it is not proximally compact.  It follows from Proposition  \ref{prop41.1} that there exists $g\in C_b(Y)$ such that $S^{-1}g\notin U(X)$.
Hence $T^{-1}g= S^{-1}g\notin \Lip_b(X)$ or $\lip_{\al,b}(X)$.  This contradicts the choice of $T$.
Conversely, if $X$ is separated, we choose $Y = X$ and all the spaces are equal.
\end{proof}

\begin{lem}\label{lem45}
If $X$ is proximally compact, then $C(X)= U(X)= U^\loc(X)$, $C_b(X)= U_b(X) = U^\loc_b(X)$, $\Lip_b(X) = \Lip^\loc_b(X)$, and $\lip_{\al,b}(X) = \lip^\loc_{\al,b}(X)$.
\end{lem}

\begin{proof}
Assume  that $X$ is proximally compact.  
Suppose that $f\in C(X)\bs U(X)$. There are sequences $(x_n)$ and $(x'_n)$
in $X$ and $\ep > 0$ such that $d(x_n,x_n') \to 0$ and $|f(x_n)-f(x'_n)| > \ep$ for all $n$.
By the proximal compactness of $X$, $(x_n)$ has a subsequence $(x_{n_k})$ convergent to some $x_0\in X$.  Then $(x'_{n_k})$ converges to $x_0$ as well.
By continuity, $\lim (f(x_{n_k}) - f(x_{n_k}')) = 0$, a contradiction.  
Therefore, $C(X) = U(X)$ and $C_b(X) = U_b(X)$.
We also have $U(X) \subseteq U^\loc(X) \subseteq C(X)$.  Hence $U(X) = U^\loc(X)$ as well.  Similarly, $U_b(X) = U^\loc_b(X)$.

Obviously, $\lip_{\al,b}(X) \subseteq \lip^\loc_{\al,b}(X)$. Suppose, if possible, that there exists $f\in \lip^\loc_{\al,b}(X)\bs \lip_{\al,b}(X)$.  
There are sequences $(x_n)$ and $(x'_n)$ in $X$ such that either $|f(x_n) - f(x_n')| > n d(x_n,x'_n)^\al$
or $\lim d(x_n,x'_n) = 0$ and $\lim |f(x_n)-f(x'_n)|/d(x_n,x'_n)^\al \neq 0$.
Since $f$ is bounded, $0< d(x_n,x'_n) \to 0$ even in the first case.
Take a subsequence $(x_{n_k})$ of $(x_n)$ that converges to some $x_0$.  Then $(x'_{n_k})$ also converges to $x_0$.
Since $f \in \lip^\loc_{\al,b}(X)$, there exists an open neighborhood $U$ of $x_0$ such that $f$ is Lipschitz with respect to $d^\al$ on $U$, and 
\[ \lim_{\substack{d(x,y)\to 0\\x,y \in U}}\frac{|f(x)-f(y)|}{d(x,y)^\al} = 0.\]
This is clearly impossible since $x_{n_k},x'_{n_k}\in U$ for all sufficiently large $k$.
The proof that $\Lip_b(X) = \Lip^\loc_b(X)$ is similar.
\end{proof}

\begin{thm}\label{thm43}
\begin{enumerate}
\item There exists $Y$ such that one of the following holds  if and only if $X$ is proximally compact.
\begin{align*}
U(X) \sim C(Y),\quad &U_b(X)\sim C_b(Y),\\ U(X)\sim U^\loc(Y),\quad & U_b(X)\sim U^\loc_b(Y).
\end{align*}
\item There exists $Y$ such that either $U^\loc(X) \sim C(Y)$ or $U^\loc_b(X)\sim C_b(Y)$ if and only if $X$ is locally proximally compact.
\item There exists $Y$ such that one of the following holds if and only if $X$ is  compact.  
\[ U(X) \sim C_b(Y) \text { or }U^\loc_b(Y),\ \lip_\al(X) \sim \lip^\loc_\al(Y).\] 
\item There exists $Y$ such that either $U(X) \sim \Lip^\loc_b(Y)$ or $\lip^\loc_{\al,b}(Y)$ if and only if $X$ is a finite set.
\item There exists $Y$ such that one of the following holds if and only if $X$ is separated.
\[ U(X)\sim \Lip^\loc(Y) \text{ or } \lip^\loc_\al(Y), U_b(X)\sim \Lip^\loc_b(Y) \text{ or } \lip^\loc_{\al,b}(Y).\]
\end{enumerate}
\end{thm}

\begin{proof}
(a) If $T: U(X) \to C(Y)$ is an order isomorphism, by Proposition \ref{prop34}, it can be extended to an order isomorphism from $C(X)$ onto $C(Y)$.  Thus $U(X) = C(X)$.  
Similarly, using Proposition \ref{prop34}  or Corollary \ref{cor35},  the other three order isomorphisms hold if and only if 
\[ U_b(X) = C_b(X),\  U(X)= U^\loc(X),\ U_b(X)= U^\loc_b(X)\]
respectively.  By Corollary \ref{cor42.1}, $X$ is proximally compact.
The converse follows from Lemma \ref{lem45}.

\noindent(b) As in case (a), $U^\loc(X) \sim C(Y)$ implies that $U^\loc(X) = C(X)$.
Suppose that there exists $x_0\in X$ such that $\ol{B(x_0,r)}$ is not proximally compact for any $r > 0$.
Let $r_1 = 1$ and apply Corollary \ref{cor42.1} to the metric space $\ol{B(x_0,r_1)}$.
We find $f_1 \in C(\ol{B(x_0,r_1)})\bs U(\ol{B(x_0,r_1)})$ with $\|f_1\|_\infty\leq 1$.
Assume that $r_n > 0$ and $f_n \in C(\ol{B(x_0,r_n)})\bs U(\ol{B(x_0,r_n)})$ have been chosen with $\|f_n\|_\infty \leq 1/n$.
There exists $r_{n+1}>0$ such that $2r_{n+1} < r_n$ and that $f_n$ is not uniformly continuous on $A_n = \{x\in \ol{B(x_0,r_n)}: d(x,x_0) \geq 2r_{n+1}\}$.
Finally, choose $f_{n+1} \in C(\ol{B(x_0,r_{n+1})})\bs U(\ol{B(x_0,r_{n+1})})$ such that $\|f_{n+1}\|_\infty \leq (n+1)^{-1}$.
Denote the restriction of $f_n$ to $A_n$ by $g_n$.
Define $g: \cup A_n \cup \{x_0\} \to \R$ by $g = f_n$ on $A_n$ and $g(x_0) = 0$.
Then $g$ is continuous. By the Tietze Extension Theorem, $g$ extends to a (bounded) continuous function on $X$.
Clearly, $g$ is not uniformly continuous on any neighborhood of $x_0$.  Thus $g\notin U^\loc(X)$.
This completes the proof that $X$ is locally proximally compact if $U^\loc(X) \sim C(Y)$. 
Conversely, suppose that $X$ is locally proximally compact and that $f \in C(X)$.  For any $x_0 \in X$, choose $r >0$ such that $\ol{B(x_0,r)}$ is proximally compact.
By Lemma \ref{lem45}, $f$ is uniformly continuous on $\ol{B(x_0,r)}$.  Reduce $r$ if necessary to assume that $f$ is bounded on $\ol{B(x_0,r)}$.
There exists a uniformly continuous function $g:\ol{B(x_0,r)}\to [0,1]$ such that $g =1$ on $B(x_0,r/3)$ and $g= 0$ outside $B(x_0,2r/3)$.
Then $fg$ is uniformly continuous on $\ol{B(x_0,r)}$ and equals $0$ outside $B(x_0, 2r/3)$.
We may extend it to a function $h\in U(X)$ by defining $h$ to be $0$ outside $\ol{B(x_0,r)}$.
Observe that $h = f$ on $B(x_0,r/3)$. This proves that $f\in U^\loc(X)$.

Similarly, $U^\loc_b(X)\sim C_b(Y)$ if and only if $U^\loc_b(X) = C_b(X)$ if and only if $X$ is locally proximal compact.

\noindent(c) Suppose that $T$ is an order isomorphism from $U(X)$ onto $C_b(Y)$ or $U^\loc_b(Y)$.  
We may assume that $T0 = 0$.  
By Proposition \ref{prop34}, $T$ may be extended to an order isomorphism $S:C(X)\to C(Y)$ such that  $S(C_b(X)) \subseteq C_b(Y)$.
By Proposition \ref{prop41.1}, $X$ is proximally compact.
It follows from Lemma \ref{lem45} that $U(X) = C(X)$.
Then we have $C(X) = U(X) \sim C_b(Y)$ or $U^\loc_b(Y)$.
By Corollary \ref{cor36.2}(a), $X$ is compact.

Next, suppose that $T: \lip_\al(X)\to \lip_{\al}^\loc(Y)$ is an order isomorphism.
By Corollary \ref{cor35}, $T$ can be extended to an order isomorphism from $\lip_\al^\loc(X)$ onto $\lip^\loc_\al(Y)$. It follows that $\lip_\al(X) = \lip_\al^\loc(X)$.
Suppose that $X$ is not compact. Since $X$ is complete by assumption, $X$ contains a separated sequence $(x_n)$.
Choose $r> 0$ such that $d(x_m,x_n) > 2r$ if $m\neq n$.
For each $n$, there exists $h_n\in \lip_\al(X)$ such that $h_n(x_n) = 1$ and $h_n(x) = 0$ if $x\notin B(x_n,r)$.
Take $a_n = nd(x_n,x_1)^\al$ and let $h$ be the pointwise sum $\sum a_nh_n$.
It is clear that $h \in \lip^\loc_\al(X)$.  However, $h \notin \lip_\al(X)$ since $h(x_n) - h(x_1) = nd(x_n,x_1)^\al$ for all $n$.
This contradicts the fact that $\lip_\al(X) = \lip_\al^\loc(X)$ and 
 shows that $X$ is compact if $\lip_\al(X) \sim \lip^\loc_\al(Y)$. 

Conversely, if $X$ is compact, then $U(X) = C_b(X)= U^\loc_b(X)$ and $\lip_\al(X) = \lip^\loc_\al(X)$.

\noindent(d) If $U(X)\sim \Lip^\loc_b(Y)$, respectively $\lip^\loc_{\al,b}(Y)$, then by Corollary \ref{cor35},
$U^\loc(X)$ $\sim \Lip^\loc(Y)$, respectively $\lip^\loc_\al(Y)$. It follows from Theorem \ref{thm41}(a) that $X$ is discrete. But then $Y$ is discrete since it is homeomorphic to $X$.  Therefore, $\Lip^\loc_b(Y) = \lip^\loc_{\al,b}(Y) = C_b(Y)$.  Hence $U(X) \sim C_b(Y)$.  By part (c), $X$ is compact.  Since $X$ is discrete and compact, it is finite.
The converse is trivial.

\noindent(e) If $U(X) \sim \Lip^\loc(Y)$, respectively $\lip^\loc_\al(Y)$, we also have $U^\loc(X)\sim \Lip^\loc(Y)$, respectively $\lip^\loc_\al(Y)$, by Corollary \ref{cor35}.
Thus $X$ is discrete by Theorem \ref{thm41}(a). Hence $Y$ is discrete.
But then $\Lip^\loc(Y)= \lip^\loc_\al(Y) = C(Y)$ and we have $U(X) \sim C(Y)$.   By (a), $X$ is proximally compact.
Thus $X$ is separated by  Proposition \ref{prop38.1}.
Similarly, $U_b(X)\sim \Lip^\loc_b(Y) \text{ or } \lip^\loc_{\al,b}(Y)$ implies that 
$U^\loc_b(X)\sim \Lip^\loc_b(Y) \text{ or } \lip^\loc_{\al,b}(Y)$
by Corollary \ref{cor35}, and thus $X$ and $Y$ are discrete by Theorem \ref{thm41}(a).
Therefore, $U_b(X) \sim C_b(Y)$ and it follows from (a) that $X$ is proximally compact.  As above, $X$ is separated since it is both discrete and proximally compact.
Conversely, if $X$ is separated, then $U(X) = \Lip^\loc(X) = \lip^\loc_\al(X) = C(X)$ and $U_b(X) = \Lip^\loc_b(X) = \lip^\loc_{\al,b}(X) = C_b(X)$.
\end{proof}

\begin{thm}\label{thm47}
There exists $Y$ such that $\Lip_b(X)\sim \Lip^\loc_b(Y)$ or $\lip_{\al,b}(X)\sim \lip^\loc_{\al,b}(Y)$ if and only if $X$ is proximally compact.
\end{thm}

\begin{proof}
By Corollary \ref{cor35}, any order isomorphism from $\Lip_b(X)$ onto $\Lip^\loc_b(Y)$
can be extended to an order isomorphism from $\Lip^\loc_b(X)$ onto $\Lip^\loc_b(Y)$. Thus $\Lip_b(X) = \Lip^\loc_b(X)$.
Similarly, $\lip_{\al,b}(X)\sim \lip^\loc_{\al,b}(Y)$ implies that $\lip_{\al,b}(X)=  \lip^\loc_{b}(X)$.
If $X$ is not proximally compact, by Corollary \ref{cor42.1}, there exists $f\in \Lip^\loc_b(X) \subseteq \lip^\loc_{\al,b}(X)$ such that $f\notin U(X)$.
Thus $f\notin \Lip_b(X) \cup \lip_{\al,b}(X)$, contrary to the above.
The converse follows from Lemma \ref{lem45}.
\end{proof}

\begin{thm}\label{thm57}
\begin{enumerate}
\item There exists $Y$ such that one of the following holds if and only if $X$ is discrete.
\begin{align*}
\Lip^\loc(X) &\sim \lip^\loc_\al(Y),\\
\Lip^\loc_b(X) &\sim \lip_{\al,b}(Y) \text{ or } \lip^\loc_{\al,b}(Y).
\end{align*}
\item There exists $Y$ such that one of the following holds if and only if $X$ is a finite set.
\begin{align*}
\Lip^\loc(X)&\sim \lip_\al(Y), \lip_{\al,b}(Y) \text{ or } \lip^\loc_{\al,b}(Y), \\
\lip^\loc_\al(X) &\sim \Lip_b(Y) \text { or } \Lip^\loc_b(Y).
\end{align*}
\item There exists $Y$ such that $\Lip_b(X) \sim \lip^\loc_{\al,b}(Y)$ if and only if $X$ is separated.
\end{enumerate}
\end{thm}

\begin{proof}
(a) Suppose that $X$ is not discrete.  
There is a sequence of distinct points $(x_n)$ in $X$ convergent to some $x_0 \in X$, $x_0 \neq x_n$ for all $n$.
Let $T$ be one of the indicated order isomorphisms.  We may assume that $T0=0$. Let $\vp :X\to Y$ be the homeomorphism associated with $T$. Set $y_n =\vp(x_n)$, $y_0 = \vp(x_0)$.  There exists $C<\infty$ such that $d(x_n,x_0) \leq C$ for all $n$.
The function $f(x) = d(x,x_0)\wedge C$ belongs to $\Lip^\loc_b(X)$ and $f(x_0) = 0$.
For each $k\in \N$, $T(kf) \in \lip^\loc_\al(Y)$, and $T(kf)(y_0) = 0$.
Since $T(kf)$ agrees with a function in $\lip_\al(Y)$ on a neighborhood of $y_0$ and $(y_n)$ converges to $y_0$, $T(kf)(y_n)/d(y_n,y_0)^\al \to 0$ as $n \to \infty$.
Choose $n_1 < n_2 < \cdots$ such that $d(y_{n_{k+1}}, y_0) < \frac{1}{4}d(y_{n_k},y_0)$ and 
\[ \frac{T(kf)(y_{n_k})}{d(y_{n_k},y_0)^\al} \to 0 \text{ as $k \to \infty$}.\]
Then there exists $g\in \lip_{\al,b}(Y)$ such that $g(y_{n_k}) = T(kf)(y_{n_k})$ for all $k$ and $g(y_0) = 0$.
For example, take 
\[ g(y) = \begin{cases}
            T(kf)(y_{n_k})(1 - \frac{2d(y,y_{n_k})}{d(y_{n_k},y_0)})&\text{if $d(y,y_{n_k}) < \frac{d(y_{n_k},y_0) }{2}$ for some $k$,}\\
            0 &\text{otherwise.}
           \end{cases}
\]
Since $T^{-1}g(x_{n_k}) = kf(x_{n_k})= k d(x_{n_k},x_0)$ and $T^{-1}g(x_0) = 0$, $T^{-1}g$ is not Lipschitz on any neighborhood of $x_0$, contrary to the assumption.

Conversely, if $X$ is discrete, then $\Lip^\loc(X) = C(X) = \lip^\loc_\al(X)$ and $\Lip^\loc_b(X) = C_b(X) = \lip^\loc_{\al,b}(X) = \lip_{\al,b}(Y)$, where $Y$ is the set $X$ endowed with the discrete metric.

\noindent (b)  Suppose that any one of the order isomorphisms in part (b) holds. 
By Corollary \ref{cor35}, it leads to either $\Lip^\loc(X) \sim \lip^\loc_\al(Y)$ or $\lip^\loc_\al(X) \sim \Lip^\loc(Y)$.
In either case, $X$ and $Y$ are discrete by (a).
Thus $\Lip^\loc(X) = C(X) = \lip^\loc_\al(X)$.
Therefore, we have
\[
C(X)\sim \lip_\al(Y), \lip_{\al,b}(Y), \lip^\loc_{\al,b}(Y),  \Lip_b(Y) \text { or } \Lip^\loc_b(Y).\]
In all cases, it follows from Theorem \ref{thm41}(b) that $Y$ is finite.  Hence $X$ is finite as well.
The converse is trivial.

\noindent (c) By Corollary \ref{cor35}, an order isomorphisms $T: \Lip_b(X)\to \lip^\loc_{\al,b}(Y)$ can be extended to an  order isomorphism from $\Lip^\loc(X)$ to   $\lip^\loc_\al(Y)$.  By part (a), $X$ is discrete and hence so is $Y$.
Thus $\lip^\loc_{\al,b}(Y) = C_b(Y)$.  It follows from Theorem \ref{thm41}(c) that $X$ is separated.
 
Conversely, if $X$ is separated, then $\Lip_b(X) = C_b(X) = \lip^\loc_{\al,b}(X)$.
\end{proof}

\begin{lem}\label{lem56}
Let $A(X) = \lip_\al(X), \lip_{\al,b}(X), U(X)$ or $U_b(X)$ and $A(Y) =  \lip_\al(Y), \lip_{\al,b}(Y), U(Y), U_b(Y)$ or $\Lip_b(Y)$.  Suppose that $T:A(X)\to A(Y)$ is an order isomorphism such that $T0=0$. Denote by $\vp : X\to Y$ its associated homeomorphism.  Let  $(x_n), (x'_n)$  be sequences in $X$ such that $(x_n)$ has no convergent subsequence and that $0< d(x_n,x'_n) \to 0$.  
Set  $y_n = \vp(x_n)$ and $y'_n = \vp(x'_n)$.  
Then there exist $r > 0$, $n_1 < n_2 <\cdots$ and a bounded function $f \in A(X)$ such that $Tf(y_{n_m}) = m(r \wedge d(y_{n_m}, y'_{n_m}))$ and $Tf(y'_{n_m}) = 0$.
In particular, if $(Tt(y_n))$ is bounded for some $0 < t\in \R$, then $\liminf d(y_n,y'_n ) = 0$.
\end{lem}

\begin{proof}
Observe that neither $(x_n)$ nor $(x'_n)$ can have a convergent subsequence.
Thus the same holds for $(y_n)$ and $(y'_n)$.  Since $Y$ is complete by assumption, we may assume that there is an $r > 0$ so that $d(y_n,y_m) > r$ if $n \neq m$.  
If $d(y_n,y'_n) \not\to 0$, then by using a subsequence if necessary, we may further assume that the set $(y_n)\cup (y_n')$ is separated.  On the other hand, if $d(y_n,y'_n) \to 0$, then we may assume that $d(y_n,y'_n) < r/2$ for all $n$.  In either case,
there exists $g\in \Lip_b(Y)$ such that 
$g(y_n) = r \wedge d(y_n,y_n')$  and $g(y'_n) = 0$ for all $n$.   Then $g\in A(Y)$.
Hence $T^{-1}(mg) \in A(X)$ for all $m \in \N$.
Therefore,
\[ \frac{T^{-1}(mg)(x_n)}{d(x_n,x'_n)^\al} = \frac{T^{-1}(mg)(x_n) - T^{-1}(mg)(x'_n)}{d(x_n,x'_n)^\al} \to 0,\]
where we take $\al =0$ if $A(X) = U(X)$ or $U_b(X)$.
Choose $n_1 < n_2 < \cdots$ such that 
\[  \frac{T^{-1}(mg)(x_{n_m})}{d(x_{n_m},x'_{n_m})^\al} \to 0.\]
Since $(x_{n_m})$ has a separated subsequence, by taking a further subsequence if necessary, we may assume that 
there exists  a bounded function $f\in A(X)$ such that 
\[ f(x_{n_m}) = {T^{-1}(mg)(x_{n_m})}\text{ and } f(x'_{n_m}) = 0 \text{ for all $m$}.\]
Thus $Tf(y_{n_m}) = mg(y_{n_m}) = m(r \wedge d(y_{n_m},y'_{n_m}))$ and $Tf(y'_{n_m}) = 0$.

Suppose that $(Tt(y_n))$ is bounded for some $0 < t\in \R$.   Observe that $f(x_{n_m}) \leq t$ for all sufficiently large $m$.  Then $Tf(y_{n_m}) \leq Tt(y_{n_m})$ for all sufficiently large $m$ and hence $(Tf(y_{n_m}))$ is bounded.  Since $Tf(y_{n_m})  = m(r \wedge d(y_{n_m},y'_{n_m}))$, we must have $d(y_{n_m},y'_{n_m})\to 0$.
\end{proof}

\begin{prop}\label{prop57.1}
Let $A(X) = \lip_\al(X), \lip_{\al,b}(X), U(X)$ or $U_b(X)$. If there exists $Y$ such that 
\begin{enumerate}
\item $A(X) \sim \Lip_b(Y)$, then $X$ is proximally compact.
\item $A(X) \sim \lip_{\al,b}(Y)$ or $U_b(Y)$, then the associated homeomorphism $\vp: X\to Y$ is uniformly continuous.
\end{enumerate}
\end{prop}

\begin{proof}
Let $T$ be one of the indicated  order isomorphisms and let $\vp:X\to Y$ be the associated homeomorphism.  We may assume that $T0 = 0$. 
If $X$ is not proximally compact, there exist sequences $(x_n), (x'_n)$ in $X$ such that $(x_n)$ has no convergent subsequence and $0< d(x_n,x'_n) \to 0$.  
Set  $y_n = \vp(x_n)$ and $y'_n = \vp(x'_n)$.  
If $\vp$ is not uniformly continuous, we obtain such sequences with the additional property that  $\inf d(y_n,y'_n) > 0$.
Choose $f$, $r$ and $(n_m)$ as in Lemma \ref{lem56}.
Since $(T1(y_{n_m}))$ is bounded,  $d(y_{n_m},y'_{n_m}) \to 0$.
In particular, this yields a contradiction if $\vp$ is not uniformly continuous and hence completes the proof in case (b).
For case (a), we have
$Tf(y_{n_m}) = md(y_{n_m},y'_{n_m})$ for all sufficiently large $m$.
It follows that $Tf$ is not Lipschitz on $Y$, which is absurd.
\end{proof}

\begin{prop}\label{prop58.2}
Let $T$ be an order isomorphism from $A(X) = U(X)$ or $U_b(X)$ onto $A(Y) = U(Y)$ or $\lip_\al(Y)$.
Then the associated homeomorphism $\vp:X\to Y$ is uniformly continuous.
\end{prop}

\begin{proof}
We may assume that $T0 = 0$. Suppose that the proposition fails.  There exist sequences $(x_n)$ and $(x'_n)$ in $X$ such that $0 < d(x_n,x'_n)\to 0$ and $\inf d(y_n,y'_n) > 0$, where $y_n = \vp(x_n)$ and $y'_n = \vp(x'_n)$.
Since none of the sequences $(x_n)$, $(y_n)$ or $(y'_n)$ can have convergent subsequences, and $X$ and $Y$ are complete by assumption, we may assume that $(x_n)$ and $(y_n) \cup (y'_n)$ are separated.

\medskip

\noindent\underline{Case 1}. There exists $r> 0$ such that $(y_n) \subseteq Y_r$.

\noindent Let $r_n = d(y_n, Y\bs \{y_n\})$. Then $r_n \geq r$ for all $n$. 
First suppose that $A(Y) = U(Y)$.
There exists $g \in U(Y)$ such that $g(y_n ) = Tn(y_n)$ and $g(y'_n) = 0$ for all $n$.
Then $T^{-1}g(x_n) = n$ and $T^{-1}g(x'_n) = 0$ for all $n$.  Clearly, $T^{-1}g \notin U(X)$, contrary to the assumption.

Next, suppose that $A(Y) = \lip_\al(Y)$.
Choose $y_n'' \in Y$ such that $r_n \leq d(y_n,y_n'') < 2r_n$ and set $x_n'' = \vp^{-1}(y_n'')$.  We may assume that $y_n \neq y_m''$ for all $m$ and $n$.
Define $g: Y \to \R$ by $g(y_n) = r_n^\al$ for all $n$ and $g(y) = 0$ if $y \neq y_n$ for any $n$.
Then $g\in \lip_\al(Y)$. 
Hence $T^{-1}(mg) \in U(X)$ for all $m\in \N$.
As $T^{-1}(mg)(x'_n) = T^{-1}0(x'_n) = 0$ for all $n$ and $d(x_n,x'_n) \to 0$, $\lim_nT^{-1}(mg)(x_n) = 0$. Choose $n_1 < n_2 < \cdots$ such that $T^{-1}(mg)(x_{n_m}) \to 0$.
There exists $f\in U_b(X)$ such that $f(x_{n_m}) = T^{-1}(mg)(x_{n_m})$ and 
$f(x_{n_m}'') = 0$ for all $m$.
Then  $Tf(y_{n_m}) = mg(y_{n_m}) = mr^\al_{n_m}$ and $Tf(y_{n_m}'') = 0$ for all $m$.
However, $d(y_{n_m}, y''_{n_m}) <2r_{n_m}$.  This contradicts the fact that $Tf \in \lip_\al(Y)$.

\medskip

\noindent\underline{Case 2}. $(y_n) \not\subseteq Y_r$ for all $r > 0$.

\noindent By using a subsequence if necessary, we may assume that there exists $(y_n'')$ in $Y$ so that $0 < d(y_n,y_n'')\to 0$.  Set $x_n'' = \vp^{-1}(y''_n)$.
By Lemma \ref{lem56}, if there exists $0< t\in \R$ such that $(Tt(y_n))$ is bounded, then $\lim\inf d(y_n,y_n') = 0$, contrary to the choices of $y_n$ and $y_n'$.
Thus, by using a further subsequence if necessary, we may assume that $T1(y_n) \to \infty$.
Then there exists $0 < s\in \R$ such that $T1(y_n) \geq s$ for all $n$.
Hence $(T^{-1}s(x_n))$ is bounded above by $1$.
Applying Lemma \ref{lem56} to $T^{-1}$, we conclude that $\liminf d(x_n,x''_n) = 0$.
Since this applies to any subsequence, in fact $d(x_n,x''_n) \to 0$.
Observe that $\liminf d(y_n',y_n'') = \liminf d(y_n',y_n) > 0$.
So, by using even further subsequences if necessary, we can find $g, h \in A(Y)$ such that
\[ g(y_n) = h(y_n'') = 1, \ g(y_n') = h(y_n') = 0\]
for all $n$.  
Since $T^{-1}(mg), T^{-1}(mh)\in U(X)$, $T^{-1}(mg)(x'_n)=T^{-1}(mh)(x'_n) = 0$ and $d(x_n,x'_n), d(x''_n,x'_n)\to 0$, 
\[ \lim_n T^{-1}(mg)(x_n) = \lim_nT^{-1}(mh)(x_n'') = 0\]
 for all $m$.
Choose $n_1< n_2 < \cdots$ such that
\[ \lim_nT^{-1}(mg)(x_{n_m}) = \lim_nT^{-1}((m+1)h)(x_{n_m}'') = 0.\]
There exists $f\in U_b(X)$ such that $f(x_{n_m}) = T^{-1}(mg)(x_{n_m})$ and 
$f(x_{n_m}'') = T^{-1}((m+1)h)(x_{n_m}'')$ for all $m$.
Then $Tf(y_{n_m}) = mg(y_{n_m})= m$ and $Tf(y_{n_m}'') = (m+1)h(y_{n_m}) = m+1$.
This is impossible since $Tf \in U(Y)$ and $d(y_n,y_n'') \to 0$.
\end{proof}

\begin{thm}\label{thm48}
There exists $Y$ such that $U(X)\sim \Lip_b(Y)$ if and only if $X$ is a finite set.
There exists $Y$ such that $\lip_\al(X)$, $\lip_{\al,b}(X)$  or $U_b(X)\sim \Lip_b(Y)$ if and only if $X$ is separated.
\end{thm}

\begin{proof}
Let $T$ be one of the order isomorphisms indicated.  We may assume that $T0=0$. By Corollary \ref{cor35}, we find that  $U^\loc(X)$ or   $\lip^\loc_\al(X)\sim \Lip^\loc(Y)$.  By Theorems \ref{thm41}(a) and \ref{thm57}(a), $X$ is  discrete.  By Proposition \ref{prop57.1}(a), $X$ is proximally compact. Therefore, $X$ is separated by Proposition \ref{prop38.1}.
Conversely, if $X$ is separated then let $X^\al$ be the metric space $(X,d^\al)$.  We have
$\lip_\al(X) = \Lip(X^\al)\sim \Lip_b(Y)$ for some $Y$ by Theorem \ref{thm5.3} and $\lip_{\al,b}(X) = U_b(X) = \Lip_b(X)$, since all three coincide with $C_b(X)$.

Finally, if $U(X)\sim \Lip_b(Y)$, then $X$ is separated by the above.  Hence $C(X) = U(X) \sim \Lip_b(Y)$.
By Theorem \ref{thm41}(b), $Y$ is a finite set; hence so is $X$.  The converse is trivial.
\end{proof}

\begin{lem}\label{lem58.1}
Let $X$ and $Y$ be complete metric spaces.  Assume that $T:C(X)\to C(Y)$ is an order isomorphism such that $T0=0$.
Express $T$ in the form  $Tf(y) = \Phi(y,f(\vp^{-1}(y)))$, where $\vp:X\to Y$ is a homeomorphism and $\Phi:Y\times \R \to \R$ 
is a function such that $\Phi(y,\cdot):\R\to\R$ is an increasing homeomorphism for all $y\in Y$.
Let $(y_n)$, $(y'_n)$ be sequences in $Y$ so that $(y_n)$ is a sequence of distinct points and that $0 < d(y_n,y'_n) \to 0$.  Set $A(Y) = U(Y)$ or $\lip_\al(Y)$. If $T(U_b(X))\subseteq A(Y)$, then for all $0 < t\in \R$,
\[ \sup_n\frac{\Phi(y_n,t)}{d(y_n,y'_n)^\al} < \infty,\]
where we take $\al = 0$ if $A(Y) = U(Y)$.
\end{lem}

\begin{proof}
First we show that for any $t\in \R$ and any $\ep > 0$, there exists $\delta > 0$ so that
\[ \limsup_{n\to\infty}\frac{|\Phi(y_n,s) - \Phi(y_n',t)|}{d(y_n,y'_n)^\al} \leq \ep\]
whenever $|s-t| <\delta$.
If not, there exist $t\in \R$, $\ep > 0$, a sequence $(s_k)$ converging to $t$ and a subsequence $(y_{n_k})$ of $(y_n)$ so that
\begin{equation}\label{extraeq}
 \frac{|\Phi(y_{n_k},s_k) - \Phi(y_{n_k}',t)|}{d(y_{n_k},y'_{n_k})^\al} >\ep
\end{equation}
for all $k$. Set $x_n =\vp^{-1}(y_n)$ and $x'_n = \vp^{-1}(y_n')$ for all $n$.
By using a further subsequence if necessary, we may assume that either $(x_{n_k})$ has no convergent subsequence or that $(x_{n_k})$ converges to some $x_0 \neq x_{n_k}$ for any $k$.
In either case, using yet another subsequence, we may assume that $x'_{n_j} \neq x_{n_k}$ for all $j,k$.
Then there exists $f\in U_b(X)$ such that $f(x_{n_k}) = s_k$ and $f(x'_{n_k}) = t$ for all $k$.
Since $Tf(y_{n_k}) = \Phi(y_{n_k}, s_k)$ and $Tf(y'_{n_k}) = \Phi(y_{n_k}, t)$,
inequality (\ref{extraeq}) contradicts the fact that $Tf\in \lip_\al(X)$. 

From what was shown above, for all $t\in \R$, there exists $\delta_t > 0$ such that $|s-t| < \delta_t$ implies 
\[
\limsup_{n\to\infty}\frac{|\Phi(y_n,s) - \Phi(y_n',t)|}{d(y_n,y'_n)^\al} \leq 1.\]
If $|s-t|, |s'-t| < \delta_t$, then
\begin{align*}
&\limsup_{n\to\infty}\frac{|\Phi(y_n,s) - \Phi(y_n,s')|}{d(y_n,y'_n)^\al}\\ &\quad\leq 
\limsup_{n\to\infty}\frac{|\Phi(y_n,s) - \Phi(y_n',t)|}{d(y_n,y'_n)^\al} + \limsup_{n\to\infty}\frac{|\Phi(y_n,t) - \Phi(y_n',s')|}{d(y_n,y'_n)^\al} \leq 2.
\end{align*}
Fix $m \in \N$. By Lebesgue's Lemma \cite[Lemma 3.7.2]{M}, there exists $k\in \N$ such that for each $1\leq i \leq mk$, there exists $t\in [0,m]$ such that $(i-1)/k, i/k \in (t-\delta_t, t+\delta_t)$.
Therefore, keeping in mind that $\Phi(y_n,0) = T0(y_n) = 0$ for all $n$,
\begin{align*} 
\limsup_{n\to\infty}\frac{\Phi(y_n,m)}{d(y_n,y'_n)^\al}  & = 
\limsup_{n\to\infty}\frac{\Phi(y_n,m) - \Phi(y_n,0)}{d(y_n,y'_n)^\al} \\&\leq 
\limsup_{n\to\infty}\sum^{km}_{i=1}\frac{\Phi(y_n,\frac{i}{k}) - \Phi(y_n,\frac{i-1}{k})}{d(y_n,y'_n)^\al}\\
&\leq 2mk.
\end{align*}
Hence 
\[ \sup_n \frac{\Phi(y_n,m)}{d(y_n,y_n')^\al}< \infty.\]
For any $0< t \in \R$, choose $m\in \N$  such that $0 < t \leq m$,
then $0 < \Phi(y_n,t) \leq \Phi(y_n,m)$ for all $n$.  The desired conclusion follows.
\end{proof}

We say that a metric space $X$ with a distinguished point $e$ is {\em expansive} if $e$ is an isolated point and there exists $c > 0$ such that $d(p,q) \geq cd(p,e)$ if $p\neq q$.  The definition is independent of the distinguished point $e$. Clearly, if $X$ is expansive, then it is separated.

\begin{prop}\label{prop58}
Let $(X,d)$ be a metric space with a distinguished point $e$.
Recall the metric $d'$ on $X$ defined by equation (\ref{eq5.1}) and let $(X,d')$ be denoted by $X'$.
Then $X$ is expansive if and only if $X'$ is separated.
\end{prop}

\begin{proof}
Suppose that $X$ is expansive with constant $c$ given by the definition.
Since $e$ is an isolated point, there exists $r >0$ such that $d(x,e) \geq r$ for all $x\in X$, $x\neq e$.
Let $p\neq q\in X$ with $d(q,e) \geq d(p,e)$. In particular, $q \neq e$.
Hence $d(p,q) \geq cd(q,e) \geq cr$.
Therefore, $d(p,q) \geq (c\wedge cr)\xi(q)$.
By Proposition \ref{prop5.1}(b), $d'(p,q) \geq c\wedge cr$.
 Hence $X'$ is separated. 

Conversely, suppose that $X'$ is separated.
Choose $r> 0$ such that $d'(p,q) \geq r$ if $p\neq q \in X$.
By Proposition \ref{prop5.1}(b), for $p \neq q \in X$,
\[d(p,q) \geq \frac{d(p,q)}{\xi(p)\vee \xi(q)}\cdot d(p,e)\geq \frac{d'(p,q)}{3} \cdot d(p,e)\geq \frac{r}{3}\cdot d(p,e).\]
Finally, if $e$ is not an isolated point in $X$, then there exists $p \neq e$ such that $d(p,e) < \frac{r}{3} \wedge 1$.
Then $\xi(p) = \xi(e) = 1$.
By Proposition \ref{prop5.1}(b),
$d'(p,e) \leq 3d(p,e) < r$, contrary to the choice of $r$.
This completes the proof that $X$ is expansive.
\end{proof}

\begin{thm}\label{thm59}
\begin{enumerate}
\item There exists $Y$ such that one of the following holds if and only if $X$ is expansive.
\[
\lip_\al(X) \sim C_b(Y), U^\loc_b(Y), \Lip^\loc_b(Y), U_b(Y).\]
\item There exists $Y$ such that $\lip_{\al,b}(X) \sim U_b(Y)$ if and only if $X$ is separated.
\end{enumerate}
\end{thm}

\begin{proof}
By Corollary \ref{cor35}, the given order isomorphisms extend to order isomorphisms from $\lip^\loc_\al(X)$ onto $C(Y), U^\loc(Y)$ or $\Lip^\loc(Y)$.
By Theorems \ref{thm41}(a) and \ref{thm57}(a), $Y$ is discrete.  Since $X$ and $Y$ are homeomorphic, $X$ is also discrete.

For the first three cases in part (a), observe that  $C_b(X) \sim C_b(Y) = U^\loc_b(Y) = \Lip^\loc_b(Y)$.
So in all three cases, we arrive at the conclusion that $\lip_\al(X) \sim C_b(X)$.  Moreover, as $X$ is discrete, any permutation on $X$ is continuous.  Thus we may assume that there is an order isomorphism $T:  C_b(X)\to \lip_\al(X)$ whose associated homeomorphism is the identity map on $X$.

In the remaining two cases, by Propositions \ref{prop57.1}(b) and \ref{prop58.2}, the homeomorphism $\vp:X\to Y$ associated with the order isomorphism is a uniform homeomorphism.  It follows easily that there is an order isomorphism $T$ from $U_b(X)$ onto $\lip_\al(X)$ or $\lip_{\al,b}(X)$ whose associated homeomorphism is the identity map on $X$.

To continue with the proof, in all cases,
extend $T$ to an order isomorphism $S: C(X)\to C(X)$ by Proposition \ref{prop34}.  $S$ has the form $Sf(x) = \Phi(x,f(x))$. Since $1\in \lip_{\al,b}(X)$, $S^{-1}1 = T^{-1}1 \in C_b(X)$.
Choose $m\in \N$ such that $T^{-1}1 \leq m$.
Suppose that $X$ is not separated.  There are sequences $(x_n)$, $(x'_n)$ in $X$ such that $0 < d(x_n,x'_n) \to 0$.  Since $X$ is discrete, we may assume that $(x_n)$ is a sequence of distinct points.
By Lemma \ref{lem58.1}, 
\[ \sup_n\frac{\Phi(x_n,m)}{d(x_n,x'_n)^\al} < \infty.\]
But by choice of $m$, $\Phi(x_n,m) = Tm(x_n) \geq 1$ for all $n$.
As $d(x_n,x'_n) \to 0$, the absurdity of the final inequality above is evident.
This proves that $X$ is separated.

In case (a),
it follows that 
\[\Lip(X^\al) = \lip_\al(X) \sim C_b(X) = U_b(X) = C_b(X^\al),\]
where $X^\al$ denotes the metric space $(X,d^\al)$.  Let ${X^\al}'$ be the space $(X,d')$, with the metric $d'$ given by (\ref{eq5.1}), starting  with the original metric $d^\al$ instead of $d$.
By Theorem \ref{thm5.3}, $\Lip(X^\al) \sim \Lip_b({X^\al}')$.
Hence $\Lip_b({X^\al}') \sim C_b(X^\al)$.
It follows from Theorem \ref{thm41}(c) that ${X^\al}'$ is separated.
By Proposition \ref{prop58}, $X^\al$ is expansive.
It is easy to see that $X^\al$ is expansive if and only if $X$ is expansive.

Conversely, in case (a), if $X$ is expansive, then $X$ and ${X^\al}'$ are separated. Hence 
\begin{align*}
\lip_\al(X) = \Lip(X^\al) \sim \Lip_b({X^\al}') &= C_b({X^\al}') = U^\loc_b({X^\al}') \\&= \Lip^\loc_b({X^\al}') = U_b({X^\al}').
\end{align*}
In case (b), if $X$ is separated, then  $\lip_{\al,b}(X) = C_b(X) = U_b(X)$.
\end{proof}

\subsection{The cases $U(X)\sim \lip_\al(Y)$ or $\lip_{\al,b}(Y)$}

Let $A(Y) = \lip_\al(Y)$ or $\lip_{\al,b}(Y)$ and consider an order isomorphism $T:U(X) \to A(Y)$ such that $T0 = 0$.
Express $T$ in the form $Tf(y) = \Phi(y,f(\vp^{-1}(y)))$, where $\vp: X\to Y$ is a homeomorphism and $\Phi: Y\times \R \to \R$ is a function such that $\Phi(y,\cdot):\R\to \R$ is an increasing homeomorphism for all $y\in Y$.
By Propositions \ref{prop57.1}(b) and \ref{prop58.2}, $\vp$ is uniformly continuous.
By Corollary \ref{cor35}, $T$ can be extended to an order isomorphism from $U^\loc(X)$ onto $\lip^\loc_\al(Y)$.  Hence $X$ and $Y$ are discrete by Theorem \ref{thm41}(a).
We seek to show that $X$ is separated.
Suppose on the contrary that $X$ is not separated.
By Proposition \ref{prop38.1}, $X$ is not proximally compact.
Thus there is a sequence $(x_n)$ with no convergent subsequence and a sequence $(x'_n)$ in $X$ such that $0 < d(x_n,x'_n) \to 0$.
Since $X$ is complete by assumption, we may assume that $(x_n)$ is a separated sequence.
Set $y_n = \vp(x_n)$ and $y_n' = \vp(x_n')$ for all $n$.  Since $\vp$ is uniformly continuous, $d(y_n,y'_n)\to 0$.

\begin{lem}\label{lem6.3.1}
Let $(u_n)$ be a sequence in $X$ such that $0 < d(x_n,u_n)\to 0$. Set $v_n = \vp(x_n)$.
Then 
\[ 0 < \inf\frac{d(y_n,v_n)}{d(y_n,y_n')} \leq \sup\frac{d(y_n,v_n)}{d(y_n,y'_n)} < \infty.\]
\end{lem}

\begin{proof}
Suppose that the upper inequality fails.  By using a subsequence if necessary, we may assume that $d(y_n,v_n)/d(y_n,y'_n) \to \infty$.
Also, as $(y_n)$ has no convergent subsequence and $Y$ is complete by assumption, we may assume that $(y_n)$ is a separated sequence.
Uniform continuity of $\vp$ implies that $d(y_n,v_n)\to 0$. Thus, we may assume that $\inf_{m\neq n}d(\{y_m,v_m\},\{y_n,v_n\}) > 0$.
By Lemma \ref{lem58.1}, $\sup \Phi(y_n,1)/d(y_n,y_n')^\al < \infty$.
Hence $\Phi(y_n,1)/d(y_n,v_n)^\al \to 0$.
Therefore, there is a function $g\in \lip_{\al,b}(Y)$ such that $g(y_n) = \Phi(y_n,1)= T1(y_n)$ and $g(v_n) = 0$ for all $n$.
However, $T^{-1}g(x_n) = 1$ and $T^{-1}g(u_n) = 0$, making it impossible for $T^{-1}g$ to belong to $U(X)$.
The lower inequality is proved in the same way with the roles of $y_n'$ and $v_n$ switched.
\end{proof}

For any $x\in X$ and any $r> 0$, let $B'(x,r)$ be the deleted ball $\{z\in X: 0 < d(z,x) < r\}$.

\begin{lem}\label{lem6.3.2}
There exist $r> 0$ and $1 \leq C <\infty$ such that for any $u_n \in B'(x_n,r)$, $v_n = \vp(u_n)$, 
\[ \frac{1}{C}  \leq \frac{d(y_n,v_n)}{d(y_n,y_n')} \leq C \text{ for all $n$}.\]
\end{lem}

\begin{proof}
For any $m \in \N$, take $r = 1/m$ and $C = m$.  Suppose that the upper inequality is violated by some $u_m \in B'(x_{n_m},\frac{1}{m})$. Set $v_m = \vp(u_m)$.
Since $X$ is discrete, $(n_m)$ cannot contain a constant subsequence.
Without loss of generality, we may assume that $n_1 < n_2 < \cdots$.
 This contradicts Lemma \ref{lem6.3.1} applied to the subsequence $(n_m)$. 
Similarly, the lower inequality holds for some $r$ and $C$.
\end{proof}

Define the function $\Psi: X\times \R\to \R$ by $\Psi(x,t) = T^{-1}t(x)$.
It is easy to see that $\Psi(x,\cdot)$ is the inverse of $\Phi(\vp(x),\cdot)$.  In particular, $\Psi(x,\cdot): \R\to \R$ is an increasing homeomorphism.

\begin{lem}\label{lem6.3.3}
There exist $s, t > 0$ such that $\Phi(y_n,t) \geq sd(y_n,y'_n)^\al$ for all $n$.
\end{lem}

\begin{proof}
It suffices to show that there exists $s> 0$ such that the sequence $(\Psi(x_n,sd(y_n,y_n')^\al))$ is bounded above.
Otherwise, there are $n_1 < n_2 < \cdots$ and $0 < s_m \to 0$ such that
\[ T^{-1}c_m(x_{n_m}) = \Psi(x_{n_m}, c_m) \to \infty,\text{ where $c_m = s_md(y_{n_m}, y'_{n_m})^\al$}.\]
Note that $(y_n)$ has no convergent subsequence and $d(y_n,y'_n)\to 0$.
By using a further subsequence if necessary, we may assume that there exists  $g\in \lip_{\al,b}(Y)$ with $g(y_{n_m}) = c_m$ and $g(y_{n_m}') = 0$ for all $m$.
Then $T^{-1}g(x_{n_m}) = \Psi(x_{n_m},c_m)$ and $T^{-1}g(x_{n_m}') = 0$  for all $m$.
This contradicts the fact that $T^{-1}g\in U(X)$.
\end{proof}

Recall the following notation. For any $x\in X$ and $0 \leq r_1 < r_2$, take
\[ \Ann(x,r_1,r_2) = \{z\in X: r_1 < d(z,x) < r_2\}.\]
Let $r$ and $C$ be the numbers arrived at in Lemma \ref{lem6.3.2}.
Since $(x_n)$ is a separated sequence, we may  assume without loss of generality that $d(x_n,x_m) > 2r$ if $n\neq m$, and  that $d(x_n,x'_n) < r/2$ for all $n$.

\begin{lem}\label{lem6.3.4}
$\Ann(x_n,\frac{r}{2}, {r}) = \emptyset$ for infinitely many $n$.
\end{lem} 

\begin{proof}
If the lemma fails, for all sufficiently large $n$, there exists $u_n\in \Ann(x_n,\frac{r}{2},r)$.
Then $u_n$ and $x'_n$ are distinct points in $B'(x_n,r)$.
By choice of $r$ and $C$, 
\[\frac{1}{C}  \leq \frac{d(y_n,v_n)}{d(y_n,y_n')} \leq C \text{ for all $n$},\]
where $v_n = \vp(u_n)$.  In particular, $d(y_n,v_n)\to 0$.
By Lemma \ref{lem6.3.3}, there are $s, t >0$ so that $\Phi(y_n, t) \geq sd(y_n,y'_n)^\al$ for all $n$.
Since $d(x_n,u_n) > r/2$ for all sufficiently large $n$, there exists $f\in U(X)$ such that $f(x_n) = t$ and $f(u_n) = 0$ for all sufficiently large $n$.
Then $Tf\in \lip_\al(Y)$, $Tf(v_n) = 0$, and 
\[ Tf(y_n) = \Phi(y_n,t) \geq sd(y_n,y_n')^\al \geq \frac{s}{C^\al}d(y_n,v_n)^\al,\]
which is impossible.
\end{proof}

\begin{lem}\label{lem6.3.5}
If there exists $Y$ such that $U(X)\sim \lip_\al(Y)$ or $\lip_{\al,b}(Y)$, then $X$ is separated.
\end{lem}

\begin{proof}
If $X$ is not separated, then it follows from the lemmas above that there is a sequence $(x_n)$ in $X$ and $r > 0$ so that $d(x_n,x_m) > 2r$ if $n\neq m$ and $\Ann(x_n,\frac{r}{2},r) = \emptyset$ for all $n$.   Choose $t_n \in \R$ such that $Tt_n(y_n) = nd(y_n,y_1)^\al$ for all $n$, where $y_n = \vp(x_n)$.
Define $f:X\to \R$ by $f(x) = t_n$ if $x\in B(x_n,r)$ for some $n$ and $f(x) = 0$ otherwise.
Since $\Ann(x_n,\frac{r}{2},r) = \emptyset$, $f\in U(X)$.
However, $Tf(y_n) = nd(y_n,y_1)^\al$ for all $n$ and so $Tf \notin \lip_\al(X)$, contrary to the assumption.
\end{proof}

\begin{thm}\label{thm6.3.6}
There exists $Y$ such that $U(X)\sim \lip_\al(Y)$ or $\lip_{\al,b}(Y)$ if and only if $X$ is a finite set.
\end{thm}

\begin{proof}
Suppose that there exists $Y$ such that $U(X)\sim \lip_\al(Y)$ or $\lip_{\al,b}(Y)$.  By Lemma \ref{lem6.3.5}, $X$ is separated.  Hence $U(X) = C(X)\sim \lip_\al(Y)$ or $\lip_{\al,b}(Y)$.
By Theorem \ref{thm41}(b), $X$ is a finite set.
The converse is trivial.
\end{proof}

\subsection{The case $\lip_\al(X)\sim \lip^\loc_{\al,b}(Y)$.}

\begin{lem}\label{lem63}
If $\lip_\al(X)\sim \lip^\loc_{\al,b}(Y)$, then $X$ is proximally compact.
\end{lem}

\begin{proof}
An order isomorphism $T: \lip_\al(X)\to \lip^\loc_{\al,b}(Y)$ can be extended to an order isomorphism $S: \lip^\loc_\al(X)\to \lip^\loc_\al(Y)$ such that $S(\lip^\loc_{\al,b}(X)) \subseteq \lip^\loc_{\al,b}(Y)$,
according to Corollary \ref{cor35}.
Hence $\lip^\loc_{\al,b}(X) \subseteq \lip_\al(X)$.
Therefore, $\lip^\loc_{\al,b}(X) = \lip_{\al,b}(X)$.
By Theorem \ref{thm47}, $X$ is proximally compact.
\end{proof}

Say that a metric space $(X,d)$ with a distinguished point $e$ is {\em expansive at $\infty$} if there exists   $C< \infty$ such that $p = q$ if $d(p,e) \geq C$ and $d(p,q) < d(p,e)/C$.
A direct comparison of the definitions shows that any metric space that is expansive at $\infty$ is almost expansive at $\infty$.  The converse is not true, as evidenced by the subspace of $\R$ consisting of the points $2^n$ and $2^n+ n^{-1}$ for all $n\in \N$.
Let $X'$ be the space $X$ endowed with the metric $d'$ given by equation (\ref{eq5.1}) in \S \ref{sec5}.

\begin{prop}\label{prop64}
Let $(X,d)$ be a complete metric space with a distinguished point $e$.  Then $X'$ is proximally compact if  $X$ is proximally compact and expansive at $\infty$.
\end{prop}

\begin{proof}
Suppose that $X$ is proximally compact and expansive at $\infty$.
Assume that $X'$ is not proximally compact. There exist sequences $(x_n)$ and $(x'_n)$ in $X$ such that $(x_n)$ has no $d'$-convergent subsequence and that $0 < d'(x_n,x'_n)\to 0$.
By Proposition \ref{prop5.1}(b),
\begin{equation}\label{eq10}
 \frac{d(x_n,x'_n)}{\xi(x_n)\vee \xi(x'_n)} \leq d'(x_n,x'_n).
\end{equation}
If $(\xi(x_n))$ and $(\xi(x_n'))$ are both bounded, then $d(x_n,x'_n) \to 0$.
Since $X$ is proximally compact, we may assume that $(x_n)$ $d$-converges to some $x_0 \in X$.
By Proposition \ref{prop5.1}(b) again, $d'(x_n,x_0) \to 0$, contrary to the choice of $(x_n)$.
Thus, we may assume that $d(x_n,e) \to \infty$.
Let $C$ be the constant resulting from the fact that $X$ is expansive at $\infty$.
For all sufficiently large $n$, $d(x_n,x'_n) \geq d(x_n,e)/C = \xi(x_n)/C$.
Since $d'(x_n,x'_n)\to 0$, it follows from  inequality (\ref{eq10}) that $\xi(x'_n) \geq \xi(x_n)$ for all sufficiently large $n$ and $d(x_n,x'_n)/\xi(x'_n) \to 0$.
In particular, $\xi(x'_n)\to \infty$ and hence $\xi(x_n') = d(x_n',e)$ for all sufficiently large $n$.
Therefore,
\[ \xi(x'_n) \geq \xi(x_n) = d(x_n,e) \geq d(x'_n,e) - d(x_n,x'_n) = \xi(x'_n)\bigl(1 - \frac{d(x_n,x'_n)}{\xi(x'_n)}\bigr)\]
 for all sufficiently large $n$.
Hence $\xi(x_n)/\xi(x'_n) \to 1$.
Thus 
\[ 0 = \lim \frac{d(x_n,x'_n)}{\xi(x'_n)} =   \lim \frac{d(x_n,x'_n)}{\xi(x_n)} \geq \frac{1}{C} > 0,\]
which is clearly absurd. 
\end{proof}

\begin{thm}\label{thm65}
There exists $Y$ such that $\lip_\al(X)\sim \lip^\loc_{\al,b}(Y)$ if and only if $X$ is proximally compact and expansive at $\infty$.
\end{thm}

\begin{proof}
Suppose that $T: \lip_\al(X) \to \lip^\loc_{\al,b}(Y)$ is an order isomorphism such that $T0=0$.  Denote the associated homeomorphism by $\vp: X\to Y$.
By Lemma \ref{lem63}, $X$ is proximally compact.
Suppose that $X$ is not expansive at $\infty$. Fix a distinguished point $e\in X$.
There exist sequences $(x_n), (x'_n)$ in $X$ such that $d(x_n,e)\to \infty$ and 
\[ 0 < \frac{d(x_n,x'_n)}{d(x_n,e)} \to 0.\]
In particular, $d(x'_n,e) \to \infty$ as well.  Thus we may assume that the points in $(x_n) \cup(x'_n)$ are  distinct.
Set $y_n = \vp(x_n)$ and $y'_n = \vp(x'_n)$.
The function $h:[0,\infty)\to \R$ given by $h(t) = (t^\al-1)^+$ belongs to $\lip_\al[0,\infty)$.
Define $f:X\to \R$ by $f(x) = h(d(x,e))$.
Then $f\in \lip_\al(X)$ and hence $Tf\in \lip^\loc_{\al,b}(Y)$.
In particular, $(Tf(y_n))$ is a bounded sequence.
Since the points in $(y_n)\cup(y'_n)$ are  distinct, and neither $(y_n)$ nor $(y_n')$ has a convergent subsequence, there exists $g\in \lip^\loc_{\al,b}(Y)$ such that $g(y_n)  = Tf(y_n)$ and $g(y'_n) = 0$.
Now $T^{-1}g\in \lip_\al(X)$, $T^{-1}g(x_n) = f(x_n)$, $T^{-1}g(x'_n) = 0$.
So there is a finite constant $K$ such that 
\[ (d(x_n,e)^\al-1)^+ \leq Kd(x_n,x'_n)^\al \text{ for all $n$.}\]
This is not possible since $d(x_n,e) \to \infty$ and $d(x_n,x'_n)/d(x_n,e) \to 0$.
This shows that $X$ is expansive at $\infty$.

Conversely, suppose that  $X$ is proximally compact and  expansive at $\infty$.
By Proposition \ref{prop64},  $X'$ is proximally compact, where $X'$ is the space $X$ endowed with the metric $d'$ given by equation (\ref{eq5.1}) in \S \ref{sec5}.
Lemma \ref{lem45} yields that $\lip_{\al,b}(X') = \lip_{\al,b}^\loc(X')$.
Since $X$ is expansive at $\infty$, it is almost expansive at $\infty$.  By Proposition \ref{prop5.12}, $\lip_\al(X) \sim \lip_{\al,b}(X')$.
Thus $\lip_\al(X) \sim \lip^\loc_{\al,b}(X')$.
\end{proof}

\subsection{The case $\lip_\al(X) \sim \lip_{\al,b}(Y)$}

Proposition \ref{prop5.12} implies that $\lip_\al(X) \sim \lip_{\al,b}(Y)$ for a specific $Y$ if $X$ is almost expansive at $\infty$.
The aim of this part is to show conversely that if $\lip_\al(X)\sim \lip_{\al,b}(Y)$ for any $Y$, then $X$
is almost expansive at $\infty$.

\begin{thm}\label{thm66}
There exists $Y$ such that $\lip_\al(X) \sim \lip_{\al,b}(Y)$ if and only if $X$ is almost expansive at $\infty$.
\end{thm}

Let $T: \lip_\al(X) \to \lip_{\al,b}(Y)$ be an order isomorphism such that $T0=0$, and let $\vp:X\to Y$ be the associated homeomorphism. 
Suppose, if possible, that $X$ is not almost expansive at $\infty$.
There are sequences $(p_n)$ and $(q_n)$ in $X$ and $r > 0$ such that 
\[  d(p_n,p_1)\to \infty, \frac{d(p_n,q_n)}{d(p_n,p_1)} \to 0  \text{ and } d(p_n,q_n) >r \text{ for all $n$}.\]
Let $u_n = \vp(p_n)$ and $v_n = \vp(q_n)$.

\begin{lem}\label{lem78}
$d(u_n,v_n) \not\to 0$.
\end{lem}

\begin{proof}
Assume, if possible, that $d(u_n,v_n) \to 0$. Let  $r_n = d(p_n,q_n) > r$ for all $n$.  Since $d(p_n,e)\to \infty$, we may assume that $d(p_m,p_n) > 2r_m+2r_n$ if $m \neq n$.
The function $f: X\to \R$ defined by 
\[ f(x) = \begin{cases}
            (r_n^\al - d(x,p_n)^\al)^+ \wedge \frac{r_n^\al}{2} &\text{if $x\in B(p_n,r_n)$ for some $n$},\\
            0 &\text{otherwise},
           \end{cases} \]
belongs to $\lip_\al(X)$.
Thus $T(mf) \in \lip_{\al,b}(Y)$ for all $m \in \N$.
Since $mf(q_n) = 0$, $T(mf)(v_n) = 0$ for all $m$ and $n$.
Choose $n_1 < n_2 < \cdots$ such that 
\[ \frac{T(mf)(u_{n_m})}{d(u_{n_m},v_{n_m})^\al} =  \frac{T(mf)(u_{n_m})-T(mf)(v_{n_m})}{d(u_{n_m},v_{n_m})^\al}\to 0.\]
Since $(p_n)$ has no convergent subsequence, neither does $(u_n)$.
As $Y$ is complete by assumption, we may assume that $(u_n)$ is a separated sequence in $Y$.
Then there exists $g\in \lip_{\al,b}(Y)$ such that $g(u_{n_m}) = T(mf)(u_{n_m})$ and 
$g(v_{n_m}) = 0$ for all $m$.
However, $T^{-1}g(p_{n_m}) = mf(p_{n_m}) = md(p_{n_m},q_{n_m})^\al/2$ and $T^{-1}g(q_{n_m}) = 0$, contradicting the fact that $T^{-1}g \in \lip_\al(X)$.
\end{proof}

By taking further subsequence if necessary, we may assume that $X_0 = (p_n)\cup (q_n)$ and $Y_0 = (u_n)\cup (v_n)$ are both separated sets.  In particular, $X_0$ and $Y_0$ are complete metric spaces.
We will also assume that $d(p_{n+1},p_1) \geq 2d(p_n,p_1)\geq 2$ for all $n$.
Let $X_0^\al$ be the set $X_0$ endowed with the metric $d^\al$, and let $X_0'$ be the set  $X_0$ endowed with  the metric 
\[d'(p,q) = \sup \bigl|\frac{f(p)}{\xi(p)} - \frac{f(q)}{\xi(q)}\bigr|,\]
where $\xi(p) = 1 \vee d(p,p_1)^\al$  for $p\in X_0$ and the supremum is taken over all $f\in \Lip(X_0^\al)$ with $|f(p_1)| \leq 1$ and Lipschitz constant with respect to $d^\al$ at most $1$.
Note that $\xi \in \lip_\al(X)$.

\begin{lem}\label{lem79}
There exists an order isomorphism $S: C(X_0') \to C(Y_0)$ such that $S0 =0$, $S(C_b(X_0'))\subseteq C_b(Y_0)$ and $f\in \Lip_b(X_0')$ if $Sf \in C_{b}(Y_0)$ 
\end{lem}

\begin{proof}
By Proposition \ref{prop34}, there is an order isomorphism  $R:C(X_0)\to C(Y_0)$ such that $R(f_{|X_0}) = (Tf)_{|Y_0}$ for all $f\in \lip_\al(X)$.
By Theorem \ref{thm5.3}, the map  $Q: \Lip(X_0^\al) \to \Lip_b(X_0')$ given by $Qf = f/\xi$ is an  order isomorphism.
Apply  Proposition \ref{prop34} to extend $Q$ to an order isomorphism $Q': C(X_0^\al) \to C(X_0')$. Since $X_0$ is a separated metric space, so is $X_0^\al$.  Thus $C(X_0) = C(X_0^\al)$.
The map $S = R\circ(Q')^{-1}: C(X_0')\to C(Y_0)$ is an order isomorphism such that $S0 = 0$.
For any $t\in \R$, $t\xi \in \lip_\al(X)$. Hence 
\[R(t\xi_{|X_0}) = (Tt\xi)_{|Y_0} \in \lip_{\al,b}(Y_0)\subseteq C_b(Y_0).\]
If $f\in C_b(X_0')$, choose $0 < t\in \R$ such that $|f| \leq t$.
Then 
\[ -t\xi_{|X_0} = (Q')^{-1}(-t) \leq (Q')^{-1}f \leq (Q')^{-1}t = t\xi_{|X_0}.\]
It follows from the above that $Sf = R\circ (Q')^{-1}f$ is bounded.

Finally, suppose that $Sf \in C_b(Y_0)$.  Since $Y_0$ is separated, $Sf\in \Lip_b(Y_0)$.
By \cite[Theorem1.5.6]{W}, there exists $g\in \Lip_b(Y)$ such that $g_{|Y_0} = Sf$.
Now $g \in \lip_{\al,b}(Y)$ and hence $h = T^{-1}g \in \lip_\al(X)$.
By the choice of $R$, 
\[R(h_{|X_0})= (Th)_{|Y_0} = g_{|Y_0} = Sf.\]
Therefore, $f = S^{-1}R(h_{|X_0}) = Q'(h_{|X_0})$.  Since $h_{|X_0} \in \Lip(X_0^\al)$,  $f =Q'(h_{|X_0}) = Q(h_{|X_0}) \in \Lip_b(X_0')$.
\end{proof}

\begin{proof}[Proof of Theorem \ref{thm66}]
We first show that $X_0'$ is not proximally compact.
The sequences $(p_n)$ and $(q_n)$ are contained in $X_0'$.
By Proposition \ref{prop5.1}(b), for $m<  n$, 
\[ d'(p_n, p_m) \geq \frac{d(p_n,p_m)^\al}{\xi(p_n)} \geq \frac{(d(p_n,p_1) - d(p_m,p_1))^\al}{\xi(p_n)} \geq \frac{d(p_n,p_1)}{2^{\al}\xi(p_n)} = 2^{-\al}.\]
Thus $(p_n)$ has no $d'$-convergent subsequence.
On the other hand, by Proposition \ref{prop5.1}(b) again, 
\[ d'(p_n,q_n) \leq \frac{3d^\al(p_n,q_n)}{\xi(p_n)\vee \xi(q_n)} \leq 3\biggl(\frac{d(p_n,q_n)}{d(p_n,p_1)}\biggr)^\al\to 0. \]
This completes the proof that $X_0'$ is not proximally compact.

Apply Proposition \ref{prop41.1} to the order isomorphism $S: C(X_0') \to C(Y_0)$ obtained in Lemma \ref{lem79} to
find a function $f\in C(X_0')\bs U(X_0')$ such that $Sf \in \Lip^\loc_b(Y_0)$.
Since $Y_0$ is separated, $\Lip_b^\loc(Y_0) = C_b(Y_0)$.
But then it follows from Lemma \ref{lem79} that $f \in \Lip_b(X_0')$, contrary to the choice of $f$.
This concludes the proof that $X$ is almost expansive at $\infty$ if $\lip_\al(X) \sim\lip_{\al,b}(Y)$ for some $Y$.
The converse is a direct consequence of Proposition \ref{prop5.12}; see also the subsequent remark.
\end{proof}

\subsection{The case $U(X) \sim U_{b}(Y)$}

\begin{lem}\label{lem}\label{lem80}
If there exists $Y$ such that $U(X)\sim U_b(Y)$, then $U(X) = U_b(X)$.
\end{lem}

\begin{proof}
Suppose that there is an order isomorphism from $U(X)$ onto $U_b(Y)$ and let $\vp:X\to Y$ be the associated homeomorphism.
By Propositions \ref{prop57.1}(b) and \ref{prop58.2}, $\vp$ is a uniform homeomorphism.
It follows that there is an order isomorphism $T: U_b(X)\to U(X)$ whose associated homeomorphism is the identity map.
Let $\Phi:X\times \R \to \R$ be a function such that $\Phi(x,\cdot):\R\to \R$ is an increasing homeomorphism for all $x\in X$ and that $Tf(x) = \Phi(x,f(x))$.  We may assume that $T0=0$.
If $U(X) \neq U_b(X)$, there exists a function $0 \leq f\in U(X)$ and a sequence $(x_n)$ in $X$ such that $f(x_n) \to \infty$.
Clearly, $(x_n)$ has no convergent subsequence.  Since $X$ is complete by assumption, we may assume that the sequence $(x_n)$ is separated.

\medskip

\noindent\underline{Case 1}. There exists $\ep >0$ such that $(x_n) \subseteq X_\ep$.

\noindent In this case, the function $g:X\to \R$ defined by $g(x_n) = \Phi(x_n,n)$ and $g(x) = 0$ otherwise is uniformly continuous.  However, $T^{-1}g(x_n) =n$ for all $n$ and thus $T^{-1}g \notin U_b(X)$, contrary to the assumption.

\medskip

\noindent\underline{Case 2}. $(x_n)\not\subseteq X_\ep$ for any $\ep > 0$.

\noindent By using a subsequence if necessary, we may assume that there is a sequence $(x'_n)$ in $X$ such that $0 < d(x_n,x'_n) \to 0$. 
Extend $T$ to an order isomorphism $S:C(X)\to C(X)$ by Proposition \ref{prop34}.
By Lemma \ref{lem58.1}, for any $0< t\in \R$,
$\sup_n{\Phi(x_n,t)}< \infty$.
Since $T^{-1}f \in U_b(X)$, there exists $0< t\in \R$ such that $T^{-1}f\leq t$.
Then $f\leq Tt$.
In particular, $\sup_n f(x_n) \leq \sup_nTt(x_n) = \sup_n\Phi(x_n,t) < \infty$, contrary to the choice of $f$.
\end{proof}

Lemma \ref{lem80} reduces the problem of characterizing $X$ with $U(X)\sim U_b(Y)$ to characterizing $X$ with $U(X) = U_b(X)$. The latter question has been answered by O'Farrell \cite{O'F}.
Let us recall some terminology from \cite{O'F}.
Let $X$ be a metric space.  For any $\ep > 0$, define an equivalence relation on $X$ by 
$x\sim_\ep y$ if and only if there exist $x = x_0, x_1,\dots, x_n = y$ such that $d(x_{k-1},x_k)\leq \ep$, $1\leq k\leq n$.
The equivalence classes are called {\em $\ep$-step territories}.
If $x\sim_\ep y$, then the smallest $n$ in the definition above is denoted by $s_\ep(x,y)$.
An $\ep$-step territory $T$ is said to be {\em $\ep$-step-bounded} if $\sup_{x,y\in T}s_\ep(x,y)<\infty$.

\begin{thm}\label{thm81}\cite[Theorem 2.1]{O'F}
Let $X$ be a metric space.  Then $U(X) = U_b(X)$ if and only if for any $\ep >0$, $X$ has only finitely many $\ep$-step territories, each of which is $\ep$-step-bounded.
\end{thm}

Let us call a metric space satisfying the conditions of Theorem \ref{thm81} an {\em O'Farrell space}.
The following result is now immediate.

\begin{thm}\label{thm72}
There exists $Y$ such that $U(X)\sim U_b(Y)$ if and only if $X$ is an O'Farrell space.
\end{thm}

\subsection{More on spaces of differentiable functions}
In this part, we take up the remaining cases of comparing  spaces of differentiable functions.
Specifically, we consider the comparisons $C^p(\ol{X})$ or $C^p_b(\ol{X}) \sim C^q(Y)$, and 
$C^p(\ol{X})$, $C^p_b(\ol{X})$ or $C^p_b(X) \sim C^q_*(\ol{Y})$.
Recall that here $X$ and $Y$ are open sets in (possibly different) Banach spaces, and $\ol{X}$, respectively, $\ol{Y}$ are their closures in the respective ambient Banach spaces.  We also assume that all spaces considered separate points from closed sets.

\begin{prop}\label{prop6.40}
Let $U$ and $V$ be open sets in Banach  spaces $E$ and $F$ respectively and 
let $A(F) = C^q(F)$ or $C^q_*(F)$.  Assume that $A(F)$ separates points from closed sets. 
Suppose that there are a homeomorphism $\vp: U\to V$ and a function $\Psi:U\times \R\to\R$ such that $\Psi(x,\cdot):\R\to \R$ is an increasing homeomorphism with $\Psi(x,0)= 0$ for all $x\in U$.
Furthermore, suppose that for any $g\in A(F)$, the formula $Sg(x) = \Psi(x,g(\vp(x))), x\in U$, defines a function $Sg \in C^p(U)$.
Let $(x_n)$ be a sequence of  distinct points in $U$ with no convergent subsequence in $U$.
For each $n$, there exist
\begin{enumerate}
\item  an open neighborhood $U_n$ of $x_n$, such that $\ol{U_n} \subseteq U$, $\diam U_n \to 0$, $\ol{U_n}\cap \ol{U_m} = \emptyset$ if $m\neq n$ (closures taken in $E$),
\item a function $f_n \in C^p_b(E)$, such that $f_n(x) = 0$ if $x\notin U_n$, $\|f_n\|_\infty \leq 1/n$,
\item a function $g_n\in A(F)$ and a point $y_n' \in \vp(U_n)$, such that $Sg_n = {f_n}_{|U}$, $d(y'_n,\vp(x_n)) \to 0$ and  $\|Dg_n(y'_n)\| \to \infty$ .
\end{enumerate}
\end{prop}

\begin{proof}
It is clear that there is a sequence of sets $(U_n)$ satisfying condition (a).
Let $y_n = \vp(x_n)$ and $c_n> 0$ be such that $\Psi(x_n,c_n) = 1/n$.  
For any $c\in \R$, the constant function $c\in A(F)$ and hence $\Psi(x,c)$ lies is $C^p(U)$ as a function of $x\in U$.
In particular, $\Psi(x,c)$ is continuous in the variable $x$ for $x\in U$.
Choose $0 < r_n < (c_n\wedge 1)/n$ such that $B(y_n,r_n) \subseteq \vp(U_n)$ and that $\Psi(\vp^{-1}(y),c_n/2) \leq 1/n$ for all $y\in B(y_n,r_n)$.
Since $A(F)$ separates points from closed sets, there exists $g\in A(F)$ such that $0 \leq g \leq 1$, $g(0) > 0$ and $g(y) = 0$ if $\|y\| \geq 1$. In particular, there exists $z\in F$, $\|z\| <1$, such that $Dg(z) \neq 0$.
Define $g_n: F\to \R$ by 
\[ g_n(y) = \frac{c_n}{2}g\bigl(\frac{y-y_n}{r_n}\bigr).\]
Then $g_n \in A(F)$ and hence $f_n = Sg_n \in C^p(U)$.
If $x\in U$ and  $\vp(x) \notin B(y_n,r_n)$,
then ${g_n}(\vp(x)) = 0$ and hence $f_n(x) = 0$.
In particular, $f_n(x) = 0$ if $x\in U\bs U_n$.
Since $\ol{U_n} \subseteq U$, we may extend $f_n$ to a function in $C^p(E)$ by defining $f_n(x) = 0$ for all $x\notin U$.  We will continue to denote the extension by $f_n$.
Note that for $x\in U$ with $\vp(x) \in B(y_n,r_n)$, 
 \[ 0 \leq f_n(x) = \Psi(x,g_n(\vp(x))) \leq \Psi(x,\frac{c_n}{2}) \leq \frac{1}{n}.\]
Hence $0 \leq f_n(x) \leq 1/n$.  Combined with the above, we see that $0 \leq f_n \leq 1/n$ on $E$.
Thus $f_n$ satisfies the conditions in (b).
Finally, for each $n$, let $y_n' = y_n+ r_nz$. Then $y_n' \in B(y_n,r_n) \subseteq \vp(U_n)$ and $d(y_n', y_n) \to 0$.
Furthermore,
\[ \|Dg_n(y'_n)\| = \frac{c_n}{2r_n}\|Dg(z)\| \to \infty\]
since $Dg(z) \neq 0$ and $0 < r_n < c_n/n$.
\end{proof}

\begin{thm}\label{thm6.41}
There exist $Y$ and $q$ such that $C^p(\ol{X}) \sim C^q(Y)$ if and only if $X =E$, the ambient Banach space containing $X$.
\end{thm}

\begin{proof}
Suppose that there exist $Y$ and $q$ such that $C^p(\ol{X}) \sim C^q(Y)$.  To show that $X = E$, it suffices to show that $X= \ol{X}$.  Assume to the contrary that there exists $x_0 \in \ol{X}\bs X$.
Let $T: C^p(\ol{X})\to C^q(Y)$ be an order isomorphism such that $T0 = 0$.
By Theorem \ref{thm29}, we have a representation $Tf(y) = \Phi(y,f(\vp^{-1}(y)))$ for $f\in C^p(\ol{X})$ and $y\in Y$, where $\vp:\ol{X}\to Y$ is a homeomorphism and $\Phi:Y\times \R\to \R$ is a function such that $\Phi(y,\cdot):\R\to \R$ is an increasing homeomorphism with $\Phi(y,0) = 0$ for all $y\in Y$.
Set $U = X$ and $V = \vp(X)$.  Then $U$ and $V$ are open sets in the respective ambient Banach spaces $E$ and $F$ and $\vp$ is a homeomorphism from $U$ onto $V$.
For each $x\in U$, let $\Psi(x,\cdot)$ be the inverse of $\Phi(\vp(x),\cdot)$.
Then $\Psi(x,\cdot):\R\to \R$ is an increasing homeomorphism such that $\Psi(x,0) = 0$ for all $x\in U$.
By assumption, $C^q(Y)$ separates points from closed sets and hence so does $C^q(F)$.
If $g\in C^q(F)$, then $g_{|Y} \in C^q(Y)$ and hence $T^{-1}(g_{|Y}) \in C^p(\ol{X})$.
In particular, $Sg(x) = \Psi(x,g(\vp(x))) = T^{-1}(g_{|Y})(x)$ belongs to $C^p(U)$ as a function of $x\in U$.
Choose a sequence of  distinct points $(x_n)$ in $U$ that converges to $x_0$ in $E$.
Clearly, $(x_n)$ has no convergent subsequence in $U$.
Obtain sets $U_n$, functions $f_n$, $g_n$ and points $y'_n$ by applying Proposition \ref{prop6.40}. 
It is clear that the pointwise sum $f = \sum {f_n}_{|\ol{X}}$ defines a function in $C^p(\ol{X})$.
As $f = f_n$ on $U_n$, $Tf = g_n$ on the open set $\vp(U_n)$.
Since $y'_n \in \vp(U_n)$,  $\|D(Tf)(y_n')\|  = \|Dg_n(y_n')\| \to \infty$.
However, $(\vp(x_n))$ converges to $\vp(x_0)$ and $d(y_n',\vp(x_n))\to 0$. Therefore, $(y'_n)$ converges to $\vp(x_0)$.
As $Tf\in C^q(Y)$ and $\vp(x_0) \in Y$, we must have $D(Tf)(y'_n) \to D(Tf)(\vp(x_0))$.
This contradicts the fact that $\|D(Tf)(y'_n)\| \to \infty$ and  completes the proof that $X = E$.

Conversely, if $X = E$, then $C^p(\ol{X}) = C^p(E) = C^q(Y)$ if we take $q = p$ and $Y = E$. 
\end{proof}

\begin{cor}\label{cor6.43}
$C^p_b(\ol{X}) \not\sim C^q(Y)$ for any $p, q$ and any $X, Y$.
\end{cor}

\begin{proof}
Let $A(X) = C^p_b(\ol{X})$ and $A(Y) = C^q(Y)$.  Then $A^\loc(\ol{X}) = C^p(\ol{X})$ and $A^\loc(Y) = C^q(Y)$.  By Corollary \ref{cor35}, if $C^p_b(\ol{X})\sim C^q(Y)$, then $C^p(\ol{X}) \sim C^q(Y)$. By Theorem \ref{thm6.41}, $\ol{X} = X$.  Thus $C^p_b(X) = C^p_b(\ol{X}) \sim C^q(Y)$.  By Corollary \ref{cor36.2}(b), this is impossible.
\end{proof}

\begin{thm}\label{thm6.43}
Let $A(\ol{X}) = C^p(\ol{X})$, $C^p_b(\ol{X})$.  Then $A(\ol{X})\not\sim C^q_*(\ol{Y})$ and $C^p_b(X) \not\sim C^q_*(\ol{Y})$ for any $q$ and any $Y$.
\end{thm}

\begin{proof}
Let $T$ be an order isomorphism from either $A(\ol{X})$ or $C^p_b(X)$ onto $C^q_*(\ol{Y})$.  We may assume that $T0 = 0$.  Denote by $\vp$ the associated homeomorphism.
Set $U = X \cap \vp^{-1}(Y)$ and $V = \vp(X) \cap Y$, which are open sets  in the ambient Banach spaces $E$ and $F$ respectively.
By Theorem \ref{thm29} (note in particular Example C part (f)), we have a representation $Tf(y) = \Phi(y,f(\vp^{-1}(y)))$ for $f\in A(\ol{X})$, respectively $C^p_b(X)$, and $y\in Y$, where  $\Phi:Y\times \R\to \R$ is a function such that $\Phi(y,\cdot):\R\to \R$ is an increasing homeomorphism with $\Phi(y,0) = 0$ for all $y\in Y$.
For each $x\in U$, let $\Psi(x,\cdot)$ be the inverse of $\Phi(\vp(x),\cdot)$.
Then $\Psi(x,\cdot):\R\to \R$ is an increasing homeomorphism such that $\Psi(x,0) = 0$ for all $x\in U$.
If $g\in C^q_*(F)$, then $g_{|\ol{Y}} \in C^q_*(\ol{Y})$.  It is easy to check that 
\[ Sg(x) = \Psi(x,g(\vp(x))) = T^{-1}(g_{|\ol{Y}})(x), \text{ $x\in U$}.\]
Hence $Sg \in C^p(U)$.
Since $U$ is a dense  open subset of an open set $X$ in a Banach space, it is clear that there is a sequence of  distinct points $(x_n)$ in $U$ with no convergent subsequence in $X$.
Obtain sets $U_n$, functions $f_n$, $g_n$ and points $y'_n$ by applying Proposition \ref{prop6.40}. 
Let $f:E\to \R$ be the pointwise sum $\sum f_n$.  It is easy to see that $f_{|\ol{X}} \in C^p_b(\ol{X})$ and that  $f_{|X} \in C^p_b(X)$.  Hence $T(f_{|\ol{X}})$, respectively, $T(f_{|X}) \in C^q_*(\ol{Y})$.
However, as in the proof of Theorem \ref{thm6.41}, $\|D(T(f_{|\ol{X}}))(y_n')\| \to \infty$, respectively $\|D(T(f_{|{X}}))(y_n')\| \to \infty$.  This contradiction concludes the proof of the theorem.
\end{proof}

\section{Spaces of little Lipschitz functions}

In this section, we return to the study of spaces of little Lipschitz functions.  The main aim is to establish the counterpart of Theorem \ref{thm5.6} for spaces of the type $\lip_\al(X)$. A special case of the result we intend to prove was established in \cite{CC1}.  The space $\lip(X)$ is said to {\em separate points boundedly} 
if and only if there is a constant $K <\infty$ such that for any $x,y\in X$, there exists $f\in \lip(X)$ such that $L(f) \leq K$ and $|f(x)-f(y)| = d(x,y)$.  Here $L(f)$ is the Lipschitz constant of $f$. Since $\Lip_b(X)\subseteq \lip_\al(X)$ if $0 < \al < 1$, 
it is clear that  $\lip_\al(X)$ separates points boundedly for any metric space $X$ and any $0 < \al < 1$.

\begin{thm}\label{thm7.1}\cite[Theorem 2]{CC1}
Let $X$ and $Y$ be compact metric spaces such that $\lip(X)$ and $\lip(Y)$ separate points boundedly.  If $T:\lip(X)\to\lip(Y)$ is an order isomorphism, then there are a Lipschitz homeomorphism $\vp:X\to Y$ and a function $\Phi:Y\times \R\to \R$ such that $\Phi(y,\cdot):\R\to \R$ is an increasing homeomorphism for each $y\in Y$ and that $Tf(y) = \Phi(y,f(\vp^{-1}(y)))$ for all $f\in \lip(X)$ and all $y\in Y$.
\end{thm}

We set the stage by proving some extension results for little Lipschitz functions, which may be of independent interest.  For results of a similar nature, refer to \cite[Section 3.2]{W}.  Note however that the results there concern exclusively compact metric spaces.
In what follows, fix a complete metric space $(X,d)$ and $0 < \al < 1$.

\begin{prop}\label{lip1}
Let $X_0$ be a separated subset of $X$.  If $f: X_0\to \R$ is Lipschitz with respect to $d^\al$, then there exists $g\in \lip_\al(X)$ such that $g_{|X_0} = f$.
\end{prop}

\begin{proof}
Choose $r > 0$ such that $d(x,x') > r$ if $x$ and $x'$ are distinct points in $X_0$.
It suffices to prove the proposition for nonnegative $f$; otherwise, consider $f^+$ and $f^-$ separately.
Let $C$ be such that $|f(x) -f(x')| \leq Cd(x,x')^\al$ for all $x, x'\in X_0$.
Define $g: X\to \R$ by 
\[ g(x) = \inf_{z\in X_0}[f(z) + 2C(d(x,z)^\al -\frac{r^\al}{2})^+].\]
Since $f\geq 0$, $g$ is nonnegative.
Suppose that $x, z,$ are distinct points in $X_0$.  Then $t= d(x,z) > r$ and hence $t^\al < 2(t^\al -\frac{r^\al}{2})$.  Thus
\[ f(z)+ 2C(d(x,z)^\al - \frac{r^\al}{2})^+ > f(z) +Cd(x,z)^\al \geq f(x).\]
It follows that $g(x) = f(x)$ for all $x\in X_0$.

Fix $\ep > 0$.  Since the function $h:[0,\infty)\to \R$, $h(a) = (a^\al - \frac{r^\al}{2})^+$ belongs to $\lip_\al[0,\infty)$, there exist $K < \infty$ and $\delta > 0$ such that $|h(a) - h(b)| \leq K|a-b|^\al$ for all $a, b\in [0,\infty)$ and $|h(a)-h(b)| \leq \ep|a-b|^\al$ if $|a-b| <\delta$.
Let $x_1$ and $x_2$ be distinct points in $X$.  Without loss of generality, suppose that $g(x_1) \leq g(x_2)$.
Choose $z\in X_0$ such that 
\[
 f(z) + 2C(d(x_1,z)^\al -\frac{r^\al}{2})^+ \leq g(x_1) + \ep d(x_1,x_2)^\al.
\]
By definition, $g(x_2) \leq  f(z) + 2C(d(x_2,z)^\al -\frac{r^\al}{2})^+$.
Thus
\[
0 \leq g(x_2) -g(x_1) \leq   
 2C[(t^\al - \frac{r^\al}{2})^+ -(s^\al - \frac{r^\al}{2})^+] + \ep d(x_1,x_2)^\al,
\]
where $t = d(x_2,z)$ and $s = d(x_1,z)$.
Therefore,
\begin{align*}
0 \leq g(x_2) -g(x_1) & \leq 2C|h(t) -h(s)| +\ep d(x_1,x_2)^\al \\
&\leq 2CK|t-s|^\al + \ep d(x_1,x_2)^\al \\
&\leq (2CK + \ep) d(x_1,x_2)^\al.
\end{align*}
This shows that $g\in \Lip(X,d^\al)$.
If, in addition, $d(x_1,x_2)< \delta$, then $|t-s| < \delta$ and hence $|h(t)-h(s)| < \ep|t-s|^\al$.
By the calculation above,
\begin{align*}
0 \leq g(x_2) -g(x_1) & \leq 2C|h(t) -h(s)| +\ep d(x_1,x_2)^\al \\
&\leq 2C\ep|t-s|^\al + \ep d(x_1,x_2)^\al \\
&\leq (2C+1)\ep d(x_1,x_2)^\al.
\end{align*}
Hence $g\in \lip_\al(X)$.
\end{proof}

\begin{prop}\label{lip1.1}
Let $(x_n)$ be a sequence in $X$ and let $(r_n)$, $(a_n)$ be real sequences such that $r_n > 0$, $d(x_m,x_n) \geq r_m + r_n$ if $m\neq n$, and $a_n/r_n^\al\to 0$.  Then there exists $f\in\lip_\al(X)$ such that $f(x_n) =a_n$ for all $n$ and $f(x) = 0$ if $x\notin \cup B(x_n,r_n)$.
\end{prop} 

\begin{proof}
The function $h(t) = (1 - (2t^\al-1)^+)^+$ belongs to $\lip_\al[0,\infty)$; $h(0) = 1$ and $h(t) =0$ if $t \geq 1$.
Define $f: X\to \R$ by taking the pointwise sum 
\[ f(x) = \sum a_nh\bigl(\frac{2d(x,x_n)}{r_n}\bigr).\]
If $d(x,x_n) \geq r_n/2$, then $h(2d(x,x_n)/r_n) = 0$. 
It follows that at each $x$, at most one of the terms in the sum is nonzero.
Clearly $f(x_n) = a_n$ and $f(x) = 0$ if $x\notin \cup B(x_n,r_n)$.
Since $h\in \lip_\al[0,\infty)$, there is a bounded function $\eta:[0,\infty)\to\R$, continuous at $0$ with $\eta(0) = 0$ such that 
\[|h(t) - h(s)| \leq \eta(|t-s|)|t-s|^\al \text{\ for all $t,s \geq 0$}.\]
Consider $x, x' \in X$.  We divide the proof into 2 cases.

\bigskip

\noindent\underline{Case 1}. 
There exists $n$ such that one of 
$x$ or $x'$ lies in $B(x_n,r_n/2)$ and the other is not in $B(x_m,r_m/2)$ for any $m\neq n$.

\noindent In this case,
$f(x) = a_nh(2d(x,x_n)/r_n)$ and $f(x') = a_nh(2d(x',x_n)/r_n)$. 
We have
\begin{align*}
|f(x) - f(x')| & \leq |a_n| |h\bigl(\frac{2d(x,x_n)}{r_n}\bigr) - h\bigl(\frac{2d(x',x_n)}{r_n}\bigr)|\\
&\leq |a_n|\eta\bigl(\frac{2}{r_n}|d(x,x_n) - d(x',x_n)|\bigr)\frac{2^\al}{r_n^\al}|d(x,x_n) - d(x',x_n)|^\al\\
&\leq |a_n|\eta\bigl(\frac{2}{r_n}|d(x,x_n) - d(x',x_n)|\bigr)\frac{2^\al}{r_n^\al}d(x,x')^\al\\
&\leq 2^\al\sup_m\frac{|a_m|}{r^\al_m}\|\eta\|_\infty d(x,x')^\al.
\end{align*}
In particular, there is a constant $C$, independent of $x$ and $x'$, such that $|f(x) - f(x')| \leq Cd(x,x')^\al$.
Let $\ep > 0$ be given.  There exists $N\in \N$ such that $|a_n/r_n^\al| < \ep$ if $n \geq N$.
Then choose $\delta > 0$ such that $\eta(2w/r_n) < \ep$ if $0 \leq w < \delta$, $1 \leq n < N$.
If $n \geq N$, we have by the preceding calculation that 
\[ |f(x) - f(x')|\leq 2^\al\ep\|\eta\|_\infty d(x,x')^\al.\]
On the other hand, if $1\leq n < N$ and $d(x,x')< \delta$, then $|d(x,x_n) - d(x',x_n)|< \delta$.
Hence $\eta({2}|d(x,x_n) - d(x',x_n)|/r_n) < \ep$.  Thus, from the preceding calculation, we obtain
\[ |f(x) - f(x')| \leq 2^\al\sup_m\frac{|a_m|}{r^\al_m}\ep d(x,x')^\al.\]

\bigskip

\noindent\underline{Case 2}. There exist $m\neq n$ such that $x\in B(x_m,r_m/2)$ and $x'\in B(x_n,r_n/2)$.

\noindent In this case, $d(x,x') \geq (r_n+r_m)/2$. We have
\begin{align*}
|f(x) - f(x')| & \leq |a_m| |h\bigl(\frac{2d(x,x_m)}{r_m}\bigr) - h(1)| + |a_n| |h\bigl(\frac{2d(x',x_n)}{r_n}\bigr) - h(1)|\\
&\leq |a_m|\eta\bigl(\bigl|\frac{2d(x,x_m)}{r_m} - 1\bigr|\bigr)\bigl|\frac{2}{r_m}d(x,x_m)-1\bigr|^\al+ \\
&\quad\ +|a_n|\eta\bigl(\bigl|\frac{2d(x',x_n)}{r_n} - 1\bigr|\bigr)\bigl|\frac{2}{r_n}d(x',x_n)-1\bigr|^\al\\
&\leq 2^\al\|\eta\|_\infty[\frac{|a_m|}{r_m^\al}|d(x,x_m)-\frac{r_m}{2}|^\al + \frac{|a_n|}{r_n^\al}|d(x',x_n)-\frac{r_n}{2}|^\al]\\
&\leq  2^\al\|\eta\|_\infty[\frac{|a_m|}{r_m^\al}\bigl(\frac{r_m}{2}\bigr)^\al + \frac{|a_n|}{r_n^\al}\bigl(\frac{r_n}{2}\bigr)^\al]\\
&\leq  2^\al\|\eta\|_\infty[\frac{|a_m|}{r_m^\al}+ \frac{|a_n|}{r_n^\al}]d(x,x')^\al.
\end{align*}
In particular,  there is a constant $C$, independent of $x$ and $x'$, such that $|f(x) - f(x')| \leq Cd(x,x')^\al$.
Given $\ep > 0$, choose $N\in \N$ such that $|a_n/r_n^\al| < \ep$ if $n \geq N$.
Set $\delta = \min_{1\leq k<N}r_k/2$.  Then $r_n/2+ r_m/2 < \delta$ implies that $n,m \geq N$.
Thus, if $d(x,x')< \delta$, then
\[ |f(x) - f(x')| \leq  2^\al\|\eta\|_\infty 2\ep d(x,x')^\al.\]
This completes the proof that $f\in \lip_\al(X)$.
\end{proof}

\begin{cor}\label{lip2}
Let $(x_n)$ be a sequence converging to $x_0$ in $X$. Assume that $d(x_{n+1},x_0) \leq d(x_n,x_0)/2$ for all $n$.
Set $X_0 = (x_n)$.  If $f:X_0\to \R$ belongs to $\lip_\al(X_0)$, then there exists $g\in \lip_\al(X)$ such that $g_{|X_0} = f$.
\end{cor}

\begin{proof}
Let $r_n  = d(x_n,x_0)/3$.  If $r_n = 0$ for some $n$, then $(x_n)$ is eventually constant and hence $X_0$ is separated.  The corollary follows from Proposition \ref{lip1}.  Assume that $r_n > 0$ for all $n$. By the assumption, $d(x_m,x_n)\geq r_m + r_n$ if $m \neq n$.
Suppose that $f\in \lip_\al(X_0)$.  
Then $a = \lim f(x_n)$ exists.  Furthermore, if 
$a_n = f(x_n)-a$, then $a_n/r_n^\al \to 0$.
By Proposition \ref{lip1.1}, there exists $h\in \lip_\al(X)$ such that $h(x_n) = a_n$ for all $n$.
Set $g = h + a$.  Then $g\in \lip_\al(X)$ and $g_{|X_0} = f$.
\end{proof}

For the sake of brevity, a sequence $(x_n)$ satisfying the hypothesis of Corollary \ref{lip2} will be said to {\em converge rapidly (to $x_0$)}.  Every convergent sequence has a subsequence that converges rapidly.

\begin{prop}\label{lip2.1}
Let $(x_n)$ and $(x'_n)$ be two sequences in $X$ such that $x_n \neq x'_n$ for all $n$.  
There is a subsequence $(n_k)$ of $\N$ such that, taking $X_0 = (x_{n_k}) \cup (x'_{n_k})$, every $f\in \lip_\al(X_0)$ extends to a function $g\in \lip_\al(X)$.
\end{prop}

\begin{proof}
Since $X$ is assumed to be complete, every sequence in $X$ has a subsequence that either converges or is separated.
Thus, by considering subsequences, and taking note of the symmetry between $(x_n)$ and $(x'_n)$, we may assume that we are in one of the following situations.
\begin{enumerate}
\item The sequence $(x_n)$ is separated and $d(x_n,x'_n)\not\to 0$.
\item The sequence $(x_n)$ is separated and $d(x_n,x'_n)\to 0$.
\item The sequences $(x_n)$ and $(x'_n)$ converge rapidly to $x_0$ and $x_0'$ respectively, $x_0 \neq x'_0$.
\item The sequences $(x_n)$ and $(x'_n)$ converge rapidly to the same limit $x_0$.
\end{enumerate}
In case (a), if $(x'_n)$ is also separated, then by taking a further subsequence, we may assume that $(x_n) \cup(x'_n)$ is separated.  The desired result  follows from Proposition \ref{lip1}.
Next, suppose that we are either in case (a) with $(x'_n)$ having no separated subsequence; or that we are in case (c).
In either of these situations, we may assume that $(x'_n)$ converges rapidly to some $x'_0$.
Set $r_n = d(x'_n,x_0')/3$. Then $d(x_m,x_n) \geq r_m + r_n$ if $m\neq n$.    
Furthermore, we may assume that $x_n \notin \cup_mB(x'_m,r_m)$ for all $n\in \N$ and that $d(x_0,x_0')/2 > d(x_1,x_0)$ if $x_0$ is defined (i.e., in case (c)).  Let $X_0 = (x_n) \cup (x'_n)$ and let $f\in \lip_\al(X_0)$.  We may extend $f$ by continuity to $x'_0$ and the resulting function will be in $\lip_\al(X_0\cup \{x_0'\})$.    
Since $\{x_0'\}\cup (x_n)$ is either separated or converges rapidly, by Proposition \ref{lip1} or Corollary \ref{lip2}, there exists $g_1 \in \lip_\al(X)$ such that $g_1(x_n) = f(x_n)$ for all $n\in \N$ and $g_1(x_0') = f(x_0')$.
Since 
\[ \frac{(f-g_1)(x_n')}{r_n^\al} = 3^\al\frac{(f-g_1)(x_n') - (f-g_1)(x_0')}{d(x_n',x_0')^\al}\to 0,\]
by Proposition \ref{lip1.1}, there exists $g_2\in \lip_\al(X)$ such that $g_2(x'_n) = (f-g_1)(x'_n)$ and $g_2(x) = 0$ if $x\notin \cup_mB(x'_m,r_m)$.  In particular, $g_2(x_n) = 0$ for all $n \in \N$.  Now $g = g_1 + g_2 \in \lip_\al(X)$ and $g_{|X_0} = f$.

The proof for case (b) is similar.  Let $r_n = d(x_n,x_n')$.  We may assume that $d(x'_m,x'_n) \geq r_m + r_n$ if $m\neq n$ and that $x_n \notin \cup_mB(x'_m,r_m)$ for all $n$.    Let $f\in \lip_\al(X_0)$, where $X_0 = (x_n)\cup(x'_n)$.
By Proposition \ref{lip1}, there exists $g_1\in \lip_\al(X)$ such that $g_1(x_n) = f(x_n)$ for all $n$. 
Since $f-g_1\in \lip_\al(X_0)$ and $(f-g_1)(x_n) = 0$ for all $n$, $(f-g_1)(x'_n)/r_n^\al \to 0$.   By Proposition \ref{lip1.1}, there exists $g_2\in \lip_\al(X)$ such that $g_2(x'_n) = (f-g_1)(x'_n)$ and $g_2(x) = 0$ for all $x\notin \cup B(x'_m,r_m)$.  Then $g = g_1 + g_2\in \lip_\al(X)$ and $g_{|X_0} = f$.

Finally, consider case (d).  If at least one of  $(x_n)$ and $(x'_n)$ is eventually constant,  then we may reduce the result to Corollary \ref{lip2}. Otherwise, we may assume that both $r_n= d(x_n,x_0)$ and $r_n' = d(x_n',x_0)$ are nonzero for all $n$.  By taking further subsequences,  we may suppose that $r_{n+1} \leq r'_{n+1}\leq r_n/2 \leq r'_n/2$ for all $n$.  Then $s_n = d(x_n,x'_n) \leq r_n + r'_n \leq 2r'_n$ for all $n$. It is easy to check that $d(x_m',x_n') \geq s_m/6 + s_n/6$ if $m\neq n$ and $x_n \notin \cup B(x'_m,s_m/6)$ for all $n$.
Let $f\in \lip_\al(X_0)$, where $X_0  = (x_n)\cup(x'_n)$.  Extend $f$ by continuity to $x_0$.  The extension belongs to $\lip_\al(X_0 \cup\{x_0\})$.
By Corollary \ref{lip2}, there exists $g_1\in \lip_\al(X)$ such that $g_1(x_n) = f(x_n)$ for all $n$.
Now
\[ \frac{(f-g_1)(x_n')}{s_n^\al} = \frac{(f-g_1)(x'_n) - (f-g_1)(x_n)}{s_n^\al} \to 0.\]
Hence, by Proposition \ref{lip1.1}, there exists $g_2\in \lip_\al(X)$ such that $g_2(x'_n) = (f-g_1)(x'_n)$ and $g_2(x) = 0$ for all $x\notin \cup B(x'_m,s_m/6)$.  In particular, $g_2(x_n) = 0$ for all $n\in \N$. Then $g = g_1 + g_2\in \lip_\al(X)$ and $g_{|X_0} = f$.
\end{proof}

\noindent{\bf Remark}.  From the proof of Proposition \ref{lip2.1}, one can see that it is possible to choose a subsequence $(n_k)$ so that the conclusion applies to any further subsequence of $(n_k)$.  Furthermore, if $f\in \lip_\al(\ol{X_0})$, then the proposition yields a $g\in \lip_\al(X)$ such that  $g_{|X_0} = f_{|X_0}$.  By continuity, $g_{|\ol{X_0}} = f$.

\medskip

\begin{prop}\label{lip3}
Let $A(X)$ and $A(Y)$ be sets of real-valued functions defined on the underlying sets $X$ and $Y$ respectively. Suppose
that there is an order isomorphism $T: A(X)\to A(Y)$ given by the formula $Tf(y) = \Phi(y,f(\vp^{-1}(y)))$, where  $\vp: X\to Y$ is a bijection  and $\Phi:Y\times \R\to \R$ is a function so that $\Phi(y,\cdot):\R\to \R$ is an increasing homeomorphism for each $y\in Y$.
Let $X_0$ be a subset of $X$ and let $Y_0 = \vp(X_0)$ and set
\[ A(X_0) = \{f_{|X_0}: f\in A(X)\}, \quad A(Y_0) = \{g_{|Y_0}: g\in A(Y)\}.\]
For any $f\in A(X_0)$, define $T_0f: Y_0\to \R$ by $T_0f(y) = \Phi(y,f(\vp^{-1}(y)))$ for all $y\in Y_0$.  Then $T_0f\in A(Y_0)$ and $T_0$ is an order isomorphism from $A(X_0)$ onto $A(Y_0)$.
\end{prop}

\begin{proof}
Let $f\in A(X_0)$.  Then $f = \ti{f}_{|X_0}$ for some $\ti{f} \in A(X)$.
Then $g = (T\ti{f})_{|Y_0} \in A(Y_0)$.
Obviously, $g(y) = \Phi(y,f(\vp^{-1}(y)))$ for all $y \in Y_0$.
Thus $T_0f\in A(Y_0)$.
For each $x\in X$, let $\Psi(x,\cdot):\R\to \R$ be the inverse of $\Phi(\vp(x),\cdot)$.
Then $T^{-1}g(x) = \Psi(x,g(\vp(x)))$ for all $g\in A(Y)$ and all $x\in X$.
In the same manner, we see that the formula $S_0g(x) = \Psi(x,g(\vp(x)))$ for all $g\in A(Y_0)$ and all $x\in X_0$
defines a map $S_0$ from $A(Y_0)$ into $A(X_0)$.  Obviously, $T_0$ and $S_0$ are mutual inverses, and both are order preserving. Hence $T_0$ is an order isomorphism.  
\end{proof}

For the remainder of the section, suppose that $X$ and $Y$ are complete metric spaces and that $T: \lip_\al(X) \to \lip_\al(Y)$ is an order isomorphism such that $T0 = 0$.
Express $T$ as $Tf(y) = \Phi(y,f(\vp^{-1}(y)))$ for some homeomorphism $\vp: X\to Y$ and a function $\Phi:Y\times \R \to \R$ such that $\Phi(y,\cdot): \R\to \R$ is an increasing homeomorphism for all $y\in Y$.
We seek to extract information on the homeomorphism $\vp$. The key idea is that the extension results for little Lipschitz functions proved above lead to a restriction of $T$ to functions defined on subspaces of $X$ and $Y$ respectively.

\begin{cor}\label{cor7.6}
Let $X_0$ be a compact subset of $X$ and let $Y_0 = \vp(X_0)$.
Suppose that any $f\in \lip_\al(X_0)$ and $g\in \lip_\al(Y_0)$ extend to  functions $\ti{f} \in \lip_\al(X)$ and $\ti{g} \in \lip_\al(Y)$ respectively.  Then $\vp$ is a Lipschitz homeomorphism from $X_0$ onto $Y_0$.
\end{cor}

\begin{proof}
By Proposition \ref{lip3}, we have an order isomorphism $T_0:\lip_\al(X_0)\to \lip_\al(Y_0)$ whose associated homeomorphism is $\vp:X_0\to Y_0$.
The result now follows from Theorem \ref{thm7.1}.
\end{proof}

\begin{prop}\label{lip4}
Let $(x_n)$ and  $(x'_n)$ be sequences in $X$.  Set $y_n = \vp(x_n)$ and $y_n' = \vp(x_n')$ for all $n$.  If $(x_n)$ is separated and $(x_n) \cup (x'_n)$, $(y_n)\cup (y'_n)$ are bounded, then there exists $C<\infty$ such that ${d(y_n,y_n')}\leq Cd(x_n,x'_n)$.
\end{prop}

\begin{proof}
If the proposition fails, we find bounded sequences $(x_n)$ and $(x'_n)$ in $X$ and bounded sequences $(y_n) =(\vp(x_n))$, $(y'_n) = (\vp(x'_n))$, such that  $(x_n)$ is separated, $x_n\neq x'_n$ for all $n$, and $d(y_n,y'_n)/d(x_n,x'_n)\to \infty$.
Since $\sup_nd(y_n,y'_n) < \infty$, we must have $d(x_n,x'_n) \to 0$.
As $(x_n)$ has no convergent subsequence, neither does $(x'_n)$.  Thus the same holds for  $(y_n)$ and $(y'_n)$.  By using subsequences if necessary, we may assume that  $(y_n)$  and $(y'_n)$ are separated.
Set $r_n = d(x_n,x'_n)$ and $R_n = d(y_n,y'_n)$.
We may assume that 
\[ \inf_{m\neq n}d(x_m,x_n), \inf_{m\neq n}d(x'_m,x'_n) > r_i+r_j\] 
for all $i,j$, $R^\al_n \geq nr^\al_n$ and $r_n^\al < 1/2$ for all $n$.
Since $(y_n)$ is bounded, $(T1(y_n))$ is bounded.  There exists $M < \infty$ such that $\Phi(y_n,1) \leq M$ for all $n$.

\bigskip

\noindent\underline{Claim}.  For all $n$, there exists $0 \leq t_n \leq 1$ such that
\[ \Phi(y_n,t_n+r^\al_n) - \Phi(y_n,t_n) \leq 2Mr^\al_n.\]

\medskip

\noindent Otherwise, there exists $n$ such that  for all $0 \leq t\leq 1$, 
\[ \Phi(y_n,t+r^\al_n) - \Phi(y_n,t) > 2Mr^\al_n.\]
Choose $k$ such that $1/2 < kr_n^\al \leq 1$.  Then
\[ \Phi(y_n,kr_n^\al) = \sum^k_{j=1}[\Phi(y_n,jr^\al_n) - \Phi(y_n,(j-1)r^\al_n)] > 2kMr^\al_n > M.\]
It follows that $\Phi(y_n,1) > M$, contrary to the choice of $M$.

\bigskip

Since $(x_n)$ is separated and $(t_n)$ is bounded, the function $f:(x_n) \to \R$ defined by 
$f(x_n) = t_n$ is Lipschitz with respect to $d^\al$.  
By Proposition \ref{lip1}, there exists $g\in \lip_\al(X)$ such that $g(x_n) = f(x_n) = t_n$.
In particular,
\[ \frac{t_n- g(x'_n)}{r^\al_n}  = \frac{g(x_n) - g(x'_n)}{r^\al_n} \to 0.\]
Thus Proposition \ref{lip1.1} applies to $(x'_n)$, with $r_n$ as chosen and $a_n = t_n - g(x'_n)$.
Furthermore, $x_n \notin B(x'_n,r_n)$ for all $n$; while for $m \neq n$,
\[ d(x_n,x_m') \geq d(x_n',x_m') -d(x_n,x_n')  > r_m.\]
Thus there exists $
h \in \lip_\al(X)$ such that $h(x_n) = 0$ and $h(x'_n) = t_n -g(x'_n)$ for all $n$.
Let $p= g+h$.  Then $p\in \lip_\al(X)$, and $p(x_n) = p(x'_n) = t_n$ for all $n$. 
From the Claim,
\[ 0 \leq \frac{\Phi(y_n,t_n+r^\al_n) - \Phi(y_n,t_n)}{R^\al_n} \leq \frac{2Mr^\al_n}{R_n^\al} \to 0.\]
If $R_n \to 0$, then we may assume that
$\inf_{m\neq n}d(y_m,y_n)> R_i+R_j$ {for all $i,j$.}
As above, Proposition \ref{lip1.1} applies to the sequence $(y_n)$, with the parameters $R_n$ and $b_n =\Phi(y_n,t_n+r^\al_n) - \Phi(y_n,t_n)$. Furthermore, $y_n' \notin \cup_mB(y_m,R_m)$ for all $n$.
Thus there exists $q \in \lip_\al(Y)$ such that 
$q(y_n) = b_n$ and $q(y'_n) = 0$
for all $n$.
On the other hand, if $R_n \not\to 0$, then we may assume that $(y_n)\cup (y_n')$ is separated.
Since the sequence $(b_n)$ converges to $0$, it is bounded.
Thus the function $\gamma:(y_n)\cup(y'_n)\to \R$ defined by $\gamma(y_n) = b_n$ and $\gamma(y'_n) = 0$ is Lipschitz with respect to $d^\al$.  By Proposition \ref{lip1}, we also obtain a function $q\in \lip_\al(Y)$ such that $q(y_n) = b_n$ and $q(y'_n) = 0$.
In either case, $Tp +q \in \lip_\al(Y)$ and thus $T^{-1}(Tp+q) \in \lip_\al(X)$.
We have
\begin{align*} 
(Tp +q)(y_n) &= \Phi(y_n,t_n) + q(y_n) = \Phi(y_n,t_n+r^\al_n),\\
(Tp+q)(y'_n) &= \Phi(y_n',t_n) + q(y_n') = \Phi(y_n',t_n).
\end{align*}
Therefore,
\[ T^{-1}(Tp +q)(x_n) = t_n + r^\al_n \text{ and } T^{-1}(Tp +q)(x'_n) = t_n.\]
This violates the fact that $T^{-1}(Tp+q) \in \lip_\al(X)$.
\end{proof}

\begin{prop}\label{lip5}
Let $(x_n)$ and  $(x'_n)$ be sequences in $X$.  Set $y_n = \vp(x_n)$ and $y_n' = \vp(x_n')$ for all $n$.  If  $(x_n) \cup (x'_n)$, $(y_n)\cup (y'_n)$ are bounded, 
then there exists $C<\infty$ such that ${d(y_n,y_n')}\leq Cd(x_n,x'_n)$.
\end{prop}

\begin{proof}
Suppose that the proposition fails. There are sequences $(x_n)$ and $(x'_n)$ such that $(x_n)\cup(x'_n)$ and $(y_n)\cup(y'_n) = (\vp(x_n))\cup(\vp(x'_n))$ are bounded, and that $d(y_n,y'_n)/d(x_n,x'_n) \to \infty$.
In particular, $d(x_n,x'_n) \to 0$.  By Proposition \ref{lip4}, $(x_n)$ cannot have a separated subsequence.
Thus, by considering a subsequence if necessary, we may assume that $(x_n)$ converges to some $x_0$.  It follows that $(x'_n)$ also converges to $x_0$.  Then $(y_n)$ and $(y'_n)$ both converge to $y_0 = \vp(x_0)$.
Applying Proposition \ref{lip2.1} to both $(x_n) \cup(x'_n)$ and $(y_n)\cup(y'_n)$, we may assume that there is a subsequence $(n_k)$ of $\N$ so that, setting $X_0 = (x_{n_k})\cup(x'_{n_k})$, $Y_0 = (y_{n_k})\cup(y'_{n_k})$, 
every function in $\lip_\al(X_0)$, respectively $\lip_\al(Y_0)$, extends to a function in $\lip_\al(X)$, respectively $\lip_\al(Y)$.
Clearly, every function in $\lip_\al(X_0\cup\{x_0\})$ also extends to a function in $\lip_\al(X)$, and every function in $\lip_\al(Y_0\cup\{y_0\})$ extends to a function in $\lip_\al(Y)$.
Since $X_0\cup\{x_0\}$ is compact,
by Corollary \ref{cor7.6}, $\vp$ is a Lipschitz homeomorphism from $X_0\cup\{x_0\}$ onto $Y_0\cup\{y_0\}$.
This yields a contradiction in view of the choices of $(x_n), (x'_n)$ and $(y_n), (y'_n)$.
\end{proof}

The next result follows immediate from Proposition \ref{lip5}.

\begin{thm}\label{lip6}
Suppose that $X$ and $Y$ are complete metric spaces with finite diameter.  If $T:\lip_\al(X)\to \lip_\al(Y)$ is an order isomorphism, then the associated homeomorphism $\vp: X\to Y$ is a Lipschitz homeomorphism.
Conversely, if $X$ and $Y$ are Lipschitz homeomorphic, then $\lip_\al(X)$ is order isomorphic to $\lip_\al(Y)$.
\end{thm}

Next, we consider complete metric spaces which may be unbounded.  For the sake of clarity, denote the metrics on $X$ and $Y$  by $d_X$ and $d_Y$ respectively.  Recall the following from \S \ref{sec5}.
Fix a point $e\in X$ and let $e' = \vp(e)\in Y$. Set $\xi(x) = 1 \vee d_X(x,e)^\al$, $\zeta(y) = 1 \vee d_Y(y,e')^\al$, 
\[ \rho_X(x,x') = \frac{d_X(x,x')^\al}{\xi(x) \vee \xi(x')} \text{\ \ and\ \ } \rho_Y(y,y')=\frac{d_Y(y,y')^\al}{\zeta(y) \vee \zeta(y')}\]
for all $x,x'\in X$ and $y,y'\in Y$.

\begin{prop}\label{lip7}
Let $T:\lip_\al(X)\to \lip_\al(Y)$ be an order isomorphism with associated homeomorphism $\vp:X\to Y$.
There exists $C<\infty$ such that $\rho_Y(\vp(x),\vp(x'))\leq C\rho_X(x,x')$ for all $x, x'\in X$.
\end{prop}

\begin{proof}
Suppose that the proposition fails.  There are sequences $(x_n)$ and $(x'_n)$ in $X$ such that $x_n\neq x'_n$ and that $\rho_Y(y_n,y'_n)/\rho_X(x_n,x'_n)\to \infty$, where $y_n = \vp(x_n)$ and $y'_n = \vp(x'_n)$. 
By taking further subsequences and taking note of the symmetry between $(x_n)$ and $(x'_n)$, we may assume that we have one of the following situations.
\begin{enumerate}
\item $(x_n) \cup(x'_n)$ is bounded.
\item $(x_n)$ is unbounded, $d_X(x_{n+1},e) \geq 2 d_X(x_n,e)$ for all $n$, $(x'_n)$ is bounded.
\item $(x_n)$ is unbounded, $d_X(x_{n+1},e) \geq 2 d_X(x_n,e)$ for all $n$, $d_X(x_n,x'_n) \to 0$.
\item $(x_n)$ and $(x'_n)$ are both  unbounded, $d_X(x_n,x'_n) \not\to 0$.
\end{enumerate}
Set $X_0 = (x_n)\cup(x'_n)$ and $Y_0 = (y_n)\cup(y'_n)$.
Applying Proposition \ref{lip2.1} and the remark thereafter, we may assume that every function in $\lip_\al(\ol{X_0})$, respectively $\lip_\al(\ol{Y_0})$, extends to a function in $\lip_\al(X)$, respectively, $\lip_\al(Y)$.
By Proposition \ref{lip3}, there is an order isomorphism $T_0: \lip_\al(\ol{X_0})\to \lip_\al(\ol{Y_0})$ whose associated homeomorphism is $\vp: \ol{X_0}\to \ol{Y_0}$.
In cases (a), (b) and (c), one can readily verify that $\ol{X_0}$ is almost expansive at $\infty$. (Refer to the definition following Theorem \ref{thm5.6}.)
By Proposition \ref{prop5.12}, and Proposition \ref{prop5.1}(a), (b) and (d), one may endow $\ol{X_0}$ with a complete bounded  metric $d'$ such that $\rho_X(x,x') \leq d'(x,x') \leq 3\rho_X(x,x')$ for all $x,x'\in \ol{X_0}$, and that $\lip_\al(\ol{X_0})$ is order isomorphic to $\lip(\ol{X_0},d')$, with the formal identity map as the associated homeomorphism.
As $d'$ is a bounded metric, $\lip_b(\ol{X_0},d') = \lip(\ol{X_0},d')\sim \lip_\al(\ol{X_0}) \sim\lip_\al(\ol{Y_0})$.
Since $\rho_X^{1/\al}$ is within a constant multiple of a metric, we may also assume that $d'$ is a H\"{o}lder metric of order $\al$.
It follows from Theorem \ref{thm66} that $\ol{Y_0}$ is almost expansive at $\infty$.
By Proposition  \ref{prop5.12}, and Proposition \ref{prop5.1}(b) again, one may endow $\ol{Y_0}$ with a complete  metric $d''$ such that $\rho_Y(y,y') \leq d''(y,y') \leq 3\rho_Y(y,y')$ for all $y,y'\in \ol{Y_0}$, and that $\lip_\al(\ol{Y_0})$ is order isomorphic to $\lip(\ol{Y_0},d'')$, with the formal identity map as the associated homeomorphism.  Again, we may assume that $d''$ is a H\"{o}lder metric of order $\al$.
Working through the chain of order isomorphisms
\[\lip(\ol{X_0},d')  \sim\lip_\al(\ol{X_0}) \sim \lip_\al(\ol{Y_0}) \sim \lip(\ol{Y_0},d''),\]
we see that $\lip(\ol{X_0},d')  \sim\lip(\ol{Y_0},d'')$ with associated homeomorphism $\vp:\ol{X_0}\to \ol{Y_0}$.
Since both $d'$ and $d''$ are bounded metrics, $\vp: (\ol{X_0},d')\to (\ol{Y_0},d'')$ is a Lipschitz homeomorphism by Theorem \ref{lip6}.
But this shows that there is a constant $C$ such that $\rho_Y(y_n,y'_n) \leq C\rho_X(x_n,x'_n)$, contrary to the choices of $(x_n)$ and $(x'_n)$.

In case (d), we may assume that $(x_n)\cup (x'_n)$ is separated.  Then neither $(y_n)$ nor $(y_n')$ has a convergent subsequence.  Hence we may assume that each one is separated.
If $d_Y(y_n,y'_n)\to 0$, then by taking further subsequences, and possibly interchanging $(y_n)$ and $(y_n')$,
we end up in one of the situations (a) or (c) for the sequences $(y_n)$ and $(y_n')$. 
By the proof above, the map $\vp^{-1}: (\ol{Y_0},d'')\to (\ol{X_0},d')$ is a Lipschitz homeomorphism.  Again this implies that there is a constant $C$ such that $\rho_Y(y_n,y'_n) \leq C\rho_X(x_n,x'_n)$, contrary to the choices of $(x_n)$ and $(x'_n)$.

Finally, suppose that $(x_n)\cup (x'_n)$ is separated and that $d_Y(y_n,y'_n)\not\to 0$.  Then we may also assume that $(y_n)\cup(y_n')$ is  separated.
In this case $\ol{X_0} = X_0$ and $\ol{Y_0} = Y_0$ are separated.
Hence $\lip_\al(\ol{X_0}) = \Lip(X_0,d_X^\al)$ and $\lip_\al(\ol{Y_0}) = \Lip(Y_0,d_Y^\al)$.
It follows that $\Lip(X_0,d_X^\al)\sim \Lip(Y_0,d_Y^\al)$, with associated homeomorphism $\vp:X_0\to Y_0$.  By Theorem \ref{thm5.6}, there is a constant $1 \leq C< \infty$ such that
$\rho_Y(\vp(x),\vp(x')) \leq C\rho_X(x,x')$ for all $x,x'\in X_0$. Again, this contradicts the choices of $(x_n)$ and $(x_n')$.
\end{proof}

\begin{prop}\label{lip8}
Let 
$T:\lip_\al(X)\to \lip_\al(Y)$ be an order isomorphism. Then the associated homeomorphism $\vp:X\to Y$ is uniformly continuous.
\end{prop}

\begin{proof}
If the proposition fails, there are sequences $(x_n)$ and $(x'_n)$ in $X$ such that $0 < d_X(x_n,x'_n) \to 0$ and $d_Y(y_n,y'_n)\not\to 0$, where $y_n = \vp(x_n)$ and $y_n' = \vp(x'_n)$.
In particular, neither $(x_n)$ nor $(x'_n)$ can have  a convergent subsequence.  Thus the same holds for $(y_n)$ and $(y_n')$.  By taking subsequences, we may assume that $(x_n)$ and $(y_n)\cup (y'_n)$ are separated.
Set $X_0 = (x_n)\cup(x'_n)$ and $Y_0 = (y_n)\cup(y'_n)$. 
Note that $X_0$ and $Y_0$ are complete. Applying Proposition \ref{lip2.1}, we may further assume that every function in $\lip_\al(X_0)$, respectively $\lip_\al(Y_0)$, extends to a function in $\lip_\al(X)$, respectively, $\lip_\al(Y)$.
It follows from Proposition \ref{lip3} that we have an order isomorphism $T_0:\lip_\al(X_0)\to \lip_\al(Y_0)$.
Since $Y_0$ is separated, $\lip_\al(Y_0) = \Lip(Y_0,d^\al_Y)$. 
By Theorem \ref{thm5.3} and Proposition \ref{prop5.1}(d), $\Lip(Y_0,d^\al)$ is (linearly) order isomorphic to a space $\Lip(Y_0,d')$, where $d'$ is a complete bounded metric on $Y_0$.
Therefore, 
\[\lip_\al(X_0) \sim \lip_\al(Y_0)  = \Lip(Y_,d^\al_Y) \sim \Lip(Y_0,d') = \Lip_b(Y_0,d').\]
By Proposition \ref{prop57.1}, $X_0$ is proximally compact.
(Refer to the definition following Proposition \ref{prop37}.)
However, $(x_n)$ is a sequence in $X_0$ with no convergent subsequence, and $(x_n') \subseteq X_0$ with $0 < d_X(x_n,x'_n)\to 0$.  This contradicts the proximal compactness of $X_0$.   
\end{proof}

\begin{lem}\label{lip8.1}
Let $X$ and $Y$ be complete metric spaces.  Suppose that $\vp:X\to Y$ is a  uniform homeomorphism such that there exists $1 \leq C < \infty$ satisfying 
\begin{equation}\label{lipeq1} 
\frac{1}{C}\rho_X(x,x') \leq \rho_Y(\vp(x),\vp(x')) \leq C\rho_X(x,x')
\end{equation}
for all $x,x'\in X$.  Then for all $f\in \lip_\al(X)$, $g:Y \to\R$ defined by
\[ g(y) = \frac{f(\vp^{-1}(y))}{\xi(\vp^{-1}(y))}\,\zeta(y)\]
belongs to $\lip_\al(Y)$.
\end{lem}

\begin{proof}
There exists a bounded  function $\eta:[0,\infty)\to \R$, with $\eta(0) = 0$ and continuous at $0$, such that
\[|f(x) - f(x')| \leq \eta(d_X(x,x'))d_X(x,x')^\al \text{\ for all $x,x' \in X$}.\]
In particular, there exists $M < \infty$ such that $|f(x)| \leq M\xi(x)$ for all $x\in X$.
For any $y, y' \in Y$, let $x = \vp^{-1}(y)$ and $x' = \vp^{-1}(y')$. We may assume that $\xi(x') \leq \xi(x)$.  Then
\begin{align}\label{lipeq1.1}
|g(y)- g(y')| &\leq |f(x) - f(x')|\frac{\zeta(y)}{\xi(x)} + |f(x')|\zeta(y)\bigl|\frac{1}{\xi(x)} - \frac{1}{\xi(x')}\bigr|+\\ \notag
&\quad + \frac{|f(x')|}{\xi(x')}|\zeta(y) - \zeta(y')|.
\end{align}
Denote the three terms on the right of inequality (\ref{lipeq1.1}) by I, II and III respectively.
For term I, we find
\begin{align}\label{lipeq2}
\text{I} &\leq \eta(d_X(x,x'))d_X(x,x')^\al\frac{\zeta(y)}{\xi(x)}
= \eta(d_X(x,x'))\rho_X(x,x')\zeta(y)\\ \notag
&\leq \eta(d_X(x,x'))C\rho_Y(y,y')\zeta(y) \leq \eta(d_X(x,x'))Cd_Y(y,y')^\al.
\end{align}
For terms II and III, use the fact that the function $h:[0,\infty)\to\R$ given by $h(t) = t^\al \vee 1$ belongs to $\lip_\al[0,\infty)$.
Then
\begin{align}\label{lipeq3}
\text{II} &= \frac{|f(x')|}{\xi(x')}\frac{\zeta(y)}{\xi(x)}|h(d_X(x,e))-h(d_X(x',e))|\\ \notag
&\leq MC\frac{|h(d_X(x,e))-h(d_X(x',e))|}{d_X(x,x')^\al}d_Y(y,y')^\al\\ \notag
&\leq MC\frac{|h(d_X(x,e))-h(d_X(x',e))|}{|d_X(x,e)-d_X(x',e)|^\al}d_Y(y,y')^\al\\
\intertext{and} \label{lipeq4}
\text{III} &= \frac{|f(x')|}{\xi(x')}|h(d_Y(y,e'))-h(d_Y(y',e'))|\\ \notag
&\leq M\frac{|h(d_Y(y,e'))-h(d_Y(y',e'))|}{|d_Y(y,e')-d_Y(y',e')|^\al}\,d_Y(y,y')^\al.
\end{align}
Since $\vp^{-1}$ is uniformly continuous, $d_X(x,x')\to 0$ as $d_Y(y,y')\to 0$.  Since $h\in \lip_\al[0,\infty)$,
it follows readily from (\ref{lipeq1.1}), (\ref{lipeq2}), (\ref{lipeq3}) and (\ref{lipeq4}) that $g$ belongs to $\lip_\al(Y)$.
\end{proof}

\begin{thm}\label{lip9}
Let $X$ and $Y$ be complete metric spaces. If $T:\lip_\al(X)\to \lip_\al(Y)$ is an order isomorphism, then the associated homeomorphism $\vp:X\to Y$ is a uniform homeomorphism and there exists $1 \leq C < \infty$ satisfying (\ref{lipeq1}).
Conversely, if there is a uniform homeomorphism $\vp:X\to Y$ and a constant $1 \leq C <\infty$ satisfying (\ref{lipeq1}), then $\lip_\al(X)$ is order isomorphic to $\lip_\al(Y)$.
\end{thm}

\begin{proof}
The first statement of the theorem is a consequence of Propositions \ref{lip7} and \ref{lip8}.
Conversely, suppose that $\vp:X\to Y$ is a uniform homeomorphism such that (\ref{lipeq1}) holds.
By Lemma \ref{lip8.1}, the map $T$ defined by 
$Tf(y) = f(\vp^{-1}(y))\zeta(y)/\xi(\vp^{-1}(y))$ maps $\lip_\al(X)$ into $\lip_\al(Y)$.
By symmetry, the map $S$ defined by $Sg(x) = g(\vp(x))\xi(x)/\zeta(\vp(x))$ maps $\lip_\al(Y)$ into $\lip_\al(X)$.  Clearly, $T$ and $S$ are mutual inverses.  Since both maps obviously preserve pointwise order, $\lip_\al(X)$ is order isomorphic to $\lip_\al(Y)$.
\end{proof}

\section{Comparing spaces of the same type}

In this final section, we identify necessary and sufficient conditions for pairs of spaces of the same type to be order isomorphic.  We adopt the same convention as in subsection \ref{subsec6.2}.  Namely, $X$ and $Y$ will always denote metric spaces, with additional properties as given in Examples B or C, depending on the space being considered.  The metrics on both spaces will be denoted by $d$, even though they may differ.
We begin by making a list of the cases that are already known and/or follow readily from results obtained earlier in the paper.

\begin{thm}\label{thm8.1}
Let $X$ and $Y$ be metric spaces and assume that the relevant spaces satisfy the conditions given in Examples B and C.
\begin{enumerate}
\item \cite{C} $C_b(X) \sim C_b(Y) \iff C(X) \sim C(Y) \iff$ $X$ and $Y$ are homeomorphic.
\item (Theorem \ref{thm5.6}) $\Lip(X) \sim \Lip(Y)$ if and only if there is a homeomorphism $\vp: X\to Y$ satisfying condition (\ref{eq5.2}).
\item $\Lip_b(X) \sim \Lip_b(Y)$ if and only if there is a Lipschitz homeomorphism $\vp: (X,d\wedge 1) \to (Y, d\wedge 1)$.
\item (Theorem \ref{lip9}) $\lip_\al(X)\sim\lip_\al(Y)$ if and only if there is a uniform homeomorphism $\vp:X\to Y$ satisfying condition (\ref{lipeq1}).
\item $\lip_{\al,b}(X)\sim\lip_{\al,b}(Y)$ if and only if there is a Lipschitz homeomorphism $\vp: (X,d\wedge 1) \to (Y, d\wedge 1)$.
\item \cite{CC3} (see also Proposition \ref{prop58.2}) $U(X) \sim U(Y)$ if and only if $X$ and $Y$ are uniformly homeomorphic.
\item (Proposition \ref{prop57.1}) $U_b(X) \sim U_b(Y)$ if and only if $X$ and $Y$ are uniformly homeomorphic.
\end{enumerate}
\end{thm}

\begin{proof}
(a) If $C_b(X)\sim C_b(Y)$, then $C(X) \sim C(Y)$ by Proposition \ref{prop34}.  
If $C(X) \sim C(Y)$, then $X$ and $Y$ are homeomorphic by Theorem \ref{thm29}.
Obviously $C_b(X)\sim C_b(Y)$ if $X$ and $Y$ are homeomorphic.

\noindent(c) Since $\Lip_b(X) = \Lip(X, d\wedge 1)$ and $\Lip_b(Y) = \Lip(Y,d\wedge 1)$, the result follows from \cite[Theorem 1]{CC1}.

\noindent(e) Since $\lip_{\al,b}(X) = \lip_\al(X,d\wedge 1)$ and $\lip_{\al,b}(Y) = \lip_\al(Y,d\wedge 1)$, 
$\lip_{\al,b}(X)\sim\lip_{\al,b}(Y)$ if and only if there is a uniform homeomorphism $\vp:(X,d\wedge 1)\to (Y,d\wedge 1)$ satisfying  condition (\ref{lipeq1}) with respect to the metrics $d\wedge 1$ on $X$ and $Y$ respectively.  But for bounded metrics, condition (\ref{lipeq1}) is equivalent to the fact the $\vp:(X,d\wedge 1)\to (Y,d\wedge 1)$ is a Lipschitz homeomorphism.  In this case, $\vp$ is automatically a uniform homeomorphism.
\end{proof}

A map $\vp:X\to Y$ is {\em locally Lipschitz} if for any $x\in X$, there is an open neighborhood $U$ of $x$ such that $\vp$ is Lipschitz on the set $U$.
If $\vp:X\to Y$ is a bijection so that both  $\vp$ and $\vp^{-1}$ are locally Lipschitz, then $\vp$ is a {\em local Lipschitz homeomorphism}.

\begin{thm}\label{thm8.2}
Let $X$ and $Y$ be metric spaces.  Then the following are equivalent.
\begin{enumerate}
\item $\Lip^\loc(X)\sim \Lip^\loc(Y)$.
\item $\lip^\loc_\al(X)\sim \lip^\loc_\al(Y)$.
\item $X$ and $Y$ are locally Lipschitz homeomorphic.
\end{enumerate}
\end{thm}

\begin{proof}
Suppose that $\vp: X\to Y$ is a homeomorphism that is not locally Lipschitz.
There exist $x_0 \in X$ and sequences $(x_n)$, $(x_n')$ converging to $x_0$ such that $x_n \neq x'_n$ for all $n$ and $d(\vp(x_n),\vp(x_n'))/d(x_n,x'_n) \to \infty$.
Set $y_n = \vp(x_n)$, $y_n' = \vp(x'_n)$, $y_0 = \vp(x_0)$, $X_0 = (x_n)\cup(x'_n)\cup\{x_0\}$ and $Y_0 = (y_n)\cup(y'_n) \cup\{y_0\}$. By Proposition \ref{lip2.1} and the subsequent remark, we may assume that every $f\in \lip_\al(X_0)$ extends to a function in $\lip_\al(X)$.

Suppose that $f\in \lip_\al(X_0)$, respectively, $\Lip(X_0)$.  Either by choice of $X_0$ or by \cite[Theorem 1.5.6]{W}, $f$ extends to a function $\ti{f} \in \lip_\al(X)$, respectively, $\Lip(X)$. In particular, $\ti{f} \in \lip^\loc_\al(X)$, respectively, $\Lip^\loc(X)$.  
Conversely, if $h \in \lip^\loc_\al(X)$, respectively,  $\Lip^\loc(X)$, then $h_{|X_0}\in \lip^\loc_\al(X_0)$, respectively, $\Lip^\loc(X_0)$.  Since $X_0$ is compact, $h_{|X_0} \in \lip_\al(X_0)$, respectively, $\Lip(X_0)$.
This proves that
\[ \lip_\al(X_0) =\{f_{|X_0}: f\in \lip^\loc_\al(X)\} \text{ and } \Lip(X_0) = \{f_{|X_0}: f\in \Lip^\loc(X)\}.\]
Similar equalities hold with $X_0$ and $X$ replaced by $Y_0$ and $Y$ respectively.
It follows from Proposition \ref{lip3} that there is an order isomorphism $\lip_\al(X_0)\sim \lip_\al(Y_0)$, respectively, $\Lip(X_0) \sim \Lip(Y_0)$, whose associated homeomorphism is $\vp:X_0\to Y_0$. 
By \cite[Theorem 2]{CC1}, respectively, \cite[Theorem 1 ]{CC1}, $\vp:X_0\to Y_0$ is a Lipschitz homeomorphism, contrary to the choices of $(x_n)$ and $(x'_n)$.
This completes the proof of (a) $\implies$ (c) and (b)$\implies$ (c).
Clearly, (c) $\implies$ (a) and (b).
\end{proof}

\begin{cor}\label{cor8.3}
Let $X$ and $Y$ be metric spaces.  The following are equivalent.
\begin{enumerate}
\item $\Lip^\loc_b(X)\sim \Lip^\loc_b(Y)$.
\item $\lip^\loc_{\al,b}(X)\sim \lip^\loc_{\al,b}(Y)$.
\item $X$ and $Y$ are locally Lipschitz homeomorphic.
\end{enumerate}
\end{cor}

\begin{proof}
If $\Lip^\loc_b(X)\sim \Lip^\loc_b(Y)$, then $\Lip^\loc(X)\sim \Lip^\loc(Y)$ by Corollary \ref{cor35}.
Thus $X$ and $Y$ are locally Lipschitz homeomorphic by Theorem \ref{thm8.2}.
This proves that (a) $\implies$ (c).  Similarly, (b) $\implies$ (c).  Clearly, (c) $\implies$ (a) and (b).
\end{proof}

A map $\vp:X\to Y$ is {\em locally uniformly continuous} if for all $x\in X$, there exists an open neighborhood $U$ of $x$ such that $\vp$ is uniformly continuous on $U$. If $\vp:X\to Y$ is a bijection so that both $\vp$ and $\vp^{-1}$ are locally uniformly continuous, then $\vp$ is a {\em local uniform homeomorphism}.

\begin{thm}\label{thm8.4}
Let $X$ and $Y$ be  metric spaces.
Then $U^\loc(X)\sim U^\loc(Y)$ if and only if $X$ and $Y$ are locally uniformly homeomorphic.
\end{thm}

\begin{proof}
Let $T: U^\loc(X)\to U^\loc(Y)$ be an order isomorphism such that $T0=0$.
Represent $T$ as $Tf(y) = \Phi(y,f(\vp^{-1}(y)))$ for all $f\in U^\loc(X)$ and all $y\in Y$, where $\vp:X\to Y$ is a homeomorphism and $\Phi(y,\cdot):\R\to \R$ is an increasing homeomorphism for all $y\in Y$.
Suppose that $\vp$ is not locally uniformly continuous.
There exists $x_0 \in X$ such that $\vp$ is not uniformly continuous on any open neighborhood of $x_0$.
Set $y_0 = \vp(x_0)$ and $r_0 = 1$.  Assume that $r_{n-1}> 0$ has been chosen for some $n\in \N$.
Since $\Phi(y,1/n)$ is a continuous function of $y$ and $\Phi(y_0,1/n) > 0$, there exists $0 < R_n < r_{n-1}$ such that $\Phi(y,1/n) < 2 \Phi(y_0,1/n)$ for all $y\in B(y_0,R_n)$.
As $\vp$ is not uniformly continuous on $\vp^{-1}(B(y_0,R_n/2))$, there are sequences $(u^n_k)_k$, $(v^n_k)_k$ in $\vp^{-1}(B(y_0,R_n/2))$ such that $d(u^n_k,v^n_k)\to 0$ and $\inf_kd(\vp(u^n_k),\vp(v^n_k)) > 0$.
Obviously, neither $(\vp(u^n_k))_k$ nor $(\vp(v^n_k))_k$ can have a convergent subsequence.  By using further subsequences, we may assume that $(\vp(u^n_k))_k\cup (\vp(v^n_k))_k$ is separated and that 
 there exists $r_n > 0$ such that $d(\vp(u^n_k),y_0), d(\vp(v^n_k),y_0) > 2r_n$ for all $k$.
Since $(\Phi(\vp(u^n_k),1/n))_k$ is bounded above by $2\Phi(y_0, 1/n)$, 
there exists a uniformly continuous function $g_n:Y\to \R$ such that $g_n(y) = 0$ if $y \notin \Ann(y_0,r_n,R_n)$, $g_n(\vp(u^n_k)) = \Phi(\vp(u^n_k),1/n)$, $g_n(\vp(v^n_k)) = 0$ for all $k$ and $\|g_n\|_\infty \leq 2\Phi(y_0,1/n)$.
Note that $\Phi(y_0,1/n) \to 0$ and the functions $(g_n)$ are pairwise disjoint.
Hence the pointwise sum $g = \sum g_n$ is well-defined and uniformly continuous on $Y$.
In particular, $g\in U^\loc(Y)$.
Therefore, $f = T^{-1}g \in U^\loc(X)$.
However, for all $n$ and $k$, $f(u^n_k) = 1/n$ and $f(v^n_k) = 0$.
It follows that $f$ is not locally uniformly continuous on any neighborhood of $x_0$.
This completes the proof of the ``only if" part of the theorem.
The converse is clear.
\end{proof}

Arguing as in Corollary \ref{cor8.3}, we have the following corollary.

\begin{cor}\label{cor8.5}
Let $X$ and $Y$ be  metric spaces.
Then $U^\loc_b(X)\sim U^\loc_b(Y)$ if and only if $X$ and $Y$ are locally uniformly homeomorphic.
\end{cor}

The tables below summarize the comparison results obtained in \S 6 - 8.  The numbers refer to the theorems where the relevant results are found.  Note that the blank boxes are left unfilled by nature of symmetry.  The open cases are indicated with ``?"; see the introduction.

\vspace{1in}

\begin{tabular}{||p{1.5cm}||p{1.5cm}|p{1.5cm}|p{1.5cm}|p{1.5cm}|p{1.5cm}||}
 \hline\hline &\ \ \ \ \   $C^p(Y)$ &\ \ \ \ \   $C^p_b(Y)$ & \ \ \ \ \  $C^p(\ol{Y})$ &\ \ \ \ \   $C^p_b(\ol{Y})$ &\ \ \ \ \   $C^p_*(\ol{Y})$\\ \hline\hline
$C^p(X)$ &?&6.7(b)&6.41&6.7(b)&6.7(b)\\ \hline
$C^p_b(X)$ &&?&6.7(b)&6.42&6.43 \\ \hline
\vspace*{0in}$C^p(\ol{X})$&&&? & 6.7(c) & 6.43 \\ \hline
\vspace*{0in}$C^p_b(\ol{X})$ &&&& ? & 6.43\\ \hline
\vspace*{0in}$C^p_*(\ol{X})$ &&&&&? \\ \hline\hline
\end{tabular}


\begin{sideways}

\begin{tabular}{||p{1.5cm}||c|c|c|c|c|c|p{1.5cm}|c|c|c|c|c|c||}
\hline\hline & $C(Y)$ & $C_b(Y)$ & $U(Y)$ & $U_b(Y)$ & $U^\loc(Y)$ & $U^\loc_b(Y)$  & $\Lip(Y)$ $\Lip_b(Y)$ &  $\Lip^\loc(Y)$ & $\Lip^\loc_b(Y)$   &  $\lip_\al(Y)$ & $\lip_{\al,b}(Y)$ & $\lip^\loc_\al(Y)$ & $\lip^\loc_{\al,b}(Y)$               \\ \hline \hline
$C(X)$ & 8.1(a) & 6.7(a)& 6.15(a)& 6.7(a)& 6.15(b)& 6.7(a) & 6.13(b)& 6.13(a)& 6.13(a) & 6.13(b)& 6.13(b)& 6.13(a)& 6.13(b) \\ \hline
$C_b(X)$ &&8.1(a)& 6.15(c)& 6.15(a)& 6.7(a)& 6.15(b)& 6.13(c)& 6.13(b)& 6.13(a)& 6.24(a)&6.13(c)&6.13(b)&6.13(a) \\ \hline
$U(X)$ &&&8.1(f)& 6.39& 6.15(a)& 6.15(c)& 6.21& 6.15(e)&6.15(d)&6.30 & 6.30 & 6.15(e)&6.15(d) \\ \hline
$U_b(X)$ &&&&8.1(g)& 6.7(a)& 6.15(a)& 6.21& 6.13(b)& 6.15(e)& 6.24(a)& 6.24(a)& 6.13(b)& 6.15(e) \\ \hline
$U^\loc(X)$ &&&&&8.4&6.7(a)&6.13(b)& 6.13(a)& 6.13(b)&6.13(b)&6.13(b)&6.13(a)&6.13(b) \\ \hline
$U^\loc_b(X)$ &&&&&&8.5 & 6.13(c)& 6.13(b)&6.13(a)&6.24(a)&6.13(c)&6.13(b)&6.13(a) \\ \hline
$\Lip(X)$ $\Lip_b(X)$ &&&&&&&5.5  8.1(b)(c)&6.7(a) & 6.16& 6.21&6.21&6.17(b)&6.17(c) \\ \hline
$\Lip^\loc(X)$ &&&&&&&&8.2&6.7(a)&6.17(b)&6.17(a)&6.17(a)&6.17(b) \\ \hline
$\Lip^\loc_b(X)$ &&&&&&&&&8.3&6.24(a)&6.17(a)&6.17(b)&6.17(a) \\ \hline
$\lip^\al(X)$ &&&&&&&&&&7.14&6.34&6.15(c)&6.33 \\ \hline
$\lip_{\al,b}(X)$ &&&&&&&&&&&8.1(e)&6.7(a)&6.16 \\ \hline
$\lip^\loc_{\al}(X)$ &&&&&&&&&&&&8.2&6.7(a) \\ \hline
$\lip^\loc_{\al,b}(X)$ &&&&&&&&&&&&&8.3 \\ \hline
$C^p(X)$ &6.8&6.8&6.8&6.8&6.8&6.8&6.8&6.8&6.8&6.8&6.8&6.8&6.8 \\ \hline
$C^p_b(X)$  &6.8&6.8&6.8&6.8&6.8&6.8&6.8&6.8&6.8&6.8&6.8&6.8&6.8 \\ \hline
\vspace*{0in}$C^p(\ol{X})$  &6.8&6.8&6.8&6.8&6.8&6.8&6.8&6.8&6.8&6.8&6.8&6.8&6.8 \\ \hline
\vspace*{0in}$C^p_b(\ol{X})$  &6.8&6.8&6.8&6.8&6.8&6.8&6.8&6.8&6.8&6.8&6.8&6.8&6.8 \\ \hline
\vspace*{0in}$C^p_*(\ol{X})$  &6.8&6.8&6.8&6.8&6.8&6.8&6.8&6.8&6.8&6.8&6.8&6.8&6.8 \\ \hline\hline
\end{tabular}

\end{sideways}


\end{document}